\documentclass[reqno,11pt]{amsart}
\usepackage{amsmath, latexsym, amsfonts, amssymb, amsthm, amscd}
\usepackage{graphics,epsf,psfrag}
\usepackage{graphics,epsf,psfrag}
\usepackage[utf8]{inputenc}
\usepackage{stmaryrd}
\usepackage{mathabx}
\usepackage{hyperref}
\usepackage{xcolor}

\renewcommand{\bar}{\overline}

\newcommand{\lint}{\llbracket}
\newcommand{\rint}{\rrbracket}

\newcommand{\bt}{\mathbf{t}}
\newcommand{\bx}{\mathbf{x}}
\newcommand{\bn}{\mathbf{n}}
\numberwithin{equation}{section}

\newtheorem{theorema}{Theorem}

\newtheorem{theorem}{Theorem}[section]
\newtheorem{lemma}[theorem]{Lemma}
\newtheorem{claim}{Claim}
\newtheorem{proposition}[theorem]{Proposition}

\newtheorem{rem}[theorem]{Remark}

\newcommand{\Det}{\mathrm{Det}}

\newcommand{\dd}{\mathrm{d}}
\newcommand{\ind}{\mathbf{1}}
\newcommand{\restrict}[1]{\raise-.3ex\hbox{\big|}_{#1}}

\newcommand{\R}{\mathbb{R}}

\renewcommand{\tilde}{\widetilde}
\renewcommand{\hat}{\widehat}
\newcommand{\cc}{\complement}

\newcommand{\cA}{{\ensuremath{\mathcal A}} }
\newcommand{\cB}{{\ensuremath{\mathcal B}} }

\newcommand{\cP}{{\ensuremath{\mathcal P}} }

\newcommand{\cC}{{\ensuremath{\mathcal C}} }

\newcommand{\cL}{{\ensuremath{\mathcal L}} }

\newcommand{\cZ}{{\ensuremath{\mathcal Z}} }

\newcommand{\cK}{{\ensuremath{\mathcal K}} }

\newcommand{\bP}{{\ensuremath{\mathbf P}} }
\newcommand{\bQ}{{\ensuremath{\mathbf Q}} }
\newcommand{\bE}{{\ensuremath{\mathbf E}} }

\newcommand{\cM}{{\ensuremath{\mathcal M}} }
\newcommand{\cW}{{\ensuremath{\mathcal W} }}

\DeclareMathSymbol{\leqslant}{\mathalpha}{AMSa}{"36} 
\DeclareMathSymbol{\geqslant}{\mathalpha}{AMSa}{"3E} 
\DeclareMathSymbol{\eset}{\mathalpha}{AMSb}{"3F}     
\renewcommand{\leq}{\, \leqslant\,}                   

\newcommand{\Var}{\mathrm{Var}}        
\newcommand{\mintwo}[2]{\min_{\substack{#1 \\ #2}}} 
\newcommand{\suptwo}[2]{\sup_{\substack{#1 \\ #2}}} 
\newcommand{\sumtwo}[2]{\sum_{\substack{#1 \\ #2}}} 
\newcommand{\inttwolimits}[2]{\int\limits_{\substack{#1 \\ #2}}}     

\newcommand{\Cov}{\mathrm{Cov}}

\newcommand{\bbE}{{\ensuremath{\mathbb E}} }

\newcommand{\bbH}{{\ensuremath{\mathbb H}} }

\newcommand{\bbN}{{\ensuremath{\mathbb N}} }

\newcommand{\bbP}{{\ensuremath{\mathbb P}} }
\newcommand{\bbQ}{{\ensuremath{\mathbb Q}} }
\newcommand{\bbR}{{\ensuremath{\mathbb R}} }

\newcommand{\bbZ}{{\ensuremath{\mathbb Z}} }

\newcommand{\ga}{\alpha}
\newcommand{\gb}{\beta}
\newcommand{\gd}{\delta}

\newcommand{\gep}{\varepsilon}       
\newcommand{\gp}{\varphi}

\newcommand{\gG}{\Gamma}

\newcommand{\go}{\omega}

\newcommand{\gO}{\Omega}

\newcommand{\ups}{\upsilon}

\makeatletter
\def\captionfont@{\footnotesize}
\def\captionheadfont@{\scshape}

\long\def\@makecaption#1#2{%
  \vspace{2mm}
  \setbox\@tempboxa\vbox{\color@setgroup
    \advance\hsize-6pc\noindent
    \captionfont@\captionheadfont@#1\@xp\@ifnotempty\@xp
        {\@cdr#2\@nil}{.\captionfont@\upshape\enspace#2}%
    \unskip\kern-6pc\par
    \global\setbox\@ne\lastbox\color@endgroup}%
  \ifhbox\@ne 
    \setbox\@ne\hbox{\unhbox\@ne\unskip\unskip\unpenalty\unkern}%
  \fi
  \ifdim\wd\@tempboxa=\z@ 
    \setbox\@ne\hbox to\columnwidth{\hss\kern-6pc\box\@ne\hss}%
  \else 
    \setbox\@ne\vbox{\unvbox\@tempboxa\parskip\z@skip
        \noindent\unhbox\@ne\advance\hsize-6pc\par}%
\fi
  \ifnum\@tempcnta<64 
    \addvspace\abovecaptionskip
    \moveright 3pc\box\@ne
  \else 
    \moveright 3pc\box\@ne
    \nobreak
    \vskip\belowcaptionskip
  \fi
\relax
}
\makeatother
\def\writefig#1 #2 #3 {\rlap{\kern #1 truecm
\raise #2 truecm \hbox{#3}}}

\newcommand{\blue}{\color{blue}}

\usepackage[foot]{amsaddr}

\title[The scaling limit of the directed polymer with power-law tail disorder]{The scaling limit of the directed  polymer\\
with power-law tail disorder
}

\author{Quentin Berger}
\address{Q. Berger: LPSM, Sorbonne Universit\'e,  UMR 8001
Campus Pierre et Marie Curie, Bo\^ite courrier 158, 4 Place Jussieu, 75252 Paris Cedex 05;
}
\email{quentin.berger@sorbonne-universite.fr}

\author{Hubert Lacoin}
\address{
H. Lacoin:   IMPA, Institudo de Matem\'atica Pura e Aplicada, Estrada Dona Castorina 110
Rio de Janeiro, CEP-22460-320, Brasil;
}
\email{lacoin@impa.br}

\begin{document}

\begin{abstract}
In this paper, we study the so-called intermediate disorder regime for a directed polymer
in a random environment
with heavy-tail. 
Consider a simple symmetric random walk $(S_n)_{n\geq 0}$
on $\mathbb{Z}^d$, with $d\geq 1$,
and modify its law using 
Gibbs weights in the product form
$\prod_{n=1}^{N} (1+\beta\eta_{n,S_n})$, where $(\eta_{n,x})_{n\ge 0, x\in \bbZ^d}$ is a field of i.i.d.\ random variables whose distribution satisfies $\bbP(\eta>z) \sim  z^{-\alpha}$ as $z\to\infty$, for some $\alpha\in(0,2)$.
We prove that if $\alpha< \min(1 +  \frac{2}{d} ,2)$,
when sending $N$ to infinity and rescaling the disorder intensity by taking 
$\beta=\beta_N \sim  N^{-\gamma}$ with $\gamma =\frac{d}{2\alpha}(1+\frac{2}{d}-\alpha)$, the distribution of the trajectory under diffusive scaling converges in law towards a random limit, which is the 
continuum polymer with L\'evy  $\alpha$-stable noise constructed in the companion paper \cite{BL20_cont}.

\end{abstract}

\maketitle

\section{Introduction}

We consider in this article 
the directed polymer model,
which has been introduced by Huse and Henley~\cite{HH85}
as an effective model for a $(1+1)$-dimensional interface
in the Ising model with impurities.
It has then been generalized and used as a model
for a $(1+d)$-dimensional stretched polymer
placed in a heterogeneous solvent
and has received a lot of attention over the past decades:
we refer to~\cite{C17} for an overview.
The main achievement of the present paper is to identify,
in the case of a power-law tail environment, a continuum limit for the model when:
\begin{itemize}
 \item The size of the system $N$ tends to infinity;
\item Space and time are rescaled diffusively;
\item The intensity of the disorder $\beta=\beta_N$ is sent to zero at an appropriate rate. 
\end{itemize}

\subsection{The directed polymer model}

Let $(S_n)_{n\geq 0}$ be a simple symmetric random walk 
on $\bbZ^d$ with $d\geq 1$, starting from the origin. Its  law is denoted by $\bP$.
Let also $(\eta_{n,x})_{(n,x)\in \bbN \times \bbZ^d}$
be a $(1+d)$-dimensional field of i.i.d.\ random variables,
whose law is denoted by $\bbP$.
We will denote by $\eta$
a generic random variable with the same law as $\eta_{n,x}$.
We make the assumption that 
\begin{equation}
\label{standard}
\bbP(\eta> -1) =1 \qquad \text{ and either }   \bbE[\eta]  =0 \ \text{ or }  \ \bbE[\eta]=\infty.
\end{equation}
Now, for a fixed realization of the environment $(\eta_{n,x})_{(n,x)\in \bbN \times \bbZ^d}$ and
given a parameter $\gb \in (0,1)$
which tunes the disorder's strength, 
we define for $N\in \bbN$ the Gibbs measure~$\bP_{N,\gb}^{\eta}$ by
\begin{equation}\label{defdiscrete}
\frac{\dd \bP_{N,\gb}^{\eta}}{\dd \bP} (S) := \frac{1}{Z_{N,\gb}^{\eta} } \prod_{n=1}^N \big(  1+ \gb \eta_{n,S_n}\big) \,,
\end{equation}
where $Z_{N,\gb}^{\eta}$
is the partition function that normalizes $\bP_{N,\gb}^{\eta}$
to a probability measure, \textit{i.e.}
\begin{equation}
\label{def:Zn}
Z^{\eta}_{N,\beta}:= \bE\Big[\prod_{n=1}^N \big( 1+\beta \eta_{n,S_n} \big) \Big] \, .
\end{equation}
The graph of $(S_n)_{n=0}^N$ models the spatial configuration of the random polymer, and the field $(\eta_{n,x})_{(n,x)\in \bbN \times \bbZ^d}$ accounts for the heterogeneous environment.
The probability measure $\bP_{N,\gb}^{\eta}$ 
then favors trajectories of the random walk that visit
space-time points in the environment
with a large value of~$\eta$.
Let us stress that  the assumptions in \eqref{standard} are merely
practical. The first one ensures that 
$1+\beta \eta_{n,S_n}$  is always positive (which is required for our density \eqref{defdiscrete} to be positive),
 the second assumption is present for the sake of normalization so that
$\bbE [ Z^{\eta}_{N,\beta}]=1$ when the expectation is finite.

\smallskip

 In the present paper, we focus on the case of a disordered field whose tail distribution has a power-law decay. More precisely, we are going to assume that there exists $\alpha\in(0,2)$  and 
 $\varphi$  a slowly varying function (that is such that $\lim_{x\to \infty}\varphi(ax)/\varphi(x)= 1$ for any $a\in(0,\infty)$, see~\cite{BGT89}) such that for every $z\ge 0$ we have\footnote{The choice to consider 
$1+\eta$ rather than $\eta$  in \eqref{def:eta} is for convenience, because it is a non-negative quantity, but this detail is of no importance for this introduction.}
\begin{equation}
\label{def:eta}
\bbP(1+\eta > z)  = \varphi(z)z^{-\alpha},
\end{equation}
We investigate here the existence of a non-trivial scaling limit of the model in a so-called \textit{intermediate disorder} regime, where the intensity of the disorder is also rescaled with the size of the system.

\smallskip

To motivate this research, let us provide a  short and necessarily incomplete review of  results that can be found in the literature concerning the localization transition for directed polymer, and convergence towards a continuum model after rescaling.
For a complement we refer to \cite{C17} for an introduction to the directed polymer model, with an extensive list of references.
Before we start, let us mention that the bulk of the literature on directed polymer (including~\cite{C17}) uses a different formalism and writes the Gibbs weight in the exponential form 
$\exp\big( \sum_{n=1}^N\beta \tilde\eta_{n,S_n} \big)$, rather than $\prod_{n=1}^N \big( 1+\beta \eta_{n,S_n} \big)$, and assumes that the variable $\tilde \eta$ has finite exponential moments of all orders 
most of the results remain valid in our setup under the assumption $\bbE[\eta^2]<\infty$.
We will later comment on the necessity to adopt the product form in this work, see Remark~\ref{rem:expoGibbs}.

\subsubsection*{The localization transition}

The typical behavior of $(S_{n})_{n=1}^N$ under $\bP^{\eta}_{N,\beta}$ in the large $N$ limit depends on the asymptotic behavior of the partition function. Under the assumption 
$\bbE[\eta]=0$, the sequence $(Z^{\eta}_{N,\beta})_{N\ge 1}$ is a martingale for the natural filtration associated with~$\eta$ (
see for instance~\cite{Bol89}) and thus converges almost surely towards a limit $Z^{\eta}_{\infty,\beta}$. A simple tail sigma-algebra argument yields the following dichotomy
\begin{equation}
 \bbP[ Z^{\eta}_{\infty,\beta}=0 ]\in \{0,1\}.
\end{equation}
Whether  $Z^{\eta}_{\infty,\beta}>0$ or $Z^{\eta}_{\infty,\beta}=0$ holds almost surely depends on the dimension, on $\beta$ and on the distribution of $\eta$.
The regime where $Z^{\eta}_{\infty,\beta}>0$ is called \textit{weak disorder} regime 
and the one where $Z^{\eta}_{\infty,\beta}=0$ is referred to as \textit{strong disorder} regime.
In the weak disorder regime, the influence of the environment is not noticeable on large scales. It has been proved that under weak disorder, the distribution of 
$(S_{n})_{n=1}^N$ rescaled diffusively converges to that of a Brownian Motion (see \cite{CY} and references therein).

\smallskip

Under the \textit{strong disorder} assumption, it is believed that the environment has 
an influence on the trajectories behavior noticeable even on large scale. So far, this phenomenon has been better understood in the sub-regime of \textit{very strong disorder} which corresponds to exponential decay to zero of the partition function, or more precisely 
when
\begin{equation}
p(\beta):=\lim_{N\to\infty} -\frac{1}{N}\log Z^{\eta}_{N,\beta} >0,
\end{equation}
where the limit is in the almost sure sense (the existence of $p(\beta)$ is proved in \cite{CSY04}). In the very strong disorder regime, trajectories are believed to localize around favorite corridors of the environment. Rigorous localization results for the endpoint of the trajectories have been proved in  \cite{CH02, CSY03}  and recently refined in  \cite{BC20}.

\smallskip

It has been proved in \cite{CY} that the quantity $p(\beta)$ is increasing in $\beta$ (in the case of exponential Gibbs weight, see in \cite[App.~A]{Vi19} how the proof adapts to the present setup), meaning that there exists $\beta_c\in[0,1]$ such that very strong disorder holds if and only if $\beta>\beta_c$. The critical intensity $\beta_c$ (when positive) should mark the transition from a diffusive to a localized regime: it is believed that weak disorder holds as long as $\beta<\beta_c$
(see \cite{Vi20} for a recent development on this conjecture).

\begin{rem}\label{hivy}
 The weak/strong disorder terminology has been defined in the case where  $Z^\eta_{N,\beta}$ has finite expectation, that is $\bbE[\eta]=0$. When $\bbE[\eta]=\infty$ we say by convention that very strong disorder holds for every $\beta>0$ (strong localization properties have been proved in that case, see \cite{vargas07}). 
\end{rem}

This phase transition has been studied, mostly under the assumption that 
$\bbE[\eta^2]<\infty$ (the common assumption
in the exponential setup is that $\bbE[e^{\gb \tilde \eta}] <\infty$ for all $\gb$).
Under this assumption, it has been showed that a diffusive phase exists in dimension $d\geq 3$
for sufficiently small $\gb$, \textit{i.e.}\ $\gb_c>0$,
see~\cite{Bol89,IS88}.
On the other hand,
in dimension $d=1,2$ there is no phase transition 
and the polymer is localized  for all $\gb>0$, \textit{i.e.}\ $\gb_c=0$,
see~\cite{CV06} for $d=1$
and~\cite{Lac10pol} for $d=2$ (see also \cite{CH02,CSY03} for earlier results in this direction).

\subsubsection*{The intermediate disorder regime}
Under the assumption $\bbE[\eta^2]<\infty$,
dimensions $d=1$ and $d=2$ are the only dimensions where the value of $\beta_c$ is known. Hence they are the ideal setup in which one can study the crossover regime between a diffusive behavior (at $\beta=0$) and localized behavior (for $\beta>0$).
The idea is to tune the disorder intensity $\beta_N$ to zero as $N$ tends to infinity so that the probability $\bP^{\eta}_{N,\beta_N}( (S_{\lceil Nt\rceil }/\sqrt{N})_{t\in[0,1]}\in \cdot)$ converges (in distribution) to a random continuum distribution (which is not the Wiener measure, obtained when $\beta=0$).

\smallskip

This has been called the \emph{intermediate disorder regime} in the literature, and has been sucessfully studied in the case $d=1$ \cite{AKQ10, AKQ14b, AKQ14}.
In this case the approach is simply to find $\beta_N$ such that $Z^{\eta}_{N,\beta_N}$ converges in distribution
to a non-degenerate limit. For the directed polymer model in dimension $d=1$,
when $\bbE[\eta^2] <\infty$,
the correct scaling turns out to be $\gb_N \sim \hat \gb  N^{-1/4}$;
note that it makes the length $N$ of the system
proportional to the correlation length $|p(\gb)|^{-1} \asymp \gb^{-4}$, see \cite{AY15,Nak16}.

\smallskip

The scaling limit which is obtained, called the \textit{continuum directed polymer}, is the analog of the discrete model where the random walk $S$ and the environement $\eta$ are replaced by their respective scaling limits: Brownian Motion  and the space-time Gaussian white noise (see below for more details on this construction). This continuum model is intimately related to the Stochastic Heat Equation (SHE) with multiplicative white noise.

\smallskip

In dimension $d=2$, the situation is more complicated. The description of the crossover regime in that case is far from complete but has witnessed important progress in recent years.
One of the reasons why this case is more delicate is that the SHE with multiplicative white noise is ill-defined (see \cite{BC98}) so that the limit must be of a different nature. We assume in the following discussion that $\bbE[\eta^2]=1$ for normalization purpose.
In~\cite{CSZ15} the scaling under which  $Z^{\eta}_{N,\beta_N}$ admits a non-trivial limit has been identified ($\gb_N \sim \hat \beta ( \log N)^{-1/2}$ with $\hat \beta<\sqrt{\pi}$), but is has been later shown in  \cite{CSZ20} that in that regime, disorder disappears in the scaling limit of $\bP^{\eta}_{N,\beta_N}( (S_{\lceil Nt \rceil}/\sqrt{N})_{t\in[0,1]}\in \cdot)$.
In order to obtain a disordered scaling limit
one needs to take $\gb_N=\sqrt{\pi} ( \log N)^{-1/2}(1 + b (\log N)^{-1} )$,
with $b\in \bbR$ as a variable parameter  (note that this choice for $\gb_N$ also makes the length~$N$ of the system  proportional to the correlation length $|p(\gb)|^{-1} \asymp e^{\pi /\gb^2}$, see \cite{BL17}). More precisely it has been shown in this regime that the distribution of the partition function is tight and that its subsequential limits are non-trivial, but uniqueness and the description of the limit
remain challenging open problems.
Progresses have been made recently in this direction,
see~\cite{CSZ18scaling,gu2019}, and the existence of a scaling limit for the polymer measure has been derived for the related hierarchical model \cite{clark2019}.

%
%
%
%
%
%
%
%
%
%
%
%

\subsubsection*{Power-law disorder and crossover regime}

The case $\bbE[\eta^2]=\infty$ 
has been investigated more recently in~\cite{Vi19}, where the author studied the localization transition under the assumption that
$\bbP(\eta>z)\sim z^{-\ga}$ as $z$ tends to infinity,
for some $\ga\in (1,2)$ (when $\alpha\in (0,1]$ according to Remark \ref{hivy}, we necessarily have very strong disorder for every $\beta$ when $\alpha\in(0,1]$, and the case $\bbE[\eta^2]<\infty$ covers the case $\alpha>2$).
In that case the presence of a phase transition depends on the dimension but also on the value of $\alpha$.
If $d>\frac{2}{\ga-1}$
 a weak disorder phase exists
(\textit{i.e.}\ $\gb_c>0$),
and if $d \leq  \frac{2}{\ga-1}$
there is no phase transition (\textit{i.e.}\ $\gb_c=0$).
Additionally, when $d<\frac{2}{\ga-1}$,
then the behavior of the free energy close to criticality
is given by $|p(\gb)| = \gb^{\nu +o(1)}$
as $\gb\downarrow 0$,
with $\nu = \frac{2\ga}{2- \ga(d-1)}$.

%
%

%
%
%

\smallskip
The main goal of this article is to
study the intermediate disorder
regime when
 $\eta$ is in the domain of attraction of an $\alpha$-stable law for some $\alpha\in (0,2)$
 such that \eqref{def:eta} holds.
 The continuum object towards which $\bP^{\eta}_{N,\beta_N}( (S_{Nt}/\sqrt{N})_{t\in[0,1]}\in \cdot)$ should converge has been constructed in the companion paper~\cite{BL20_cont} and is the ($\ga$-stable) Lévy noise counterpart of the Gaussian continuum polymer considered in~\cite{AKQ14b}, that we mentioned above. Contrary to the Gaussian model which only exists in dimension~$1$, the Lévy continuum polymer can be constructed in arbitrary dimension provided that the L\'evy measure associated with the noise satisfies some requirement, which depends on the dimension and includes the $\alpha$-stable noise when  $\alpha < \min(1+\frac{2}{d},2)$.

The main achievement of this paper is to prove that the convergence holds if $\beta_N$ is scaled correctly.
The correct scaling is given by
taking $\gb_N$
proportional to $N^{-\gamma+o(1)}$  with $\gamma = \frac{d}{2\ga}(1+\frac 2d -\ga)$. The $o(1)$ correction depends on the slowly varying function considered in \eqref{def:eta}. Note that this makes $N$ roughly proportional to $|p(\gb)|^{-1}$, see above.

\begin{rem}
\label{rem:expoGibbs}
Let us mention that the directed polymer model
with a heavy-tail environment
has already been considered, for instance
in~\cite{AL11,BT19,DZ16},
but in the setup where the Gibbs weights are  in the exponential form
$\exp(\sum_{n=1}^N \gb \tilde \eta_{n,S_n})$.
Let us simply stress that in this setup, when the distribution of $\tilde \eta$ has power-law decay,
polymer trajectories localize very strongly close to 
a single trajectory which gets its energy mostly from high energy sites (in fact this is the case even with a tail exponent $\alpha\ge 2$ since, also in that case,  $e^{\gb \tilde \eta_{n,S_n}}$ has infinite expectation).
Additionally,  the intermediate disorder limit  has less of a rich behavior (at least at the level of the partition function):
if $\beta_N$ is sent to zero the partition function either goes to $1$
or to $\infty$, see~\cite{BT19,DZ16}.
Our framework~\eqref{def:Zn}
allows for the appearance non-trivial intermediate disorder regime even when~$\eta$ has a heavy tail.
This is essentially due to the fact that after rescaling
$(\eta_{n,x})_{(n,x)\in \bbN\times \bbZ^d}$ possesses a scaling limit as a distribution, whereas $(e^{\gb \tilde \eta_{n,S_n}})_{(n,x)\in \bbN\times \bbZ^d}$ never does for heavy-tail environements, even after recentering.
\end{rem}

\subsection{The continuum directed polymer with L\'evy $\alpha$-stable noise}
\label{sec:contiZ}


Let us now describe briefly how the continuum model is constructed.
In doing so, we introduce some important notation and results that will be useful in the rest of the paper.
For more details on the construction and the main properties of the continuum model we refer to the introduction of \cite{BL20_cont}.
Formally the model is obtained by replacing
the random walk $(S_n)_{n\ge 1}$ and the field $(\eta_{n,x})_{(n,x)\in \bbN\times \bbZ^d}$  in \eqref{defdiscrete} by their corresponding  scaling limits.
A rigorous presentation of this object requires the introduction of a few definitions and notation.

\smallskip

Let us fix some finite time horizon $T=1$ for simplicity
 and let $(B_t)_{t\in [0,1]}$ be a $d$-dimensional 
standard Brownian motion.
We denote by $\bQ$ its law and $\rho_t(x)$ its transition kernel, that is
\begin{equation}
\label{def:kernel}
\rho_t(x) =  \frac{1}{(2\pi t )^{d/2}} e^{- \frac{\|x\|^2}{2t}} \, .
\end{equation}
Let us also introduce, for $0 < t_1<\cdots < t_k$
and $x_1,\ldots, x_k\in \bbR^d$ 
the multi-steps kernel
\begin{equation}
\varrho (\bt,\bx) = \prod_{i=1}^k \rho_{t_i-t_{i-1}} (x_i-x_{i-1}) \, ,
\end{equation}
with the convention $t_0 =0$ and $x_0=0$.
The scaling limit of the field
$(\eta_{n,x})_{(n,x)\in \bbN\times \bbZ^d}$
is a one-sided L\'evy $\ga$-stable noise on $\bbR \times \bbR^d$. Let us briefly introduce this object.
Consider $\go$ a Poisson point process on $\bbR \times \bbR^d \times (0,\infty)$ with intensity
\begin{equation}
\label{def:omega}
\dd t \otimes \dd x \otimes  \ga \ups^{-(1+\ga)} \dd \ups \, ,
\end{equation}
whose law we also denote $\bbP$ (it will draw no confusion).
Then formally,
in the case $\ga\in(1,2)$, the $\ga$-stable noise $\xi_{\go}$
is a random measure
obtained by summing 
weighted Dirac masses $\ups \gd_{(t,x)}$
on points $(t,x,\ups) \in \go$
and subtracting a non-random quantity 
so that it is centered in expectation.
The main difficulty is that when $\ga \in (1,2)$ the centering term is infinite,
so we need an approximation procedure.
For $a\in (0,1]$, interpreting $\go$ as a set of points,
we introduce the random measure 

 \begin{equation}
 \label{def:xia}
\xi_{\go}^{(a)} := \Big( \sum_{(t,x,\ups) \in \go} \ups \ind_{\{\ups \geq a\}}  \gd_{(t,x)} \Big) - \kappa_a  \cL \, ,
\end{equation}
where $\cL$ is the Lebesgue measure on $\bbR\times \bbR^d$ and
\begin{equation}\label{defkappaa}
\kappa_a :=
\begin{cases}
0 &\text{ if } \alpha \in(0,1), \\
\log(1/a) & \  \text{ if } \alpha = 1,\\
\frac{\alpha}{\alpha-1}a^{1-\alpha}  & \  \text{ if } \alpha \in (1,2).
\end{cases}
\end{equation}
The L\'evy $\ga$-stable noise $\xi_{\go}$
is then defined as the distributional limit of $\xi_{\go}^{(a)}$
 when $a$ is sent to zero.
Let us stress that when $\ga \in (0,1)$ the sum \eqref{def:xia} yields a locally finite Borel measure  $\xi_{\go}:= \xi^{(0)}_{\go}$
so the approximation procedure is not needed.
When  $\alpha\in [1,2)$, the total variation $|\xi_{\go}^{(a)}|$ diverges when $a \downarrow 0$, but $\xi^{(a)}_\go$ converges to a limiting distribution in the local Sobolev space $H_{\rm loc}^{-s}(\bbR^{d+1})$ for $s>\frac{d+1}{2}$.
(The definition of this functional space is recalled in Appendix~\ref{app:taixi}.)

\smallskip
Once we have defined the continuum counterparts of $(S_n)_{n\geq 0}$ and $(\eta_{n,x})_{(n,x)\in \bbN\times \bbZ^d}$ (namely $(B_t)_{t\geq 0}$
and $\xi_{\go}$),
then  the partition function of the
continuum polymer in L\'evy $\ga$-stable noise 
is  formally defined as 
\begin{equation}\label{lexpress}
\cZ^{\go}_{\beta}
= 1+ \sum_{k=1}^\infty \beta^k\int_{ \mathfrak{X}^k \times (\bbR^d)^k}  \varrho(\bt,\bx)  \prod_{i=1}^k \xi_{\go} ( \dd t_i, \dd x_i) \, ,
\end{equation}
where $\mathfrak{X}^k := \{ (t_1,\ldots t_k) \in (\bbR^d)^k 
\colon 0<t_1<\cdots <t_k <1 \}$.
This formally corresponds to the Wick
expansion of $\bE[ \mathbf{:} \exp(\gb H_{\go}(B)) \mathbf{:}]$,
where the energy functional is
$H_{\go}(B)  =  \xi_{\go} \big( \int_{0}^1 \gd_{(t,B_t)} \dd t\big) $, \textit{i.e.}\ $\xi_{\go}$ integrated
against  the graph of the Brownian trajectory $(B_t)_{t\in[0,1]}$. 
The main result of~\cite{BL20_cont}
is to give a mathematical interpretation for the formal integral \eqref{lexpress} and of the corresponding probability measure $\bQ_{\gb}^{\go}$ on the Wiener space 
\[
C_0([0,1]) := \big\{ \gp \colon [0,1] \to \bbR^d \, : \ 
\gp \text{ is continuous and } \gp(0)=0  \big\}\, ,
\]
endowed with the topology of uniform convergence.
This mathematical construction relies on
an approximation procedure which  
we now outline.
Let us  introduce two families of functions on $C_0([0,1])$
\begin{equation}\label{thespaces}
\begin{split}
 \cB&:= \left\{ \, f\colon C_0([0,1]) \to \bbR \ : \ f \text{ measurable and bounded}\, \right\},\\
 \cC&:=  \left\{ \, f \colon C_0([0,1]) \to \bbR \ : \ f \text{ continuous and bounded} \, \right\}\, .
\end{split}
\end{equation} 
We also denote $\cB_b$ (resp.~$\cC_b$)
the set of functions  $f\in \cB$ (resp.~$f\in \cC$)
with bounded support.

\noindent
For $a>0$, we define
for any $f\in \cB$
\begin{equation}\label{lafirstdef}
\cZ^{\go,a}_{\beta} (f) = \bQ(f) + \sum_{k=1}^\infty \beta^k\int_{ \mathfrak{X}^k \times (\bbR^d)^k}   \varrho(\bt,\bx,f)\xi_{\go}^{(a)} ( \dd t_i, \dd x_i) \, ,
\end{equation}
where
\[
\varrho(\bt,\bx, f) = \varrho(\bt,\bx)  \bQ\Big[ f((B_t)_{t\in [0,1]})  \, \Big| \,  B_{t_i} =x_i \ \forall  i\in \lint 1,k \rint \Big] \, .
\]
The notation $\bQ( \, \cdot  \, | \,  B_{t_i} =x_i \ \forall  i\in \lint 1,k \rint )$
is a shorthand to designate 
the law of the concatenation of $k$ independent Brownian bridges connecting $(t_{i-1},x_{i-1})$ to $(t_i,x_i)$
for $i\in \lint 1, k\rint $.
To see that~\eqref{lafirstdef} makes sense requires some work,
and is ensured by~\cite[Prop.~2.5 \& Prop.~3.1]{BL20_cont}.
Let us now state another result of~\cite{BL20_cont}, which gives
a representation of the partition function
$\cZ^{\go,a}_{\beta} (f)$ as a sum that will be useful in what follows,
and ensures in particular the positivity of $\cZ^{\go,a}_{T,\beta} (f)$ for positive $f$.

\begin{lemma}{\cite[Lem.~3.3]{BL20_cont}}
\label{prop:eazy}
We have, for any $f\in \cB$
\begin{equation}
\label{rewritefun}
 \cZ^{\go,a}_{\beta} (f):= \sum_{\sigma \in \cP(\go)} w_{a,\gb}(\sigma,f)
\end{equation}
where $\cP(\go)$ is the set of finite subsets of $\go$, 
and if $\sigma:= \{(t_i,x_i,u_i), i=1,\dots, k\}$ with $0\le t_1<\dots<t_k\le 1$, we define
$w_{a,\gb} (\sigma,f)$ as
\[
w_{a,\gb}(\sigma,f):=  e^{-\beta \kappa_a }  \gb^{|\sigma|} \varrho( {\bt},{\bx},f)  \prod_{i=1}^{k}  \ups_i\ind_{\{ \ups_i\ge a\}} \, .
\]
\end{lemma}

Note that $f \mapsto \varrho(\bt,\bx,f)$
is linear
and so is $f\mapsto \cZ_{\gb}^{\go,a}(f)$.
The above result ensures that that $\cZ_{\gb}^{\go,a}(f)\geq 0$
if $f\geq 0$
and that $\cZ_{\gb}^{\go,a} := \cZ_{\gb}^{\go,a} (\mathbf{1})$ is positive and finite.
Therefore, for $a>0$, we can define the polymer measure  with truncated noise
$\bQ_{\gb}^{\go,a}$
 on $C_0([0,1])$ by
\begin{equation}
\bQ_{\gb}^{\go,a} (A) :=  \frac{1}{\cZ_{\gb}^{\go,a}}\,  \cZ_{\gb}^{\go,a} (\mathbf{1}_A)   \, ,
\end{equation}
for any Borel set $A\subset C_0([0,1])$.
The main result of~\cite{BL20_cont} shows
that if $\ga$ is smaller than a critical threshold,
then $\bQ_{\gb}^{\go,a}$ converges almost surely
for the weak topology on probability measures,
to a non-trivial (\textit{i.e.}\ disordered) probability measure
$\bQ_{\gb}^{\go}$, referred to as the continuum polymer with stable noise.
Let us define
\begin{equation}
\label{def:alphac}
\alpha_c=\alpha_c(d)  = \min\Big( 1+\frac{2}{d},2\Big) \, .
\end{equation}

\begin{theorema}[see~\cite{BL20_cont}]
\label{thm:contmeasure} 
Assume $\alpha\in (0,\alpha_c)$ with $\alpha_c$ defined in \eqref{def:alphac}. Then
there exists a random probability on  $\bQ^{\go}_{\beta}$ on $C_0([0,1])$ such that almost surely we have
$
 \lim_{a\to 0} \bQ^{\go,a}_{\beta}= \bQ^{\go}_{\beta}
 \,.
$
More precisely we have almost surely
\begin{equation}\label{conv1}
 \lim_{a \to 0}\cZ^{\go,a }_{\beta} = \cZ^{\go }_{\beta}\in (0,\infty),
  \end{equation}
and there exists 
a linear form $\cZ^{\go}_{\beta}(\cdot)$ on $\cC$  such that $\bbP$-almost surely, 
\begin{equation}\label{conv2}
\forall f\in \cC,\quad 
\lim_{a \to 0}\cZ^{\go,a }_{\beta}(f)=\cZ^{\go}_{\beta}(f) \, .
\end{equation}
\end{theorema}

\begin{rem}
We stress that in the case $\ga\geq \ga_c$,
Proposition~2.10 in \cite{BL20_cont} shows that 
$\lim_{a\to 0} \cZ^{\go,a}_{\beta} =0 $ a.s.,
so the limiting partition function is degenerate.
This is essentially due to the fact that when  $\alpha$ gets larger, the relative importance of small atoms in the noise increases.
When $\alpha \geq \ga_c$, the accumulation of these small atoms make the fluctuations of
 $\log \cZ^{\go,a}_{\beta}$ diverge when $a$ tends to zero so that, in that case, the limiting partition function is not defined, nor is the corresponding the continuum polymer model.
\end{rem}

\begin{rem}
Let us mention that for practical reason the centering term $\kappa_a$ for the noise considered in \eqref{defkappaa} differs from the one adopted in \cite{BL20_cont} when $\alpha \ne 1$. 
More precisely, in~\cite{BL20_cont} we use the centering
\begin{equation}\label{defkappa'}
\kappa'_a :=
\frac{\alpha}{\alpha-1}(a^{1-\alpha}-1)   \quad  \text{ if } \alpha\in  (0,1)\cup(1,2).
\end{equation}
This only changes the definition of $\cZ^{\go,a}_{\beta}(f)$  and thus that of $\cZ^{\go}_{\beta}(f)$ by a multiplicative factor $e^{\beta(\kappa_a-\kappa'_a)}=e^{\frac{\beta\alpha}{\alpha-1}}$, and thus does not affect the definition of $\bQ^{\go,a}_{\beta}$ and of $\bQ^{\go}_{\beta}$.
\end{rem}

%
%

\section{Main results}

In the following, we use standard notation: $f\sim g$ means that $f/g \to 1$ and $f=o(g)$ means that $f/g \to 0$; we also use the notation $a\wedge b = \min(a,b)$.

\subsection{Convergence  to the continuum  polymer in $\alpha$-stable noise}

Our main result shows that the martingale limit
$\cZ^{\go}_{ \beta}$ corresponds to the scaling limit of
the partition function of the directed polymer with heavy-tail disorder, when $\beta_N$ is sent to $0$ with the appropriate rate.
We first state the convergence of the partition function, which is simpler but instructive,
before we turn to the convergence of the probability measure $\bP_{N,\gb_N}^{\eta}$.
In order to describe the scaling regime that we are considering for the intensity $\beta_N$,
we need to introduce some notation.

Let $\hat \varphi$ be a slowly varying function such that
$u\mapsto \hat \varphi(1/u) \, u^{-1/\alpha}$ is a generalized inverse of $z\mapsto\bbP(\eta\ge z)$, meaning that
$\bbP( \eta \geq \hat \varphi(1/u) \, u^{-1/\alpha} ) \sim u$ as $u\downarrow 0$; recall~\eqref{def:eta}.
We define
\begin{equation}
\label{def:VN}
V_N :=  (2 d^{d/2})^{-1/\alpha} N^{\frac{1}{\alpha} \left(1+\frac{d}{2}\right)}\hat \varphi(N^{ 1+\frac{d}{2}}) \,.
\end{equation}
Note that we have $\bbP[\eta>V_N ] \sim 2 d^{d/2} N^{-(1+\frac{d}{2})}$ as $N\to\infty$,
so that $V_N$ is the order of magnitude of the maximal value of the field $(\eta_{n,x})_{(n,x) \in \bbN\times \bbZ^d}$
inside the region  $[0,N]\times [-\sqrt{N},\sqrt{N}]^d$.

\begin{theorem}
\label{thm:convergence}
Let us assume that the distribution of the environment $\eta$ satisfies \eqref{def:eta} for some $\alpha\in (0,\alpha_c)$, with $\alpha_c = \min(1+\frac 2d ,2)$ as in \eqref{def:alphac}.
 Let $\hat \gb>0$ and set
\begin{equation}\label{scalechoice}
 \beta_N:= \frac{1}{2}   \hat \beta \, \Big( \frac{N}{d} \Big)^{d/2} \,    V_N^{-1} \, .
\end{equation}
Then when $\alpha\in (0,1)\cup (1,\alpha_c)$ we have the following convergence in distribution
\begin{equation}
 Z^{\eta}_{N ,\beta_N} \stackrel{N\to \infty}\Longrightarrow \cZ^{\go }_{\hat \beta}\, .
\end{equation}
 When $\alpha=1$, if one sets
$
\gamma_N:=  \frac{1}{2d^{d/2}}  N^{1+\frac{d}{2}}\, V_N^{-1}\bbE\left[\eta\ind_{\{1+\eta\le V_N\}}\right]  ,
$ 
then we have
\begin{equation}
 e^{-\gamma_N  \hat \beta}Z^{\eta}_{N ,\beta_N} \stackrel{N\to \infty}\Longrightarrow    \cZ^{\go }_{\hat \beta}\, .
\end{equation}
\end{theorem}

\noindent Let us stress that our choice~\eqref{scalechoice}
gives $\gb_N = N^{-\nu +o(1)}$
with $\nu = \frac{d}{2\ga} (1+\frac{d}{2} -\ga)$.
Our assumption that $\ga <\ga_c$ ensures
in particular that 
$\lim_{N\to\infty} \gb_N =0$.

\begin{rem}
Let us stress that in the case $\ga=1$,
we have $\gamma_N \leq 0$ if $\bbE[\eta] =0$
and $\gamma_N \geq 0$ if $\bbE[\eta] = \infty$ (at least for $N$ large).
Note also that by definition of $V_N$, we have $V_N^{-1} \gp(V_N) =2 d^{d/2} N^{-(1+\frac{d}{2})}(1+o(1))$, so that
\[
\gamma_N \stackrel{N\to\infty}{\sim}  \gp(V_N)^{-1} \bbE[ \eta \ind_{\{1+\eta \leq V_N\}}].
\]
Setting $L(t) := \bbE[\eta \ind_{\{1+\eta \leq t\}}]$, 
we get from \cite[Prop. 1.5.9.a.]{BGT89} that $L(\cdot)$ is a slowly varying function,
with $|L(t)| /\gp(t) \to \infty$ as $t\to\infty$.
This proves in particular that $\gamma_N$ is slowly varying, with $\lim_{N\to\infty} |\gamma_N| =+\infty$.
\end{rem}

\begin{rem}
Let us stress that when $\ga>\ga_c$, 
then with the definition~\eqref{scalechoice} we have
$\lim_{N\to\infty} \gb_N=+\infty$. The condition $\gb_N \in [0,1]$ is therefore violated and $\bP_{N,\gb_N}^{\eta}$ is not defined (recall the definition~\eqref{defdiscrete}).
Also the fact that there exists $\beta_c\in (0,\infty]$ such that $\lim_{N\to \infty}Z^{\eta}_{N,\beta}$ exists and is non-degenerate, for $\beta\in (0,\beta_c)$ (see \cite{Vi19}), indicates that the choice of $\beta_N$ in~\eqref{scalechoice} is not the correct scaling in that case.

The case $\ga=\ga_c$ is more delicate and 
we may have  $\beta_c=0$ and or $\beta_c>0$ depending on the slowly varying function $\varphi$, although indentifying the necessary and sufficient condition on $\varphi$ is a challenging open problem. It makes sense to look for a limit only in the case when $\beta_c=0$.
Similarly to the case $d=2$ and $\bbE[\eta^2]<\infty$ (see for instance \cite{CSZ18scaling}), there should be a scaling of the intensity $\tilde \gb_N$ satisfying $\lim_{N\to\infty}\tilde \beta_N$ and $\tilde \beta_N \ll \beta_N$ (with $\beta_N$ in~ \eqref{scalechoice})  which produces a non-trivial limit 
of the partition function and of the polymer measure;
more precisely, in that case, one may need to integrate the starting point on scale $\sqrt{N}$ for the partition function to converge, exactly as in \cite{CSZ18scaling}.
Since $\tilde \beta_N$ is much smaller than $\beta_N$ we should have  $\lim_{N\to\infty}Z_{N,\gb_N\wedge 1}^{\eta}=0$,
as proven in~\cite{Vi19} in the case of a constant function $\varphi$
and $\gb_N \equiv \gb \in (0,1)$.
The correct scaling choice for $\tilde \gb_N$ in the case $\ga=\ga_c$ is far from obvious, and determining the scaling limit seems even more challenging. We leave these questions for future work.
\end{rem}

Our convergence result is in fact much richer than Theorem~\ref{thm:convergence} in various ways. First we do not prove the convergence of the partition function alone, but also that of the measure 
$\bP^{\eta}_{N,\beta_N}$ towards~$\bQ^{\go}_{\hat \beta}$.
Additionally, we show that the continuous environment $\go$ appearing in the limit corresponds to the scaling limit of $\eta$ as a distribution.
Let us introduce $\xi_{N,\eta}$  the measure on~$\R^{1+d}$
obtained by rescaling the discrete environment (on a diffusive scale)
\begin{equation}
\label{def:xienv}
\xi_{N,\eta} :=   \frac{1}{V_N} \sum_{(n,x)\in \mathbb{H}_d}\big( \eta_{n,x}   -  \bbE[\eta\ind_{\{\eta \leq V_N\}}] \ind_{\{\ga=1\}}  \big) \gd_{(\frac{n}{N}, \frac{x}{\sqrt{N/d}})} 
 \,,
\end{equation}
where $\bbH_d$ denotes the set of time-space lattice points that can be reached by a random walk starting from $0$, that is the set of points  $(n,x)$ in $\bbN\times \bbZ^d$ such that $n$ and $\|x\|_1$ have the same parity.

\smallskip

We let $\mathcal M_1$ denote the space of probability distributions on $C_0([0,1])$ equipped with the topology of weak convergence.
Finally, let $S_t^{(N)}$ be the linear interpolation 
of a random walk trajectory, rescaled diffusively:
\begin{equation}\label{defsn}
S^{(N)}_t:= \sqrt{\frac{d}{N}} \Big( (1-\alpha_t) S_{\lfloor Nt \rfloor}  +\alpha_t S_{\lfloor Nt \rfloor+1}    \Big)\,,
 \quad \text{ with } \  \alpha_t= Nt-\lfloor N t \rfloor \, .
\end{equation}
The definition of the local Sobolev space $H^{-s}_{\mathrm{loc}}$ is recalled in Appendix~\ref{app:taixi}.

\begin{theorem}\label{thm:moreconverge}
With the same assumption as for Theorem \ref{thm:convergence}, we have 
the following convergence in distribution in $H^{-s}_{\mathrm{loc}}\times \cM_1$, $s>d$
\begin{equation}
\left(\xi_{N,\eta} \,,\, \bP^{\eta}_{N,\beta_N}\Big( (S^{(N)}_t)_{t\in [0,1]} \in \cdot \,\Big) \right)  \stackrel{N\to \infty}\Longrightarrow    \big(  \xi_\go,\bQ^{\go }_{\hat \beta} \big) 
\, .
\end{equation}
\end{theorem}

\begin{rem}
The powers of $2$ in the definition of~$V_N$ (and $\gb_N$) just comes from the fact that our random walk can only visit
half of the lattice sites (just like the $\sqrt{2}$ factor appearing in~\cite{AKQ14}). 
They would not appear if one considered a lazy random walk, adjusting the diffusion coefficient in the definition of  $\rho_t(x)$ accordingly.
The powers of $d$ in $V_N$ (and in $\gb_N$ and~\eqref{defsn})
comes from the adjustment of the diffusion coefficient
of the simple random walk.
\end{rem}

\begin{rem}
Let us mention that the joint convergence of the environment and of the partition function is not specific
to the case of heavy-tailed environments. It also holds for systems with disorder in the Gaussian universality class \cite{AKQ14,  CSZ14, CSZ13}, in which case the environment converges to the Gaussian white noise that appears in the construction of the limiting partition function. This fact is however usually not mentioned in the literature on the subject (a notable exception is the recent contribution \cite{BS2020} on the random field Ising model), even though it can  be deduced  almost directly  from the proofs presented in \cite{AKQ14,  CSZ14, CSZ13} (according to personal communication with Caravenna, Sun and Zygouras). 
The situation of systems with marginal disorder such as the
two-dimensional directed polymer, discussed in \cite{CSZ15, CSZ18scaling, clark2019, gu2019} is very different. In that case, the randomness appearing in the limit is expected to be independent of the distributional limit of the environment.
\end{rem}

\subsection{Possible extensions of the result}
We now comment on some directions in which our result could be extended.

\subsubsection*{Point-to-point partition functions}
In the above, we only treat the case of a point-to-line partition function, 
but we could also consider the point-to-point partition functions.
For integers $n_1\leq  n_2$ and $x_{1}, x_2 \in \bbZ^d$ such that $\|x_2-x_1\|_1$ has the same parity as $n_2-n_1$, we
define
\[
Z_{\gb_N}^{\eta} \big[  (n_1,x_1) , (n_2,x_2) \big] := \bE\Big[\prod_{n=n_1 +1}^{n_2} \big( 1+\beta \eta_{n,S_n} \big) \ind_{\{S_{n_2}=x_2\}}  \, \Big| \, S_{n_1} =x_1\Big] \, .
\]
Then, choosing
$\gb_N$ as in~\eqref{scalechoice}
and taking $n_1,n_2 \in \bbN$ and $x_1, x_2\in \bbZ^d$
(with $\|x_2-x_1\|_1$ having the same parity as $n_2-n_1$), 
  with the following scaling 
\begin{equation*}
 \begin{split}
  \lim_{N\to\infty} N^{-1} n_1 = t\,,\qquad  &\lim_{N\to\infty} (N/d)^{-1/2} x_1  = x\,,\\
  \lim_{N\to\infty} N^{-1} n_2 = t'\,,
  \qquad 
&\lim_{N\to\infty} (N/d)^{-1/2} x_2  = x\,'.
 \end{split}
\end{equation*}
where  $t,t' \in \bbR_+$ satisfy $t<t'$ and $x, x'\in \bbR^d$,
one should obtain the following convergence in distribution,
in the case $\ga\in (0,\ga_c)$
\begin{equation}
\label{conv:p2p}
\tfrac12  (\tfrac1d  N)^{d/2} \,e^{-  \hat \beta \gamma_N\ind_{\{\alpha=1\}}} Z_{\gb_N}^{\eta} \big[ (n_1,x_1) , (n_2,x_2) \big]  \stackrel{N\to \infty}\Longrightarrow     \cZ_{\hat \gb}^{\go} \big[ (t,x) , (t',x') \big]   \, .
\end{equation}
The limit $\cZ_{\hat \gb}^{\go} [ (t,x) , (t',x') ]$ is the continuum point-to-point partition function introduced in~\cite[Section~2.2]{BL20_cont}.
The convergence \eqref{conv:p2p} could also be extended to the joint distribution of finitely many point-to-point partition functions.
We choose not to present the proof of 
such a result in order to keep
the paper lighter and because it does not bring much more insight than Theorem~\ref{thm:moreconverge}.

\subsubsection*{Random walks in the Gaussian domain of attraction}
In this paper, we only consider the model based on the nearest neighbor symmetric simple random walk. The main reason for it is that it is the most frequent setup in which directed polymer is presented. However, our proofs are quite flexible and other types of random walks may be considered provided that they remain within the Gaussian universality class. Indeed, we only make use of the local central limit theorem to establish our result (and a little  bit more when $\alpha\le1$).

Let us assume that the walk (starting from the origin) has i.i.d.\ increments with finite second moment, mean ${\bf m}$ and invertible covariance matrix $\Sigma^2$, meaning that for any $i,j$
\begin{equation}\label{finitevar}
\bbE[ S^{(i)}_1]=m_i \quad \text{ and } \quad \Cov(S^{(i)}_1, S^{(j)}_1)=\Sigma^2(i,j),
\end{equation}
where $S^{(i)}_1$ is the $i$-th coordinate of $S_1$.
Our matrix $\Sigma$ is implicitly defined as the positive definite matrix whose square is the covariance matrix.
For the commodity of exposition, let us also assume that our walk is irreducible and aperiodic, that is for every $x$ there exists some $n_0(x)$ such that $\bP(S_n=x)>0$ for $n\ge n_0$ (doing without these assumptions only entails additional constants in the normalization).

\smallskip
In the case $\alpha\le 1$, we must have an additional assumption on the tail of the increments which guarantees that
the walk does not reach for atypically attractive sites beyond the scale $\sqrt{N}$. We assume that
\begin{equation}\label{extramoment}
\bE[|S_1|^{\gamma}]<\infty \quad  \text{ for some }  \gamma >\frac{d(1-\alpha)}{\alpha}.
\end{equation}
(This ensures that the
bound \eqref{checkitout} holds for all $n\ge1 $, for some $\theta <\alpha$).
Note that the condition~\eqref{extramoment} is stronger than the finite variance assumption only when $\alpha < \frac{d}{d+2}$. 
Let us also stress that this restriction is by no means due to technical limitation: indeed, if for instance $\bP(S_n=x)=(1+o(1)) \|x\|^{-d-\gamma}$ for some $2<\gamma< \frac{d(1-\alpha)}{\alpha}$, then the directed polymer itself is not well defined, in the sense that $Z^{\go}_{N,\beta}=\infty$ a.s.\ for every $\beta>0$ (and in particular Theorem~\ref{extend} below fails to hold).
Before stating the extension of our result, let us redefine the scaling parameters: we set
 \begin{equation}\begin{split}
S^{(N)}_t&:= \frac{1}{\sqrt{N}} \Sigma^{-1}\Big( (1-\alpha_t) S_{\lfloor Nt \rfloor}  +\alpha_t S_{\lfloor Nt \rfloor+1}- {\bf m} Nt    \Big)\,,\\
V_N &:=  \Det(\Sigma)^{1/\alpha} N^{\frac{1}{\alpha} \left(1+\frac{d}{2}\right)}\hat \varphi(N^{ 1+\frac{d}{2}}) \, ,\\
\xi_{N,\eta} &:=   \frac{1}{V_N} \sum_{(n,x)\in \bbN\times \bbZ^d}\big( \eta_{n,x}   -  \bbE[\eta\ind_{\{\eta \leq V_N\}}] \ind_{\{\ga=1\}}   \big) \gd_{(\frac{n}{N}, \frac{x}{\sqrt{N/d}}-{\bf m} N)}  \, .
\end{split}
\end{equation}
One can extend the proof of Theorem \ref{thm:moreconverge} to prove the following.

\begin{theorem}\label{extend}
Assume that the distribution of the environment $\eta$ satisfies \eqref{def:eta} for some $\alpha\in (0,\alpha_c)$, with $\alpha_c = \min(1+\frac 2d ,2)$ as in \eqref{def:alphac}. 
 Under the assumption  \eqref{finitevar}  and \eqref{extramoment}, then setting
 $$\beta_N= \Det(\Sigma)N^{\frac{d}{2}}(V_N)^{-1} \,,$$ 
we have the following convergence in distribution in $H^{-s}_{\mathrm{loc}}\times \cM_1$, $s>d$
\begin{equation}
\left(\xi_{N,\eta} \,,\, \bP^{\eta}_{N,\beta_N}\Big( (S^{(N)}_t)_{t\in [0,1]} \in \cdot \,\Big) \right)  \stackrel{N\to \infty}\Longrightarrow    \big(  \xi_\go,\bQ^{\go }_{\hat \beta} \big) 
\, .
\end{equation}
\end{theorem}

\subsubsection*{Extension to $\gamma$-stable walks}
A natural question that now comes to mind is whether the system admits a similar scaling limit when $S$ is in the domain of attraction of a stable process with exponent $\gamma\in (0,2)$.
For the simplicity of exposition, let us assume that $\bP(S_n=x)=(1+o(1)) \|x\|^{-d-\gamma}$ as $\|x\| \to \infty$.

\smallskip

In that case we strongly believe that our proof techniques can be adapted without major changes. The procedure starts with defining the continuum model, replacing the Brownian kernel by that of an isotropic $\gamma$-stable kernel --- this part is discussed in \cite[Section~2.4]{BL20_cont}.
Then the proof in the present paper, which mostly relies on the local limit theorem, should go through.
Note that in that case our requirement on $\alpha$ becomes
\begin{equation}
  \frac{d}{d+\gamma}<\alpha< \left(1+\frac{\gamma}{d}\right)\wedge 2.
\end{equation}
For a descrition of the scaling limit, we refer to \cite[Section 2.4]{BL20_cont}.

\subsubsection*{Extension to the disordered pinning model}
Since most of the techniques used in the proof do not rely on the specificity of the directed polymer model, we believe that similar results can be derived for other disordered model. 
One specific model for which we are confident that our techniques could adapt is the disordered pinning model -- see~\cite{GB07} for a general introduction. We refer to \cite{CSZ14} for a study of the scaling limit when the environment has a finite second moment and to \cite{LS17} for a study of the model with power-tail disorder.

\smallskip

Consider a sequence $(\eta_n)_{n\in \bbN}$ of i.i.d.\ random variables with a distribution satisfying \eqref{def:eta} and a recurrent renewal process $\tau = \{\tau_0=0,\tau_1,\tau_2,\ldots\}$ on $\bbN$ 
whose inter-arrival law satisfies 
$\bP[ \tau_1=n ]\stackrel{n\to \infty}{\sim} n^{-(1+\gamma)}$, $\gamma\in (0,1)$.
The disordered pinning model is then defined as the modification of the renewal distribution defined by
\begin{equation}\begin{split}
\label{def:pinning}
\frac{\dd \bP_{N,\gb,h}^{\eta}}{\dd \bP} &:=  \frac{1}{Z_{N,\gb,h}^{\eta}} \bE\Big[ \prod_{n=1}^N e^{h \ind_{\{n\in\tau \}}} \big( 1+ \gb \eta_{n} \ind_{\{n\in \tau\}}  \big) \Big] \, ,\\
\text{with } \qquad Z_{N,\gb,h}^{\eta} &:= \bE\Big[ \prod_{n=1}^N e^{h \ind_{\{n\in\tau \}}} \big( 1+ \gb \eta_{n} \ind_{\{n\in \tau\}}  \big) \Big] \, ,
\end{split}
\end{equation}
where $h\in \R$ is a parameter, corresponding to a (homogeneous) pinning field.
We believe that the approach used in the present paper could be used to prove the existence of a non-trivial limit for the distribution 
of the rescaled renewal set $\frac{1}{N}\left( [0,N]\cap \tau \right)$ (considering the Hausdorff topology for subsets of $[0,1]$) when $\alpha<\min(\frac{1}{1-\gamma},2 )$,
if $\beta_N$ and~$h_N$ are  scaled as $\hat \beta N^{1-\gamma- \frac{1}{\alpha}+o(1)}$ and $\hat h N^{-\gamma+o(1)}$. The $o(1)$ in the exponent accounts for a slowly varying correction which is absent if 
$\varphi$ in \eqref{def:eta} is asymptotically equivalent to a constant.
For a description of the scaling limit, we refer again  to \cite[Section 2.4]{BL20_cont}.

\section{Main steps of the proof of Theorem~\ref{thm:moreconverge}}
\label{sec:steps}

In this section, we outline our proof strategy,
and highlight the main steps needed to obtain our main theorem.
Some important notation is introduced.

\subsection{Convergence of marginals and tightness}

Recall here that $\cC$ denote the set of real-valued continuous and bounded functions on $C_0([0,1])$.
Recalling the definition \eqref{defsn} and writing $S^{(N)}$ for $(S^{(N)}_t)_{t\in [0,1]}$ we define for  any $f\in \cC$
\begin{equation}
Z^{\eta}_{N, \gb_N}(f):= \bE \Big[ f(S^{(N)}) \prod_{n=1}^N \big(1+\beta_N \eta_{n,S_n} \big) \Big]  \, .
\end{equation}
Our main task in the proof of Theorem~\ref{thm:moreconverge} is to prove the convergence of finite-dimensional marginals. Since both $ \xi_{N,\eta}$ and $Z^{\eta}_{N,\gb_N}(\cdot)$ are linear forms, it is in fact sufficient to prove the convergence for one-dimensional marginals.
\begin{proposition}\label{marginals}
Given $\psi$ a smooth compactly supported function on $\bbR^{d+1}$ and   $f \in \cC$, we have the following joint convergence in distribution, under the assumptions of Theorem~\ref{thm:convergence},
\begin{equation}
 \left( \langle \xi_{N,\eta}, \psi \rangle,  e^{-\hat \beta \gamma_N \ind_{\{\alpha=1\}}} Z^{\eta}_{N,\gb_N}(f)  \right) \stackrel{N\to \infty}{\Longrightarrow}  \left( \langle \xi_{\go}, \psi \rangle, \cZ^{\go}_{\hat\beta}(f)  \right) \,.
\end{equation}
\end{proposition}

The proof of Proposition~\ref{marginals} is the core of the paper and its steps are outlined in Section~\ref{martdecompo}.
The actual proof is carried out in Sections~\ref{sec:letrucmoinsdur} and~\ref{sec:letrucdur}.
Once we have Proposition~\ref{marginals}, it only remains to prove tightness of $\xi_{N,\eta}$ and $\bP^{\eta}_{N,\hat\beta}( S^{(N)}\in \cdot\, )$. 
The first result is standard. A proof is included in Appendix~\ref{app:taixi} for completeness.
\begin{lemma}\label{letight}
 The sequence of distributions of $(\xi_{N,\eta})_{N\ge 1}$ under $\bbP$ is tight in $H^{-s}_{\mathrm{loc}}$ for $s>d$.
\end{lemma}
\noindent For the second tightness result, proven in Section~\ref{sec:latight}, recall that~$\cM_1$ denotes the set of probability distributions on $C_0([0,1])$ equipped with the weak convergence topology.
\begin{proposition}\label{latight}
  The sequence of distributions of random probabilities $\bP^{\eta}_{N,\gb_N}(S^{(N)}\in \cdot\, )$ under $\bbP$ is tight in $\cM_1$.
\end{proposition}

\noindent
Our main result then follows readily from the above statements, as we now show.
\begin{proof}[Proof of Theorem \ref{thm:moreconverge} using Propositions \ref{marginals} and \ref{latight}]
 Since tightness is already proven, one only needs to prove the convergence for the marginals of the type
 \begin{equation}
   \left( \langle \xi_{N,\eta}, \psi\rangle,  \bE^{\eta}_{N,\gb_N}\big(f(S^{(N)})\big)\right)= \Bigg( \langle \xi_{N,\eta}, \psi \rangle,  \frac{Z^{\eta}_{N,\gb_N}(f)}{Z^{\eta}_{N,\gb_N}}\Bigg) . 
 \end{equation}
But this follows immediately  from the almost-sure positivity of  $\cZ^{\go}_{\hat\beta}$  (cf.~\eqref{conv1}) and from Proposition \ref{marginals}, which yields as a corollary
 \begin{equation}
 \left( \langle \xi_{N,\eta}, \psi \rangle, e^{-\hat\beta \gamma_N \ind_{\{\ga=1\}} } Z^{\eta}_{N,\gb_N}(f), e^{-\hat\beta \gamma_N \ind_{\{\ga=1\}} }Z_{N,\gb_N}^{\eta}\right) 
 \stackrel{N\to \infty}{\Longrightarrow}  \left( \langle \xi_{\go}, \psi \rangle, \cZ^{\go}_{\hat\beta}(f) , \cZ^{\go}_{\hat\beta}\right) \,.
 \end{equation}
 This concludes the proof of Theorem~\ref{thm:convergence}.
\end{proof}


\subsection{Proposition~\ref{marginals} via cutoff approximation}
\label{martdecompo}

Let us now describe the main ingredients in the proof of Proposition~\ref{marginals}.
Firstly, we replace $(Z^{\eta}_{N,\beta_N}(f))$ by a
cutoff approximation, analogous to the one used in the continuous case, see~\eqref{lafirstdef}. 
This martingale approximation is obtained by keeping the environment only at sites where~$\eta$ is larger than a certain threshold.
We fix the threshold to be equal to $ a V_N$,
recall~\eqref{def:VN}.
 The scaling $V_N$ is chosen so that  the region typically visited by the polymer, 
\textit{i.e.} a cylinder of length $N$ and width $\sqrt{N}$ according to the diffusive scaling of the random walk, contains only finitely many points above the threshold.
These points, in the limit, correspond to the points in the Poisson point process $\go$ (defined in~\eqref{def:omega}) for which $u\ge a$. 
When $\eta_{n,x}$ is smaller than this value we replace it by its conditional average.
Given $a\in (0,1]$, we set 
\begin{equation}
 \eta^{(a)}_{n,x} := 
 \begin{cases}  \eta_{n,x} \quad & \text{ if } \ 1+\eta_{n,x}\ge  a V_N,\\
 - \kappa_N^{(a)} \quad &  \text{ if } \  1+\eta_{n,x} <  a V_N \, ,
 \end{cases} 
\end{equation}
with 
\begin{equation}
\label{def:kappaN}
\kappa_N^{(a)} :=
\begin{cases}
-\bbE \big[ \eta_{n,x} \, \big| \, 1+ \eta_{n,x}<  a V_N \big] \,  &\text{ if } \alpha \in [1,2) \, ,\\
 0 \,  &\text{ if } \alpha \in (0, 1).
\end{cases}
\end{equation}
The assumption $\bbP(\eta > -1)=1$ ensures that $\eta^{(0)}=\eta$ almost surely. 
Note that $\kappa^{(a)}_N$ is positive for large $N$ when $\alpha\ne 1$ but it is negative when $\alpha=1$ and $\bbE[\eta]=\infty$.
When $\bbE[\eta]=0$ (this implies $\alpha\ge 1$),
note that $ \eta^{(a)}_{n,x}$ is still a centered variable, and that $(\eta^{(a)}_{n,x})_{a \in (0,1]}$ is a càdlàg time-reversed martingale
for the filtration
\begin{equation}\label{lafiltr}
\mathcal G_{a} =\mathcal G^{(N)}_{a}:=\sigma \big( \eta_{n,x}\ind_{\{ (1+ \eta_{n,x})\ge  a V_N \}}, (n,x)\in \bbN\times \bbZ^d \big) \, .
\end{equation}
We also set for $b> a$ a truncated version of $\eta^{(a)}$
\begin{equation}
\label{doubletronc}
 \eta^{[a,b)}_{n,x} := 
 \begin{cases}   - \kappa_N^{(a)} \quad &  \text{ if }  1+\eta_{n,x} <  a V_N \, ,\\ \eta_{n,x} \quad & \text{ if }  1+\eta_{n,x}\in [  a V_N, b V_N ) \, , \\
 0  \quad & \text{ if } 1+\eta_{n,x}\geq b V_N.
 \end{cases} 
\end{equation}
With some abuse of notation, we will denote $\eta^{(a)}$ (resp.~$\eta^{[a,b)}$) a generic random variable with the same law as $\eta_{n,x}^{(a)}$ (resp.~$\eta^{[a,b)}_{n,x}$).
We now define, for $f\in \cC$, the approximation $ Z^{\eta,a}_{N,\gb_N}(f)$ of $Z_{N,\gb_N}^{\eta}$ using the above cutoff:
\begin{equation}
\label{martingaprox}
\begin{split}
 Z^{\eta,a}_{N,\gb_N}(f):=  \bE \Big[ f(S^{(N)})\prod_{i=1}^N \big(1+\beta_N \eta^{(a)}_{n,S_n}\big) \Big],\\
  Z^{\eta,[a,b)}_{N,\gb_N}(f):=  \bE \Big[ f(S^{(N)})\prod_{i=1}^N \big(1+\beta_N \eta^{[a,b)}_{n,S_n}\big) \Big] .
\end{split}
\end{equation}
Let us stress that, if $f\geq 0$, we have $Z^{\eta,[a, b)}_{N,\gb_N}(f) \leq Z^{\eta,a}_{N,\gb_N}(f)$.
Note that when $\bbE[\eta]=0$
(which implies $\alpha\ge 1$)  we have
\begin{equation}
\label{lesmartingales}
  Z^{\eta,a}_{N,\gb_N}(f) = \bbE \Big[  Z^{\eta}_{N,\beta_N}(f) \, \Big| \, \mathcal G_a \Big] \, \quad \text{ and } \quad  Z^{\eta,[a,b)}_{N,\gb_N}(f):=  \bbE \Big[ Z^{\eta,[0,b)}_{N,\gb_N}(f) \, \Big| \, \mathcal G_a\Big].
\end{equation}
These identities are also valid when $\bbE[\eta]=\infty$  and $\alpha=1$ (when $\alpha\in(0,1)$ it is false due to the different definition of $\kappa^{(a)}_N$), the first one only in the case~$f\ge 0$ for  which the conditional expectations are unambiguously defined.

To prove Theorem~\ref{thm:convergence}, we are first going to prove 
that  $Z^{\eta,a}_{N,\gb_N}(f)$ converges to the corresponding continuum (truncated) partition function $\cZ^{\go,a}_{\hat \beta} (f)$ defined in \eqref{lafirstdef}.  In fact we even prove a joint convergence with the cutoff approximation of the environment.

\begin{proposition}
\label{letrucpasdur}
 For any $a>0$, 
given $\psi$ a smooth compactly supported function on $\bbR^{d+1}$ and $f\in \cC$, we have the following joint convergence in distribution
\[
 \Big( \langle\xi_{N,\eta}^{(a)},\psi \rangle  , e^{-\hat \gb \gamma_N  \ind_{\{\ga=1\}}}  Z^{\eta,a}_{N,\gb_N}(f)  \Big) \stackrel{N\to \infty}{\Longrightarrow} \left( \langle \xi_{\go}^{(a)},\psi \rangle ,  \cZ^{\go,a}_{\hat\gb}(f)  \right) \,,
\] 
where $\xi_{N,\eta}^{(a)}$
is the cutoff approximation of $\xi_{N,\eta}$,
\begin{equation}
\label{xicut}
\xi_{N,\eta}^{(a)} :=  \frac{1}{V_N} \sum_{(n,x)\in \mathbb{H}_d} \big( \eta^{(a)}_{n,x}  -  \bbE[ \eta \ind_{\{\eta \leq V_N\}}]  \ind_{\{\ga =1\}} \big)\, \gd_{(\frac{n}{N}, \frac{x}{\sqrt{N/d}})}  \,.
\end{equation}
\end{proposition}

While the proof of Proposition \ref{letrucpasdur} requires some care, it follows a quite standard roadmap.
The proof is carried out in Section~\ref{sec:letrucpasdur}.
To complete the proof of Theorem~\ref{thm:convergence}, we need to show that when $a$ is close to $0$, 
then~$Z^{\eta,a}_{N,\gb_N}(f)$ and $\langle\xi_{N,\eta}^{(a)},\psi \rangle$ are close to  $Z_{N,\hat \gb}^{\eta,0}(f)$and $\langle \xi_{N,\eta},\psi \rangle$ respectively, 
uniformly in~$N$. Recall that from Theorem~\ref{thm:contmeasure}  (more precisely~\eqref{conv1}), we already know that~$ \cZ^{\go,a}_{\hat \beta}(f)$ is close to $ \cZ^{\go}_{\hat \beta}(f)$.
Similarly, $\langle \xi_{\go}^{(a)},\psi \rangle$ is close to $\langle \xi_{\go},\psi \rangle$ as a result of convergence of $\xi_{\go}^{(a)}$ in $H^{-s}_{\mathrm{loc}}$, $s>(d+1)/2$.

\begin{proposition}\label{letrucdur} 
If $\ga\in(0,2)$, we have for any $f\in \cC$
\begin{equation} 
\label{uniformity}
\lim_{a\to 0} \sup_{N\ge 1}\bbE\Big[ \left(e^{-\hat \beta \gamma_N \ind_{\{\ga=1\}} }\big| Z^{\eta,a}_{N,\gb_N}(f)- Z^{\eta,0}_{N,\gb_N}(f) \big|\right)\wedge 1  \Big]=0 \, .
\end{equation}
\end{proposition}

\begin{rem}
When $\ga\in(1,2)$, we can in fact prove uniform convergence in 
 $L^1(\bbP)$
in~\eqref{uniformity} instead of convergence in probability, that is 
 $$
 \lim_{a\to 0} \sup_{N\ge 1}\bbE\Big[ \big| Z^{\eta,a}_{N,\gb_N}(f)- Z^{\eta,0}_{N,\gb_N}(f) \big|  \Big]=0 \,. 
 $$

\end{rem}
Let us also state the analogous result for $\langle \xi_{N,\eta},\psi \rangle$, whose proof is easy (and postponed to Appendix~\ref{app:taixi}).
\begin{lemma}\label{lem:uniformxi} 
If $\ga\in (0,2)$, then for any smooth and compactly supported $\psi$ we have
\begin{equation} 
\lim_{a\to 0} \sup_{N\ge 1}\bbE\Big[  \langle \xi_{N,\eta} - \xi_{N,\eta}^{(a)} ,  \psi  \rangle^2 \Big]=0 \, .
\end{equation}
\end{lemma}

We stress that the core of the proof lies in Proposition~\ref{letrucdur}. Its proof is carried out in Section~\ref{sec:letrucdur} and  follows some of the ideas developed  in \cite[Sec.~4]{BL20_cont} for the construction of the continuum partition function, but presents additional technical challenges.

\begin{proof}[Proof of Proposition~\ref{marginals} from Propositions \ref{letrucpasdur} and \ref{letrucdur}]

Given an arbitrary  $\delta>0$, a compactly supported smooth function $\psi$ and  $f\in\cC$,  we are going to show that 
there is some $N_0 =N_0(\delta, \psi)$
such that  for  every $N\ge N_0$, we can find a coupling between $\eta$ and $\go$ (with some abuse of notation we use $\bbP$ for the law of the coupling) such that 
 \begin{equation}\label{zioups}
  \bbP \left( \big|Z^{\eta}_{N,\beta_N}(f)-\cZ^{\go}_{\hat \beta}(f)  \big|> \delta \right) \le \delta \qquad \text{ and }  \qquad  \bbP\big( \big|\langle \xi_{N,\eta}-\xi_{\go} ,  \psi \rangle \big| > \delta \big) \leq \delta\, .
 \end{equation}
First we want to approximate the two partition functions
by their counterparts with truncated environment.
Using Proposition \ref{letrucdur} and Theorem \ref{thm:contmeasure}-\eqref{conv1}, we can choose $a_0=a_0(\delta)$ small enough such that for every value of $N$ we have for all $a\le a_0(\delta)$
 \begin{equation}\label{lavrai}
 \begin{split}
  \bbP\Big( e^{-\hat\beta \gamma_N \ind_{\{\ga=1\}} } \big|  Z^{\eta,a}_{N,\gb_N}(f)- Z^{\eta}_{N,\beta_N}(f) \big|> \delta/3  \Big)&\le \delta/3,\\
  \bbP\Big( \big| \cZ^{\go,a}_{\hat \beta}(f)- \cZ^{\go}_{\hat \beta}(f) \big|> \delta/3  \Big) \le \delta/3.  
  \end{split}
 \end{equation}
Similarly, thanks to Lemma~\ref{lem:uniformxi} and since $\xi_{\go}^{(a)}$ converges to $\xi_{\go}$ in $H_{\rm loc}^{-s} (\bbR^{d+1})$, have for $a\le a_0(\delta)$ (lowering the value of $a_0$ if necessary)
and every value of $N$ 
  \begin{equation}\label{lavrai2}
 \begin{split}
  \bbP\big(  \big| \langle  \xi_{N,\eta} - \xi_{N,\eta}^{(a)}, \psi  \rangle  \big|> \delta/3  \big) \le \delta/3,\\
   \bbP\big( \big| \langle   \xi_{\go} - \xi_{\go}^{(a)}, \psi \rangle  \big|> \delta/3  \big)  \le \delta/3.  
  \end{split}
 \end{equation}

Now we can conclude by observing that from Proposition \ref{letrucpasdur} ,
for $N\ge N_0$ sufficiently large
one can find a coupling of $\eta$ and $\go$ (depending of course on $N$) which is such 
that with probability larger than $1-\delta/3$ one has 
   \begin{equation}\label{lavrai3}
 \begin{split}
 \bbP\left[| e^{-\hat\beta \gamma_N \ind_{\{\ga=1\}} } Z^{\eta,a}_{N,\gb_N}-\cZ^{\go,a}_{\hat \beta}(f)|>\delta/3\right]&\le \delta/3,\\
  \bbP\left[ |\langle \xi_{N,\eta}^{(a)}-\xi_{\go}^{(a)} , \psi \rangle|>\delta/3 \right]\le \delta/3,
  \end{split}
  \end{equation}
which combined with  \eqref{lavrai}-\eqref{lavrai2} implies \eqref{zioups}.
\end{proof}

\subsection{Organization of the remainder of the paper}

Now that we have outlined the main steps of the proof,
let us briefly describe how the different parts of the proof
are articulated.

\begin{itemize}
\item  In Section~\ref{sec:uzeful}, some technical preliminaries are presented. In particular we describe an expansion
of the partition function analogous to~\eqref{rewritefun}
and give some comparison  estimates between $Z^{\eta,a}_{N,\gb_N}(f)$ and  $Z^{\eta,[a,b)}_{N,\gb_N}(f)$.

\item  In Section~\ref{sec:latight}, we prove the tightness of $\bP_{N,\gb_N}^{\eta} ( S^{(N)}\in \cdot)$, \textit{i.e.}\
Proposition~\ref{latight}.
Note that the tightness of $(\xi_{N,\eta})$ is proven in Appendix~\ref{app:taixi}.

\item In Section~\ref{sec:letrucpasdur}, we carry out the proof  of Proposition~\ref{letrucpasdur},
\textit{i.e.}\ of the convergence of the partition function 
with cutoff environment.

\item Section~\ref{sec:letrucdur} contains the proof of Proposition~\ref{letrucdur}, \textit{i.e.}\ of the uniformity (in $N$) of the martingale convergence of $Z_{N,\gb_N}^{\eta,a}$ as $a\downarrow 0$.
This is the most technical part of the paper and adapts ideas 
developed in~\cite{BL20_cont} in the continuum setting.

\item In the Appendix, some further technical estimates are collected: in Appendix~\ref{app:compare}, we prove an estimate that allows us to control different expectations with respect to $\eta$;
in Appendix~\ref{app:taixi}, we collect results on the measure $\xi_{N,\eta}$ and in particular we prove Lemma~\ref{lem:uniformxi}.
\end{itemize}

\section{Technical preliminaries}

\subsection{A collection of useful estimates}
\label{sec:uzeful}

Let us collect here a few identities and asymptotic equivalents that will be useful in the computations in the remainder of the paper.
By definition~\eqref{def:VN} of~$V_N$, we have
\[
V_N^{-\ga} \gp(V_N) \stackrel{N\to\infty}{\sim}     2 d^{d/2}  N^{-(1+\frac{d}{2})}.
\]
Also, by definition $\gb_N :=  \frac12\hat \gb  (\frac{N}{d})^{d/2}  V_N^{-1} $, see~\eqref{scalechoice}, so it verifies
\begin{equation}\label{uzefulrel}
\begin{split}
\gb_N V_N& \ \ =  \ \ \tfrac{1}{2} \hat \gb \, \big( \tfrac{N}{d}\big)^{\frac d2} \, , \\
\gb_N V_N^{1-\ga}  \varphi(V_N) &\stackrel{N\to\infty}{\sim}    \hat \gb  N^{-1},\\
\gb_N^2 V_N^{2-\ga}  \varphi(V_N) & \stackrel{N\to\infty}{\sim} \tfrac{1}{2} d^{-d/2} \hat \gb^2 N^{\frac{d}{2}-1} .
\end{split}
\end{equation}
In the case $\ga=1$, we will also use that by definition of $\gamma_N$,
\begin{equation}
\label{uzeful2}
\gb_N \bbE[\eta \ind_{\{1+\eta \leq V_N\}}] = \hat \gb N^{-1} \gamma_N\, .
\end{equation}
As far as truncated first and second moments of $\eta$ are concerned, 
we have asymptotically,  as $u$ diverges to infinity
\begin{equation}
\label{moments}
\begin{split}
 \bbE[ \eta \ind_{\{ (1+\eta) < u\}}]&=
  \frac{\ga}{1-\alpha} \, u^{1-\ga} \gp(u) (1+o(1)),  \quad \text{ for } \ga \in (0,1)\cup(1,2)\\
   \bbE[\eta^2 \ind_{\{(1+\eta)<u\}}] & = \frac{\ga}{2-\ga} \, u^{2-\ga}\gp(u) (1+o(1)).
  \end{split}
\end{equation}
Choosing $u=aV_N$ for some fixed $a>0$ and combining~\eqref{moments} with~\eqref{uzefulrel}, this yields in the large $N$ limit 
\begin{align}
\label{moments2}
\gb_N \bbE[ \eta \ind_{\{ (1+\eta) < a V_N\}}]&=
  \frac{\ga}{1-\alpha} \,\hat \gb  a^{1-\ga}  N^{-1} +o(N^{-1}) \, ,  \quad \text{ for } \ga \in (0,1)\cup(1,2) \,,\\
 \gb_N^2  \bbE[\eta^2 \ind_{\{(1+\eta)<aV_N\}}] & =  \frac{\ga d^{-d/2}}{2(2-\ga)}  \hat \gb^2 \, a^{2-\ga}  N^{\frac d2 -1}+ o(N^{d/2-1}) \,. \label{moments3}
\end{align}
Note furthermore (simply by monotonicity) that 
the $o(N^{-1})$ term in \eqref{moments2} is uniform in $a\in(1,\infty)$ (in the sense that $\lim_{N\to\infty}\sup_{a>1} N \times \cdot =0$) when $\alpha>1$. 
In the same manner the  $o(N^{d/2-1})$ term in \eqref{moments3} is uniform in $a\in(0,1)$ for any $\alpha\in (0,2)$.

We therefore  find that when $\alpha\in(1,2)$, $\kappa_N^{(a)}$ defined in \eqref{def:kappaN}  (recall also the definition of $\kappa_a$ in~\eqref{defkappaa}) satisfies, as~$N$ tends to infinity for any $a\in (0,\infty)$
\begin{equation}
\label{meaneta}
\beta_N \kappa_N^{(a)} = - \gb_N \bbE\big[ \eta \, \big| \,(1+ \eta) <  a V_N ]
 =  \hat \gb \,  \kappa_a \, N^{-1}+o(N^{-1}) \,.
\end{equation}
Note that when $\ga\in (0,1)$ we have set $\kappa_N^{(a)} =0$ and $\kappa_a=0$,
so that  \eqref{meaneta} (without the equality with the middle term) is also valid in that case.
 When $\alpha=1$, after a direct computation (and using~\eqref{uzeful2}), we have
\begin{equation}
\label{cas=1}
 \beta_N \kappa_N^{(a)}=  - \hat \beta N^{-1}\gamma_N + \hat \gb \,  \kappa_a \, N^{-1}(1+o(1)).
\end{equation}
All together, in view of the definition of $\gb_N$ and $\gamma_N$, we can rewrite~\eqref{meaneta}-\eqref{cas=1}
as
\begin{equation*}
\gb_N \kappa_N^{(a)} = - \hat \gb N^{-1} \gamma_N \ind_{\{\ga=1\}} +   \hat \gb \kappa_a N^{-1} +o(N^{-1}).
\end{equation*}
In particular, we see that for any $a \in (0,1)$
\begin{equation}
\label{lalimit}
 \lim_{N\to\infty} e^{-\hat \gb \gamma_N \ind_{\{\ga =1\}} } (1- \gb_N \kappa_N^{(a)})^N  = e^{ - \hat \gb \kappa_a} \, .
\end{equation}

\subsection{Expansion of the partition function}

In order to prove the convergence of the truncated partition function, we are going to rewrite it as a sum, which is the discrete equivalent of \eqref{rewritefun}.
Then the convergence of the partition function is going to follow from the convergence of each individual term.
 Let us define, for $a\in [0,1]$ and $b \in(1,\infty]$,
\begin{equation}
\gO_{N}^{[a,b)}(\eta) := \big\{ (n,x)\in  \lint 1, N\rint\times  \bbZ^d \, : \,  1+ \eta_{n,x} \in [  aV_N,bV_n)  \big\} \, ,
\end{equation}
and let $\cP(\gO^{[a, b)}_N)$ denote the set of finite sequences 
$(n_i,x_i)_{i=1}^{k}$ taking values in $\gO_{N}^{[a,b)}$ 
and satisfying $n_1<n_2< \dots<n_k$. 
We let $(\bn, \bx)=(n_i,x_i)_{i=1}^{|\bn, \bx|}$ denote a generic element of $\cP(\gO^{[a, b)}_N)$ where  $|\bn, \bx|\ge 0$ is the length of the sequence (a length zero corresponds to the empty sequence).
We set $p_n(x)=\bP(S_n=x)$, and using the convention $n_0=0$, $x_0=0$ we define 
\begin{equation}
 \label{lesprods}
 p(\bn, \bx,f)=\bE \big[  f(S^{(N)}) \ind_{\{\forall i \in \lint 1, |\bn, \bx|\rint \,, \, S_{n_i}=x_i \} } \big]
 \quad \text{ and } \quad \eta_{\bn, \bx} =\prod_{i=1}^{|\bn, \bx|} \eta_{n_i,x_i}.
\end{equation}
We write simply  $p(\bn, \bx)$ when $f\equiv 1$.
Let us set for $a\in(0,1]$ 
\begin{multline}\label{znlabar}
\bar Z_{N, \gb_N}^{\eta,[a,b)}(f):= \bE \Big[f(S^{(N)})\prod_{n=1}^N \big( 1+\beta_N\eta_{n,S_n}\ind_{\{(1+\eta_{n,S_n})\in [aV_N,bV_N)\}} \big) \Big]\\
= \sum_{(\bn, \bx)\in \cP(\gO^{[a,b)}_N)} \beta^{|\bn, \bx|}_N p(\bn, \bx,f) \, \eta_{\bn, \bx} \,,
\end{multline}
where the second expression is simply obtained by performing an expansion of the product and taking the expectation of each term.
We write $\gO_{N}^{(a)}(\eta)$ and $\bar Z_{N,\hat \gb}^{\eta,a}$ when $b =\infty$. Recall that 
$\cB_b$ designates the set of bounded functions with bounded support on $C_0([0,1])$.
\begin{proposition}\label{revritez}
For  any non-negative function $f\in \cB_b$, and any $a\in (0,1]$ and $ b\in(1,\infty]$ we have
the following convergence in probability
 \begin{equation}\label{resultf}
  \lim_{N\to \infty}
  \frac{e^{-\hat \beta \gamma_N \ind_{\{\alpha=1\}}}
  Z^{\eta,[a,b)}_{N,\beta_N}(f)}{\bar Z_{N,\gb_N}^{\eta,[a,b)}(f)}=e^{-\hat \beta \kappa_a} \,.
 \end{equation}
Furthermore, 
 for  all  $a\in (0,1]$ there exists $N_0(a)\ge 1$
such that for any
non-negative function $g$  and
${\blue b}\in(1,\infty]$, for $N\ge N_0(a)$ we have
\begin{equation}\label{resultg}
 e^{-\hat \beta \gamma_N \ind_{\{\alpha=1\}}} Z^{\eta,[a, b)}_{N,\beta_N}(g)  \le  2e^{-\hat \beta \kappa_a} \bar Z^{\eta,[a,b)}_{N,2\beta_N}(g).
\end{equation}
\end{proposition}

\begin{proof}
For notational simplicity we prove the result only for $b=\infty$. We are going to control the quotient for the contribution of every single trajectory.
We have for any nearest neighbor trajectory 
\begin{equation}\label{unration}
 \frac{\prod_{n=1}^N(1+\beta_N\eta^{(a)}_{n,S_n})}{\prod_{n=1}^N (1+\beta_N\eta_{n,S_n}\ind_{\{(1+\eta_{n,S_n})>aV_N\}})}=(1-\kappa^{(a)}_N\gb_N)^{N-\#\{ n\in \lint 1,N\rint, (1+\eta_{n,S_n})> aV_N  \}}.
\end{equation}
Recalling~\eqref{lalimit},
we just have to verify that the term $(1-\kappa^{(a)}_N\gb_N)^{-\#\{ n\in \lint 1,N\rint, (1+\eta_{n,S_n})> aV_N  \}},$ can be controlled uniformly over the set of trajectories which are contributing to the partition function.

\smallskip

Recall that we assumed that $f$ has bounded support: we let $A=A_f$ be such that $f( \varphi)=0$ if $\|\varphi \|_{\infty}\ge A$. For any realization of $S$ such that $f(S^{(N)})>0$ we have
\begin{multline}\label{therectangle}
 \#\{ n\in \lint 1,N\rint, (1+\eta_{n,S_n})> aV_N  \} \\ \le  \#\{ (n,x)\in \lint 1,N\rint\times  \lint -A \sqrt{N/d},  A \sqrt{N/d}\rint^d, (1+\eta_{n,x})> aV_N  \}.
\end{multline}
Now, with our definition of $V_N$, the r.h.s.~in \eqref{therectangle} has an expectation which is uniformly bounded in $N$. Since $\beta_N \kappa^{(a)}_N$ tends to $0$, this implies in particular that
\begin{equation}
 \lim_{N\to \infty} (1-\beta_N \kappa^{(a)}_N)^{\#\{ (n,x)\in \lint 1,N\rint\times  \lint -A \sqrt{N/d},  A \sqrt{N/d}\rint^d, (1+\eta_{n,x})> aV_N  \}}=1
\end{equation}
in probability,
from which we can conclude that \eqref{resultf} holds. 
For \eqref{resultg}, we simply note that we have
\begin{multline}\label{2unration}
 \frac{\prod_{n=1}^N(1+\beta_N\eta^{(a)}_{n,S_n})}{\prod_{n=1}^N (1+2\beta_N\eta_{n,S_n}\ind_{\{(1+\eta_{n,S_n})>aV_N\}})}\\
 = (1-\gb_N\kappa^{(a)}_N)^{N} \prod_{n=1}^N\frac{1+\beta_N \eta_{n,S_n}\ind_{\{(1+\eta_{n,S_n})>aV_N\}}}{(1-\beta_N \kappa^{(a)}_N \ind_{\{(1+\eta_{n,S_n})>aV_N\}} )(1+2\beta_N  \eta_{n,S_n}\ind_{\{(1+\eta_{n,S_n})>aV_N\}})} \, .
\end{multline}
Now, this is bounded by $ (1-\gb_N\kappa^{(a)}_N )^{N}$ since for $N$ sufficiently large all terms in the last product  are smaller than one.
 Then one concludes thanks to
 \eqref{lalimit}.
\end{proof}

\subsection{Truncating large weights}

We prove here the following proposition,
which allows us to truncate large weights in the partition function (this is especially needed when $\ga\in (0,1]$).

\begin{proposition}
\label{prop:tronquons}
We have for any $\alpha \in (0,2)$,
\begin{equation}
\lim_{b\to\infty} \suptwo{N\geq 1}{a\in[0,1]} \bbE\left[  \left(e^{-\hat \beta  \gamma_N  \ind_{\{\ga=1\}} } \big( Z^{\eta,a}_{N,\gb_N}-  Z^{\eta,[a,b)}_{N,\gb_N} \big)\right)\wedge 1\right] =0 \, .
\end{equation}
\end{proposition}
%

\begin{proof}
First, let us get rid of the small jumps in the noise. We observe that by using conditional Jensen's inequality 
(recall \eqref{lafiltr}) the quantity we have to bound is smaller than 
\begin{equation}
\bbE\left[  \left(e^{-\hat \beta  \gamma_N  \ind_{\{\ga=1\}} }  \bbE \big[ Z^{\eta,a}_{N,\gb_N}-  Z^{\eta,[a,b)}_{N,\gb_N} \, \big| \, \mathcal G_1 \big]\right)\wedge 1\right] \, .
\end{equation}
Now we have
$ \bbE [ Z^{\eta,a}_{N,\gb_N}-  Z^{\eta,[a,b)}_{N,\gb_N} \, | \, \mathcal G_1 ] = (Z^{\eta,1}_{N,\gb_N}-  Z^{\eta,[1,b)}_{N,\gb_N})$ for $\alpha\ge 1$
(recall~\eqref{lesmartingales}).
For $\alpha\in(0,1)$ (this distinction is necessary because of our choice $\kappa^{(a)}_N=0$ in that case) we have for every $a\ge 0$
\begin{align*}
\bbE [ Z^{\eta,a}_{N,\gb_N}-  Z^{\eta,[a,b)}_{N,\gb_N} \, | \, \mathcal G_1 ] & \le \left(1+ \beta_N \bbE[\eta  \ | \  1+ \eta\le V_N ]\right)^N  (Z^{\eta,1}_{N,\gb_N}-  Z^{\eta,[1,b)}_{N,\gb_N}) \\
& \le C_{\hat \beta}  (Z^{\eta,1}_{N,\gb_N}-  Z^{\eta,[1,b)}_{N,\gb_N}).
\end{align*}
Therefore,
it is sufficient to prove that
$e^{-\hat \beta  \gamma_N  \ind_{\{\ga=1\}} }(Z^{\eta,1}_{N,\gb_N}-  Z^{\eta,[1,b)}_{N,\gb_N})$
converges to $0$ in probability as $b\to\infty$, uniformly in $N$.
Using Proposition~\ref{revritez}-\eqref{resultf}, it is sufficient to show that for some $\theta \in (0,1\wedge \alpha)$ we have
\begin{equation}\label{trucaprouv}
 \lim_{b\to\infty} \sup_{N\geq 1} \bbE\left[\left(\bar Z^{\eta,1}_{N,2\gb_N}-\bar Z^{\eta,[1,b)}_{N,2\gb_N}\right)^{\theta}\right] =0 \, .
\end{equation}
Now, using the representation~\eqref{znlabar} and since $\theta<1$ we have (recall the notation   \eqref{lesprods})
\begin{equation}
\left(\bar Z^{\eta,1}_{N,2\beta_N}-\bar Z^{\eta,[1,b)}_{N,2\beta_N}\right)^{\theta}\le \sum_{(\bn,\bx)\in \cP(\gO^{1}_N)\setminus\cP(\gO^{[1,q)}_N)} \left((2\beta_N)^{^{|\bn,\bx|}} p(\bn,\bx) \eta_{\bn,\bx}\right)^{\theta} \, ,
\end{equation}
Note that since $(\bn,\bx)\in \cP(\gO^{1}_N)$ we have
$1+ \eta_{n_i,x_i} \geq V_N$ for all $i \in \lint  1, |\bn,\bx|\rint$; in particular we have $\eta_{\bn,\bx} \geq 0$.
Taking the expectation with respect to the $\eta$'s and recalling the definition
of $\cP(\gO^{1}_N)$, $\cP(\gO^{[1,q)}_N)$,
we obtain that 
\begin{multline*}
 \bbE\left[\left(\bar Z^{\eta,1}_{N,2\beta_N}-\bar Z^{\eta,[1,b)}_{N,2\beta_N}\right)^{\theta}\right]\\ 
 \le \sum_{k=1}^{N} (2\beta_N)^{k\theta} \bbE \bigg[ \Big(\prod_{i=1}^k \eta_i \Big)^{\theta}\ind_{\{\forall i\in \lint 1,k\rint, \eta_i\ge V_N  ;  \exists j\in \lint 1,k\rint,  \, \eta_j\ge b V_N\}}\bigg] \sumtwo{ n_1<n_2<\dots<n_k} {\bx\in (\bbZ^{d})^k} 
p(\bn,\bx)^{\theta}.
\end{multline*}
Now, we have 
\begin{multline}\label{spoof}
\bbE \bigg[ \Big(\prod_{i=1}^k \eta_i \Big)^{\theta}\ind_{\{\forall i\in \lint 1,k\rint, \eta_i\ge V_N  ;  \exists j\in \lint 1,k\rint,  \, \eta_j\ge bV_N\}}\bigg]\\
\le\bbE \bigg[ \Big(\prod_{i=1}^k \eta_i \Big)^{\theta}\sum_{j=1}^k\ind_{\{\forall i\in \lint 1,k\rint, \eta_i\ge V_N  ;   \, \eta_j\ge b V_N\}}\bigg] = k \bbE \big[  \eta^{\theta}\ind_{\{\eta \ge V_N \}}  \big]^{(k-1)} \bbE \big[  \eta^{\theta}\ind_{\{\eta \ge b V_N \}}  \big]
\end{multline}
One can then easily check  (using Potter's bound,  see~\cite[Thm.~1.5.6]{BGT89})
that there is a constant $C$ such that
for any $c \geq 1$ (we will use it with $c=1$ or $c=b$),
for all $N$ sufficiently large,  
\[
\bbE \big[ \eta^\theta \ind_{\{\eta \ge c  V_N \}}  \big] \leq C c^{\frac12(\theta-\alpha)} V_N^{\theta-\alpha} \gp(V_N).
\]
Together with \eqref{spoof} this implies that

\begin{equation}
\label{hypodelta}
 \bbE \bigg[ \Big(\prod_{i=1}^k \eta_i \Big)^{\theta}\ind_{\{\forall i\in \lint 1,k\rint, \eta_i\ge V_N \, ; \, \exists j\in \lint 1,k\rint,  \, \eta_j\ge bV_N\}}\bigg]\le  k  (C')^k  b^{\frac12 (\theta-\ga)}  V^{\theta k}_N N^{-k\left(1+\frac{d}{2}\right)}\, .
\end{equation}

\noindent Now let us observe that as a consequence of (a sharp version of) the local central limit theorem,
see~\cite[Thm.~2.3.11]{LL10}, there exists a  constant $C'=C'_{\theta}$ such that

\begin{equation}\label{checkitout}
 \sum_{x\in \bbZ^d}p_{n}\big( x \big)^{\theta} \le C'\,  n^{\frac{d}{2}(1-\theta)}
\end{equation}
and hence 
\begin{equation}
 \sumtwo{ n_1<n_2<\dots<n_k} {\bx\in (\bbZ^{d})^k} 
p(\bn,\bx)^{\theta}\le (C')^k \binom{N}{k} N^{ k\frac{d}{2}(1-\theta)}
\leq \frac{(C')^k}{k!} N^{ k( \frac d2(1-\theta) +1  )} .
\end{equation}
Combining these bounds and replacing $\beta_N$ by its value (recall~\eqref{uzefulrel}), we obtain that
\begin{equation}
 \bbE\left[\left(\bar Z^{\eta,1}_{N,\beta_N}-\bar Z^{\eta,[1,b)}_{N,\beta_N}\right)^{\theta}\right]  
\le  \sum_{k=1}^{\infty} k \frac{ (\hat \gb C'')^k }{k!}  b^{\tfrac12 (\theta-\ga)} \leq C_{\hat \beta} b^{\frac12 (\theta-\ga)} \, ,
\end{equation}
which proves~\eqref{trucaprouv}.
\end{proof}

\section{Proof of Proposition \ref{latight}}
\label{sec:latight}

In this section we prove Proposition \ref{latight} assuming that Proposition \ref{marginals} holds.
We start with the easier case $\alpha \in (1,2)$.
We wish to find an increasing sequence of compact sets $\cK_n \subset \cM_1$ which are such that
for all $n$ and~$N$ we have
\begin{equation}
\label{compacts}
 \bbP \left[ \bP^{\eta}_{N,\gb_N}(S^{(N)}\in \cdot\, ) \notin \cK_n \right] \le 2^{-n}.
\end{equation}
Using the tightness of $(Z^{\eta}_{N,\gb_N})^{-1}$, which is ensured by Proposition \ref{marginals} and the positivity of the limit $\mathcal Z^{\go}_{\hat \gb}$ (recall~\eqref{conv1}),
we consider a  sequence $\delta_m$ going to zero such that for all $N$ and $m$
\begin{equation}
\label{deltam}
 \bbP\big( Z^{\eta}_{N,\gb_N} \le \delta_m \big) \le 2^{-m-2}.
\end{equation}
Then, we consider $K_m$ a sequence of compact subsets  of $C_0([0,1])$
such that 
\begin{equation}\label{zorid}
\bP( S^{(N)}\notin K_m )\le 4^{-m}\delta_m \, .
\end{equation}
Note that such a sequence exists simply by the fact that $\bP[ S^{(N)}\in \cdot\, ]$ is a convergent sequence (and hence is tight).
Finally, we set 
\begin{equation}\label{defkkn}
 \cK_n:=\left\{ \mu \in \cM_1 \, :\  \forall m\ge n, \   \mu(K^{\cc}_{m})\le 2^{2-m} \right\}.
\end{equation}
 The set $\cK_n$ is closed and any sequence in $\cK_n$ is tight and thus $\cK_n$ is compact. 
Now, we have by a union bound
\begin{equation}
\label{zorid2}
 \bbP\left[  \bP^{\eta}_{N,\gb_N}(S^{(N)}\in \cdot\,)  \notin \cK_n   \right]
 \le \sum_{m = n}^{\infty} \bbP \left[ \bP^{\eta}_{N,\gb_N}(S^{(N)}\notin K_m) \ge 2^{2-m}    \right],
\end{equation}
and finally 
\begin{multline}
\bbP \left[ \bP^{\eta}_{N,\gb_N}(S^{(N)}\notin K_m) \ge 2^{2-m}    \right]
 \le  \bbP \left( Z^{\eta}_{N,\gb_N}(\ind_{K^{\cc}_m}) \ge 2^{2-m}\delta_m \right)
 +   \bbP \left( Z^{\eta}_{N,\gb_N} \le \delta_m \right)\\
 \le 2^{m-2} (\delta_m)^{-1}\bP( S^{(N)}\notin K_m) + 2^{-m-2}\le 2^{-m-1} \, ,
\end{multline}
where we used Markov's inequality and the fact that $ \bbE[Z^{\eta}_{N,\gb_N}(\ind_{K^{\cc}_m}) ] =\bP( S^{(N)}\notin K_m)$
since $\bbE[\eta]=0$.
Combined with~\eqref{zorid2}, this gives~\eqref{compacts}.

\smallskip

For the case $\ga \in(0,1)$,
we need to use the truncated version of the partition function (recall \eqref{martingaprox}).
We start with the same sequence~$\delta_m$ as above (see~\eqref{deltam}), 
then thanks to Proposition~\ref{prop:tronquons} 
we can fix $b_m$ such that 
\begin{equation}
 \bbP\left(  Z^{\eta}_{N,\gb_N}-Z^{\eta,[0,b_m)}_{N,\gb_N} \ge 2^{1-m}\delta_m\right)\le 2^{-m-3} \, .
\end{equation}
As $b_m$ tends to infinity, using \eqref{moments} together with Potter's bound~\cite[Thm.~1.5.6]{BGT89}, there exists $C>0$ (which may depend on $\hat\beta$) such that for all $N$ and $m$
\begin{equation}\label{tuut}
 \beta_N \bbE[ \eta^{[0,b_m)}]\le  C N^{-1}b_m^{2(1-\alpha)} \,.
\end{equation}
Then we choose $K_m$ a sequence of compacts such that
\begin{equation}\label{zorid3}
\bP\big( S^{(N)}\notin K_m \big)\le 4^{-m-2}\delta_m e^{-Cb_m^{2(1-\alpha)}} \, ,
\end{equation}
and we define $\cK_n$ as in \eqref{defkkn}.
Then,
\begin{multline}\label{dopz}
\bbP \Big[ \bP^{\eta}_{N,\gb_N}  (S^{(N)}\notin K_m)  \ge 2^{2-m}    \Big]
 \le  \bbP \left( Z^{\eta}_{N,\gb_N}(\ind_{K^{\cc}_m}) \ge 2^{2-m}\delta_m \right)
 +   \bbP \left( Z^{\eta}_{N,\gb_N} \le \delta_m \right)\\
 \le   \bbP\left(  Z^{\eta}_{N,\gb_N}-Z^{\eta,[0,b_m)}_{N,\gb_N} \ge 2^{1-m}\delta_m\right) + \bbP\left( Z^{\eta,[0,b_m)}_{N,\gb_N}(\ind_{K^{\cc}_m}) \ge 2^{1-m}\delta_m\right) + 2^{-m-2}\\
 \le   2^{-m-3} +  \bbE\left[Z^{\eta,[0,b_m)}_{N,\gb_N}(\ind_{K^{\cc}_m})\right] 2^{m-1}(\delta_m)^{-1} + 2^{-m-2} \le 2^{-m-1} \, .
\end{multline}
In the last inequality we used the fact that for sufficiently large $m$,  as a consequence of~\eqref{tuut} and \eqref{zorid3}, we have
\begin{equation}
\label{zorid4}
\bbE\left[Z^{\eta,[0,b_m)}_{N,\gb_N}(\ind_{K^{\cc}_m})\right] =
 \big(  1+\beta_N \bbE[ \eta^{[0,b_m)}] \big)^N \bP\big( S^{(N)}\notin K_m \big) \le 4^{-m-2}\delta_m \,.
\end{equation}

Finally for $\alpha=1$ we repeat exactly the same procedure but considering rather the normalized partition functions
\[
e^{-\hat \gb \gamma_N} Z^{\eta}_{N,\gb_N}   \quad  \text{ and } \quad e^{-\hat \gb \gamma_N} Z^{[0,b_m)}_{N,\gb_N}  \,.
\]
 We fix a sequence $(b_m)_{m\geq 1}$ which is such that 
\begin{equation}
  \bbP\left( e^{-\hat \gb \gamma_N}\left( Z^{\eta}_{N,\gb_N}-Z^{\eta,[0,b_m)}_{N,\gb_N}\right) \ge 2^{1-m}\delta_m\right)\le 2^{-m-3} \, .
\end{equation}
In analogy with \eqref{tuut} (as in \eqref{cas=1} but using Potter's bound to have a result which is uniform in $m$), we obtain  that there exists $C>0$ such that for every $m$ and $N$
\begin{equation}
 \beta_N \bbE[ \eta^{[0,b_m)}]\le  -\hat \beta  N^{-1} \gamma_N + C N^{-1}b_m \,.
\end{equation}
Here $b_m$ could be replaced by an arbitrarily small power of $b_m$ but this is irrelevant. Then we fix a sequence of compact sets $K_m$ such that 
\begin{equation}
\bP\big( S^{(N)}\notin K_m \big)\le 4^{-m-2}\delta_m e^{-Cb_m} \, .
\end{equation}
We can finally conclude by repeating \eqref{dopz} (with the extra  $e^{-\hat \gb \gamma_N}$ factor). 
\qed

\section{ proof of Proposition~\ref{letrucpasdur}}
\label{sec:letrucpasdur}

\subsection{Convergence of  $Z_{N, \gb_N}^{\eta,a} (f)$}

 We are going to assume (without loss of generality) that $ 0\leq f \leq 1$.

\subsubsection*{Step 1: Reduction to functions $f$ with bounded support.}
As a first step, we reduce to proving a statement for a function $f$ with bounded support.
For $A>0$, let $h_A \in \cC_b$ be defined by $h_A(\gp) = 1\wedge (\|\gp\|_{\infty} -A )_+$,
and for $f\in \cC$ define $f_A = f h_A \in \cC_b$.
In particular, $f_A=f$ 
on $\cA := \{ \gp \in C_0([0,1]),  \|\gp \|_{\infty} \leq A \}$.

%

\begin{lemma}
\label{lem:boundedtraj}
We have, for any $a>0$,
\begin{equation}
\label{boundedtraj}
\lim_{A\to\infty} \sup_{N\geq 1} \bbE\Big[ e^{-\hat \beta \gamma_N \ind_{\{\alpha=1\}}}\big( Z_{N, \gb_N}^{\eta,a} (f) -  Z_{N, \gb_N}^{\eta,a} (f_A)\big) \wedge 1 \Big] =0 \, .
\end{equation}
\end{lemma}

\begin{proof}
In the case $\ga\in (1,2)$, recalling that $\bbE[\eta]=0$ we have
\begin{equation}
 \bbE\Big[  Z_{N, \gb_N}^{\eta,a} (f) -  Z_{N, \gb_N}^{\eta,a} ( f_A) \Big]\le \bP\big( S^{(N)} \in \cA^{\cc} \big) = \bP\Big(  \sup_{t\in [0,1]} S_t^{(N)} \geq A \Big) \, ,
\end{equation}
which can be  made arbitrarily small by choosing $A$ large (uniformly in $N$).
In the case when $\ga\in(0,1)$ (and similarly for $\alpha=1$ with the $e^{-\hat \beta \gamma_N}$ prefactor) 
we observe that the quantity we have to bound is smaller than 
\begin{equation}
\bbE\Big[ \left(Z_{N, \gb_N}^{\eta,a}(f)- Z_{N, \gb_N}^{\eta,[a,b)} (f_A) \right)\wedge 1 \Big]
\leq \bbE\Big[  \big(  Z_{N, \gb_N}^{\eta,a}(f) -  Z_{N, \gb_N}^{\eta,[a,b)}(f) \big) \wedge 1 \Big]
+ \bbE\Big[  Z_{N, \gb_N}^{\eta,[a,b)}(\ind_{\cA^{\cc}})  \Big] \, .
\end{equation}
The first term can be made arbitrarily small by taking $b$ large by Proposition~\ref{prop:tronquons}.
The second term is equal to 
\[
\big( 1+\beta_N \bbE[\eta\ind_{\{(1+\eta)< bV_N} ] \Big) ^N\bP\big(  S^{(N)} \in \cA^{\cc} \big)
\le C_b \bP \big( S^{(N)} \in \cA^{\cc} \big)
\]
thanks to~\eqref{moments2}.
This can be made arbitrarily small by choosing $A$ large (the case $\alpha=1$ is similar).
\end{proof}

%
%
%

\subsubsection*{Step 2: convergence for $f\in \cC_b$. }
We now show the convergence of the partition function in Proposition~\ref{letrucpasdur} with $f\in \cC_b$ instead of $f\in \cC$.
Also, thanks to Proposition~\ref{revritez}, we prove the convergence of  $\bar Z_{N, \gb_N}^{\eta,a} (f)$
rather than $e^{-\hat \beta \gamma_N \ind_{\{\ga=1\}}} Z_{N, \gb_N}^{\eta,a} (f)$.

\begin{lemma}
\label{lem:conv1}
For any $f\in \cC_b$, 
we have the following convergence in distribution
\[
\bar Z_{N, \gb_N}^{\eta,a} (f)   \stackrel{N\to\infty}{\Longrightarrow } e^{\hat \gb \kappa_a} \cZ^{\go,a}_{\hat \beta} (f)  \,.
\]
\end{lemma}

\begin{proof}

Let us define $\eta^{(N)} := V_N^{-1} \eta$
the rescaled environment, and notice that thanks to~\eqref{def:eta}
and the definition of $V_N$ in~\eqref{def:VN}
we get  that for any $t>a$,
\[
\bbP\big( 1+ \eta^{(N)}_{n,x}  > t \big) \stackrel{N\to \infty}{\sim} 2 d^{d/2} t^{-\ga} N^{- (1+\frac{d}{2})}\, .
\] 
We want to show that the point process $\big(\frac{n}{N} , \frac{x}{\sqrt{N/d}},  \eta^{(N)}_{n,x}\ind_{\{(1+\eta)\ge a V_N\}} \big)_{(n,x)\in \bbH_d}$ converges in distribution towards the Poisson point process~$\go$ (recall~\eqref{def:omega}) restricted to weights larger than or equal to $a$.
Let us specify here a topology. We consider 
$$
\mathcal W:= \{ w \subset  \bbR \times \bbR^d\times \bbR_+, \#\{w\cap A\}<\infty \text{ when $A$ is bounded} \} \,,
$$
equipped with the smallest topology which makes  the maps
$$
w\mapsto \sum_{(t,x,\ups)\in w} g(t,x,u)
$$
continuous, for every continuous function $g$ with bounded support.
As a consequence of the convergence of binomials to Poisson variables
we have the following convergence in $\mathcal W$
\begin{equation}
\label{convPoisson}
 \Big\{  \Big(\frac{n}{N} , \frac{x}{\sqrt{N/d}},  \frac{\eta^{(N)}_{n,x}}{2d^{d/2}} \Big)  \, ,  (n,x) \in \gO_{N}^{a} \Big\}
 \stackrel{N\to \infty}{\Longrightarrow}  \go^{(a)} = \big\{  (t,x,  {\blue  \ups}) \in \go \cap ( [0,1] \times \bbR^d \times [a,\infty) )  \big\} ,
\end{equation}
Now, defining $\cP(w)$ like in  Lemma~\ref{prop:eazy}, notice that the function 
\begin{equation}
 w\mapsto \sum_{\sigma \in \cP(w)}\hat \beta^{|\sigma|} \varrho(\bt, \bx, f)\prod^{|\sigma|}_{i=1} \ups_i 
\end{equation}
is continuous in $\mathcal W$ (this is not a difficult statement but the proof requires some care).
As a consequence we have the following convergence 
\begin{multline}
\label{convPoisson2}
\sum_{(\bx,\bn) \in \cP(\gO^{a}_N)} \hat \beta^{|\bn,\bx|}
\varrho\Big( \frac{\bt}{N}, \frac{\bx}{\sqrt{N/d}}, f \Big) \prod^{|\bn,\bx|}_{i=1} \eta^{(N)}_{x,n} 
\\ \stackrel{N\to \infty}{\Longrightarrow}   \sum_{\sigma \in \cP(\go^{(a)})}\hat \beta^{|\sigma|} \varrho(\bt, \bx, f )\prod^{|\sigma|}_{i=1} \ups_i = e^{\hat \beta \kappa_a}\cZ^{\go,a}_{\hat \beta} (f )\, . 
\end{multline}
Now, to conclude from this that $\bar Z_{N,\gb_N}^{\eta,a} (f)$ converges, using the definition \eqref{znlabar} and our choice for $\beta_N$, we simply need to show that one can replace  
$\varrho(\frac{\bt}{N}, \frac{\bx}{\sqrt{N/d}}, f)$ in the l.h.s.~by
$ (2d^{d/2})^{-|\bn,\bx|}  N^{\frac{d}{2}|\bn,\bx|} p(\bn,\bx,  f )$.
This is a consequence of 
the local central limit theorem for the simple random walk \cite{LL10} 
and of the invariance principle for  random walk bridges \cite{Li1968}
  (the factor $2$ comes from the periodicity of the random walk and the factor $d^{d/2}$ from the fact that the simple random walk has covariance matrix $\frac1d \mathrm{Id}$).
\end{proof}

\subsubsection*{Step 3: conclusion.}
Now, to conclude the proof of Proposition~\ref{letrucpasdur}, we simply need to let $A\to\infty$,
and check that $\cZ_{\hat \gb}^{\go, a} (f_A)$ converges to $\cZ_{\hat \gb}^{\go, a}(f)$ in probability.
But this is simply a consequence of monotone convergence,
recalling the representation~\eqref{rewritefun}.
Combined with Lemma~\ref{lem:boundedtraj} and Lemma~\ref{lem:conv1}, this concludes the proof that
for any $f\in \cC$, we have the convergence
$e^{-\hat \gb \gamma_N \ind_{\{\ga=1\}}} Z_{N,\hat \gb}^{\eta,a}(f) \Rightarrow \cZ_{\hat \gb}^{\go,a}(f)$.

\subsection{Joint convergence with $\langle \xi_{N,\eta}^{(a)} , \psi \rangle$}

To prove the joint convergence of the environment and the
partition function, we simply need to adapt slightly the proof above:
in particular, we only need to adapt the proof of the second step.

\begin{lemma}
For any $a>0$, given $\psi$
a smooth compactly supported function on $\bbR^{d+1}$
and $f\in \cC_b$ we have the following joint convergence in distribution
 \[
 \Big( \langle \xi_{N,\eta}^{(a)} , \psi \rangle\, ,  \bar Z_{N, \gb_N}^{\eta,a} (f)  \Big)  \Longrightarrow    \Big( \langle \xi_{\go}^{(a)} , \psi \rangle\, , e^{\hat \gb \kappa_a}\cZ^{\go,a}_{\hat \gb} (f) \Big).
\]
\end{lemma}


\begin{proof}
Notice that in view of~\eqref{xicut}, for any fixed $a>0$, we can rewrite  $\xi_{N,\eta}^{(a)}$ as
\begin{multline}
\label{convergencexi}
  V_N^{-1} \sum_{(n,x)\in \mathbb{H}_d}  \Big(  \big( \eta_{n,x} + \kappa_N^{(a)}\big) 
\ind_{\big\{ 1+\eta_{n,x} \geq a V_N \big\}} -  \big( \kappa_N^{(a)} + \bbE[\eta\ind_{\{  1+ \eta \leq V_N\}}] \ind_{\{\ga=1\}} \big)\Big) \, \gd_{(\frac{n}{N}, \frac{x}{\sqrt{N/d}})}
 \\
  = (1+o(1))  \bigg(  \sum_{(n,x)\in \mathbb{H}_d}  \eta_{n,x}^{(N)}
\ind_{\big\{ 1+\eta_{n,x} \geq a V_N  \big\}} \, \gd_{(\frac{n}{N}, \frac{x}{\sqrt{N/d}})}   - \kappa_a N^{-(1+\frac{d}{2})}  2 d^{d/2} \sum_{(n,x)\in \mathbb{H}_d}   \gd_{(\frac{n}{N}, \frac{x}{\sqrt{N/d}})} \bigg) \,,
\end{multline}
where the $o(1)$ is a quantity that goes to $0$ as $N\to\infty$
(and does not depend on the realization of $\eta$), see the calculations in Section~\ref{sec:uzeful}; recall that $\eta^{(N)} := V_N^{-1} \eta$.
Now, we observe that 
\[
w \mapsto \Big( \langle \sum_{(t,x,\ups)\in w} \ups\,  \gd_{(t,x)} , \psi \rangle \, , \,   \sum_{\sigma \in \cP(w)}\hat \beta^{|\sigma|} \varrho(\bt, \bx, f)\prod^{|\sigma|}_{i=1} u_i \Big)
\]
 is continuous
on $\cW$.
From the Poisson convergence~\eqref{convPoisson} in $\cW$,
using~\eqref{convergencexi} above and the definition~\ref{znlabar} of~$\bar Z_{N, \gb_N}^{\eta,a} (f)$
(together with the local limit theorem analogously to \eqref{convPoisson2}),
we deduce that
\begin{equation*}
\Big(  \langle \psi, \xi_{N,\eta}^{(a)} \rangle\, ,  \bar Z_{N, \gb_N}^{\eta,a} (f) \Big) 
 \stackrel{N\to \infty}{\Longrightarrow} \Big(  \langle \!\!\sum_{(t,x,\ups) \in \go^{(a)}} \!\! \ups\, \gd_{(t,x)} , \psi \rangle - \kappa_a \langle \cL , \psi \rangle  , \!\! \sum_{\sigma \in \cP(\go^{(a)})} \!\!\hat \beta^{|\sigma|} \varrho(\bt, \bx, f )\prod^{|\sigma|}_{i=1} \ups_i  \Big) \, .
\end{equation*}
Recalling the definitions~\eqref{def:xia} and~\eqref{rewritefun},
we see that the r.h.s.~in the display above is equal to $( \langle \xi_{\go}^{(a)} , \psi \rangle\, , e^{\hat \gb \kappa_a}\cZ^{\go,a}_{\hat \gb} (f))$,
which concludes the proof.
\end{proof}

\section{Proof of Proposition~\ref{letrucdur}: the easy cases}
\label{sec:letrucmoinsdur}

 Proposition~\ref{letrucdur} is the main technical difficulty of the paper. Its proof is considerably simpler in special cases $\alpha\in(0,1)$ (for any $d$) and $d=1$ (for any $\alpha$).
 These cases are treated in the present section.
 
\smallskip 
 When $\alpha\in(0,1)$, the convergence can be deduced from a first moment computation, after using the truncation argument from Proposition \ref{prop:tronquons}. The details are carried in Section~\ref{lefirst}.

\smallskip

When $\alpha\in [1,2)$, second moment computations are necessary. Since the variables $\eta$ themselves do not have a second moment, a truncation procedure is needed. A general result which describes the requirement we have for our truncated partition function is given in Section \ref{metameta}. 
Like for the proof of Theorem \ref{thm:contmeasure} in \cite{BL20_cont}, the truncation procedure that needs to be applied is considerably simpler 
for  $d=1$ than for $d\ge 2$.

When $d=1$, only the large values of $\eta$ are a problem so that, after using  Proposition~\ref{prop:tronquons}, we only need to perform a relatively simple second moment. This is done in Section~\ref{case1dim}.

When  $d\ge 2$ (and $\alpha\in [1,2)$) a simple truncation is not sufficient. Before applying the second moment method, the partition function must undergo a more advanced surgery. These details of the procedure and the computations are postponed to  Section \ref{sec:letrucdur}.

\subsection{The case $\alpha\in(0,1)$}\label{lefirst}
We assume without loss of generality that $0\le f\le 1$.
 Note that with our choice $\kappa^{(a)}_N=0$,  $Z^{\eta,a}_{N,\gb_N}(f)$ is a decreasing function of $a$.
 Thus we want to show that for $a\in (0,1]$ sufficiently small we have
\begin{equation}
 \sup_{N\ge 1}\bbE\Big[  \big(Z^{\eta,0}_{N,\gb_N}(f)- Z^{\eta,a}_{N,\gb_N}(f) \big)\wedge 1  \Big]\le \gep \, .
\end{equation}
For this we observe that the quantity we have to bound is smaller than
\begin{equation}
 \sup_{N\ge 1}\bbE\Big[ Z^{\eta,[0,b)}_{N,\gb_N}(f)- Z^{\eta,[a,b)}_{N,\gb_N}(f)  \Big]+ \sup_{N\ge 1}  \bbE\Big[ \big( Z^{\eta}_{N,\gb_N}  -Z^{\eta,[0,b)}_{N,\gb_N}\big) \wedge 1 \Big] \, .
\end{equation}
From Proposition~\ref{prop:tronquons}
the second  term can be made  smaller than~$\gep/2$ by choosing $b=b(\gep)$ large.
Concerning the first one, note that we have 
\begin{align*}
\bbE\Big[Z^{\eta,[0,b)}_{N,\gb_N}(f) -Z^{\eta,[a,b)}_{N,\gb_N}(f)  \Big]
&\le  \bE[ f(S^{(N)})]\left(\bbE\big[ 1+\beta_N \eta^{[0,b)}\big]^N - \bbE \big[1+\beta_N \eta^{[a,b)}  \big]^N\right).
\end{align*}
Now, thanks to  \eqref{moments2},
there is a  constant $C_b$ such that for every $n$ and $a$
\begin{equation}
 \bbE \big[\beta_N \eta^{[a,b)}\big]  \leq \bbE \big[\beta_N \eta^{[0,b)}\big] \le C_b N^{-1} \,.
 \end{equation}
 Using \eqref{moments2} monotonicity in $a$ and continuity at $a=0$, for any $\gep>0$ we can find $a(\gep)$ sufficiently small such that for every $N\ge 1$ we have
\begin{equation}
 \bbE \Big[\beta_N \big(\eta^{[0,b)}-\eta^{[a,b)} \big)\Big] \le   \frac{\gep}{4}e^{-C_b}N^{-1}.
\end{equation}
Hence we have for this choice of $a$, for every $N$ 
\begin{equation}
 \bbE\Big[Z^{\eta,[0,b)}_{N,\gb_N}(f) -Z^{\eta,[a,b)}_{N,\gb_N}(f)  \Big]\le e^{C_b}(e^{ \frac{\gep}{4}e^{-C_b}N^{-1}}-1)\le \gep/2 \, .
\end{equation}
\qed

\subsection{The case $\ga\in[1,2)$: a uniformity criterion}\label{metameta}
The task is more delicate in the case $\ga\in[1,2)$.  We are going to prove some uniform  $L^2(\bbP)$ convergence. The following statement, that we are going to apply to our partition function, may help to understand this difference.
 \begin{proposition}\label{metaprop2}
 Consider $(X_{N,a})_{a\in [0,1),N\ge 1}$ a collection of positive random variables.
 Assume that there exists  
 $X^{(q)}_{N,a}$ a sequence of approximations of $X_{N,a}$, indexed by $q\ge 1$, which satisfies
\begin{align*}
\mathrm{(A)} \qquad  &  \lim_{q\to \infty} \sup_{N\ge 1} \sup_{a\in[0,1)} \bbE\big[ |X^{(q)}_{N,a}-X_{N,a}|\big]=0\,  ; \\
 \mathrm{(B)}  \qquad  & \lim_{a\to 0+} \sup_{N\ge 1}   \bbE\big[ (X^{(q)}_{N,a}- X^{(q)}_{N,0})^2\big]=0 \quad \text{ for every $q\geq 1$}.
\end{align*}
Then we have
\begin{equation}
 \lim_{a\to 0}\sup_{N\ge 1} \bbE\big[ |X_{N,a}-X_{N,0}|\big]=0.
\end{equation}
If we replace $\mathrm{(A)} $ by 
\begin{align*}
\mathrm{(A')} \qquad   \lim_{q\to \infty} \sup_{N\ge 1} \sup_{a\in[0,1)} \bbE\big[ |X^{(q)}_{N,a}-X_{N,a}|\wedge 1\big]=0\,  
\end{align*}
then we have
\begin{equation}
  \lim_{a\to 0}\sup_{N\ge 1} \bbE\big[ |X_{N,a}-X_{N,0}|\wedge 1\big]=0.
\end{equation}

 \end{proposition}
 
\noindent The assumption $(A')$ allows to treat the case $\alpha=1$ for which the  $L^1(\bbP)$ convergence of the partition function does not hold.
 
 \begin{proof}
  This simply comes from the fact that 
  \begin{equation*}
   \bbE\big[ |X_{N,a}-X_{N,0}|\big]\le   \bbE\big[ (X^{(q)}_{N,a}- X^{(q)}_{N,0})^2\big]^{1/2}+ \bbE\big[ |X^{(q)}_{N,a}-X_{N,a}|\big]+\bbE\big[ |X^{(q)}_{N,0}-X_{N,0}|\big],
  \end{equation*}
and the right-hand side can be made arbitrarily small uniformly in $N$, by taking first $q$ large and then taking $a\to 0$.
A similar reasoning holds for $(A')$.
 \end{proof}
 

\subsection{The case of dimension $d=1$}\label{case1dim}
For notational simplicity we will write the proof for the case $f\equiv 1$, 
the modifications to treat the case $f\in \cC$ are straightforward.
Let us set 
\begin{equation}
 X^{(q)}_{N,a}:= e^{-\hat \beta \gamma_N \ind_{\{\alpha=1\}}} Z^{\eta,[a,q)}_{N,\gb_N}\, .
\end{equation}
We now only need to check that the assumptions of Proposition \ref{metaprop2} are satisfied.
When $\alpha>1$, since $ Z^{\eta,[a,q)}_{N,\gb_N}\le Z^{\eta,a}_{N,\hat \beta}$ and $\bbE[Z^{\eta,a}_{N,\hat \beta} ]=1$,  Assumption $(A)$ in Proposition~\ref{metaprop2} is equivalent to the uniform convergence to $1$ of the first moment. When $\alpha=1$, Assumption $(A')$ has been already  checked in Proposition~\ref{prop:tronquons}.

\begin{lemma}
\label{lem:Wd=1}
In dimension $d=1$ (note that $\ga_c=2$ in that case), we have
\begin{align*}
\mathrm{(} \tilde{\mathrm{A}} \mathrm{)} \qquad  &  \lim_{q\to \infty} \inf_{N\ge 1} \inf_{a\in[0,1)}  \bbE \left[ Z^{\eta,[a,q)}_{N,\gb_N}\right]=1\, \quad \text{ for } \alpha \in (1,2) ; \\
 \mathrm{(B)}  \qquad  & \lim_{a\downarrow 0} \sup_{N\ge 1}   \bbE\big[  e^{-2\hat \beta \gamma_N \ind_{\{\alpha=1\}}}  ( Z^{\eta,[a,q)}_{N,\gb_N}-  Z^{\eta,[a,q)}_{N,\gb_N})^2\big]=0 \quad \text{ for every $q\geq 1$ and $\alpha \in [1,2)$}.
\end{align*}
\end{lemma}

\begin{proof}
The proof of $(\tilde A)$ is straightforward.
We have 
\begin{equation}\label{lari}
  \bbE \left[ Z^{\eta,[a,q)}_{N,\beta_N}\right]= \bbE \big[ 1+ \beta_N \eta^{[a,q)} \big]^N= \left( 1+\beta_N \bbE \big[  \eta \ind_{\{1+\eta <   q V_N\}} \big]  \right)^N.
\end{equation}
From \eqref{moments2}  and monotonicity in $q$, given $\gep>0$ there exists $q_0(\gep)$ such that for all $q\ge q_0$
for every $N\ge 1$  
\begin{equation}\label{laright}
\beta_N \bbE \big[  \eta \ind_{\{1+\eta <   q V_N\}} \big]  \geq  
 - \gep N^{-1} \,.
\end{equation}
Using that $(1-x)^N \geq 1- N x$ for all $x\geq 0$, we get from \eqref{lari} that for all $N\geq 1$
\begin{equation}
 1\geq \inf_{a\in [0,1)} \bbE \left[ Z^{\eta, [a,q)}_{N}\right] \geq
  1-\gep\, ,
\end{equation}
which concludes the proof of item $(\tilde A)$.

\smallskip

For the second moment estimate $(B)$, let us notice that $Z^{\eta,[a,q)}_{N,\gb_N}$ is a (time-reversed) martingale for the filtration $\mathcal G_{a}$ (recall \eqref{lafiltr})
in particular we have 
\begin{equation}
 \bbE \big[ (Z^{\eta,[0,q)}_{N,\gb_N}-Z^{\eta,[a,q)}_{N,\gb_N})^2 \big]=   \bbE \big[ (Z^{\eta,[0,q)}_{N,\gb_N})^2 \big]-\bbE \big[ (Z^{\eta,[a,q)}_{N,\gb_N})^2\big].
\end{equation}
Hence to show that the convergence in $L^2(\bbP)$ is uniform in $N$,
it is sufficient to show that the convergence
 $$\lim_{a\downarrow 0}  e^{-2\hat \beta \gamma_N \ind_{\{\alpha=1\}}} \bbE \big[ (Z^{\eta,[a,q)}_{N,\gb_N})^2\big] = 
  e^{-2\hat \beta \gamma_N \ind_{\{\alpha=1\}}}  \bbE \big[ (Z^{\eta,[0,q)}_{N,\gb_N})^2  \big]$$ is uniform in $N$.
Note that $ e^{-\hat \beta \gamma_N \ind_{\{\alpha=1\}}} \bbE \big[ Z^{\eta,[a,q)}_{N,\gb_N}\big]$ does not depend on $a$. In the case $\ga>1$, thanks to~\eqref{lari} and~\eqref{laright}, we find that 
 it is bounded away from $0$ uniformly in $q\geq 0$ and $N\geq 1$.
In the case $\ga=1$, a straightforward calculation (recall~\eqref{uzeful2}) gives that \eqref{laright} is replaced with
\begin{equation}\label{lawrong}
\beta_N \bbE \big[  \eta \ind_{\{1+\eta <   q V_N\}} \big] 
= \hat \gb  \gamma_N N^{-1} + \hat \gb (\log q) N^{-1} (1+o(1)) \, ,
\end{equation}
so we get that 
$e^{-\hat \gb \gamma_N  \ind_{\{\alpha=1\}}} \bbE \big[ Z^{\eta,[a,q)}_{N,\gb_N}\big]$
 is  bounded away from $0$ uniformly in  $q\geq 1$ and $N\geq 1$.
All together, what we need to show is equivalent to
\begin{equation}
\label{eq:limasupN}
\lim_{a\downarrow 0} \sup_{N\ge 1}    \bigg( 
\frac{\bbE \big[ (Z^{\eta,[0,q)}_{N,\gb_N})^2 \big]}{ \bbE \big[Z^{\eta,[0,q)}_{N,\gb_N}\big]^2}-\frac{\bbE \big[ (Z^{\eta,[a,q)}_{N,\gb_N})^2 \big]}{ \bbE \big[ Z^{\eta,[a,q)}_{N,\gb_N}\big]^2} \bigg)=0.
\end{equation}
Let us set 
\begin{equation}
 r_{N}^{a,q}:=  \frac{\bbE \left[ (1+ \beta_N \eta^{[a,q)})^2 \right]}{\bbE \left[ 1+ \beta_N \eta^{[a,q)} \right]^2}-1 = \frac{\beta^2_N\Var(\eta^{[a,q)})}{\bbE \left[ 1+ \beta_N \eta^{[a,q)} \right]^2} \,,
\end{equation}
and let us stress that $r_N^{a,q}$ is non-increasing in $a$.
A direct computation yields
\begin{equation}
\label{eq:secondmoment}
 \frac{\bbE \big[ (Z^{\eta,[a,q)}_{N,\beta_N})^2 \big]}{ \bbE \big[ Z^{\eta,[a,q)}_{N,\beta_N}\big]^2}= \bE^{\otimes 2}\big[ (1+  r_{N}^{a,q})^{L_N} \big] \,,
\end{equation}
where $\bE^{\otimes 2}$ is the expectation with respect to  two independent walks $S^{(1)}$ and $S^{(2)}$ and $L_N:=\sum_{n=1}^N \ind_{\{S^{(1)}_N= S^{(2)}_N\}}$ is the replica overlap.
By convexity of $x \mapsto (1+x)^{L_n}$, we therefore get
\begin{equation}\begin{split}
\frac{\bbE \big[ (Z^{\eta,[0,q)}_{N,\beta_N})^2 \big]}{ \bbE \big[Z^{\eta,[0,q)}_{N,\beta_N}\big]^2}-\frac{\bbE \big[ (Z^{\eta,[a,q)}_{N,\beta_N})^2 \big]}{ \bbE \big[ Z^{\eta,[a,q)}_{N,\beta_N}\big]^2}&\le ( r_{N}^{0,q}- r_{N}^{a,q})\bE^{\otimes 2}\big[ L_N (1+  r_{N}^{0,q})^{L_N} \big]\\
& \le  \frac{r_{N}^{0,q}- r_{N}^{a,q}}{r_{N}^{0,q}}\bE^{\otimes 2}\big[(r_{N}^{0,q} L_N) e^{r_{N}^{0,q}L_N} \big].
\end{split}
\end{equation}
Now in order to conclude, it is sufficient to show that 
\begin{equation}\label{stimrn}
   \lim_{a\to 0} \sup_{N\ge 1}\frac{ r_{N}^{0,q}- r_{N}^{a,q}}{r^{0,q}_N}=0 \quad \text{ and } \quad r_{N}^{0,q}\le C_q N^{-1/2}.
\end{equation}
Indeed using the second statement in \eqref{stimrn}, we get that
\begin{equation}
\bE^{\otimes 2}\big[(r_{N}^{0,q} L_N) e^{r_{N}^{0,q}L_N} \big]\le 
\bE^{\otimes 2}\big[e^{2r_{N}^{0,q}L_N} \big] \le \bE^{\otimes 2}\big[   e^{2C_q N^{-1/2} L_N} \big] \,  ,
\end{equation}
which is uniformly bounded in $N$, as it is standard for the intersection time of independent random walks, see e.g.\ \cite[Lemma~4.2]{Soh09} for a general version  (with renewal processes). 

\medskip

\noindent Let us now prove the two estimates in \eqref{stimrn}.
We have
\begin{equation}
\label{eq:etasquare}
  \beta^2_N \bbE\big[(\eta^{[a,q)})^2 \big]  =  (\gb_N\kappa_N^{(a)} )^2 \bbP\big( 1+\eta < a V_N \big) +  \gb_N^2 \bbE \Big[\eta^2 \ind_{ \{ 1+\eta \in [aV_N,qV_N) \}}  \Big].
\end{equation}
Using \eqref{def:eta} together with  \eqref{moments3},
 after simplifications (in particular, the first term is negligible), we find  that  for any fixed  $a\in [0,1)$ and $q>1$, we have asymptotically for large~$N$
\begin{equation}\label{lonzo}
 \beta^2_N \bbE\big[(\eta^{[a,q)})^2 \big] =    \frac{ \alpha \hat \beta^2 d^{-d/2}}{2(2-\alpha)}  N^{\frac{d-2}{2}}  (  q^{2-\alpha}-  a^{2-\alpha}+o(1)) \,.
\end{equation}
On the other hand  \eqref{laright} and \eqref{lawrong} ensures that $\bbE[\eta^{[a,q)} ]^2$ is always negligible w.r.t.\ $\bbE[(\eta^{[a,q)})^2]$ which is thus asymptotically equivalent to the variance.
 Note that we have $\bbE[\eta^{[a,q)}]= \bbE[\eta^{[0,q)}] = \bbE[\eta \ind_{\{1+\eta < qV_N\}}]$ for any $a\in(0,1)$ and any $q\geq 1$,
  and from \eqref{lawrong} we get that 
  $\bbE[1+\beta_N \eta^{[0,q)}]$ tends to one, so that $r^{a,q}_N\sim   \beta^2_N \bbE\big[(\eta^{[a,q)})^2 \big]$.
We therefore obtain that
 for any fixed $a\in [0,1)$ and $q>1$ (recall \eqref{uzefulrel})
\begin{equation}
r^{a,q}_{N} \stackrel{N\to \infty}{\sim} \frac{ \alpha \hat \beta^2 d^{-d/2}}{2(2-\alpha)}  N^{\frac{d-2}{2}} (  q^{2-\alpha}-  a^{2-\alpha}).
\end{equation}
This, together with the monotonicity in $a$ of $r^{a,q}_{N}$, allows us to deduce both statements in~\eqref{stimrn}
and
therefore concludes the proof of part (B) of Lemma~\ref{lem:Wd=1}.
\end{proof}

To conclude this section let us stress that this strategy cannot be applied in dimension $d\ge 2$ because we still  have in that case $r_{N}^{a,q}\asymp N^{\frac{d-2}{2}}$, and the second moment $\bbE [ (Z^{\eta,[a,q)}_{N,\beta_N})^2 ]$ diverges with $N$ for any value of $a$ and $q$ (recall~\eqref{eq:secondmoment}). We need a finer restriction on the set of trajectories, analogously to what is done in \cite[Section 4.4]{BL20_cont}.

 \section{Proof of Proposition~\ref{letrucdur}: the hard case}
\label{sec:letrucdur}

In this section, we prove Proposition \ref{letrucdur} when $\alpha\in[1,\alpha_c)$ and $d\ge 2$.
We rely again on Proposition \ref{metaprop2}, but with a more sophisticated truncation procedure detailed in Section~\ref{sec:trunc}.
We consider for notational simplicity only the case $f \equiv 1$.
The modifications which are required to treat the general case $f \in \cC$ are provided at the end of the section (in Section \ref{modifgenf}).

 \subsection{The truncation procedure}
 \label{sec:trunc}
  
 Instead of simply capping the value of $\eta$ at level~$q$, we impose a restriction on the set of paths, to avoid counting atypical paths with high value of $\eta$ which give an important contribution to the second moment of $Z_{N,\beta_N}^{\eta, a}$. 
Let $\ga \in [1,\ga_c)$ and  let us fix~$\gamma$ satisfying 
 \begin{equation}
 \label{lesdeuxconds}
\frac{d-2}{2(2-\alpha)} < \gamma < \frac{1}{\alpha-1} \qquad \big(i.e.\ \ \gamma(\alpha-1)<1 \ \text{ and } \  \frac d2 -\gamma(2-\alpha) <1 \big) \,,
\end{equation}  and
 also $\gamma<\frac d2$, which is compatible with \eqref{lesdeuxconds} when $\ga<\ga_c =1+\frac{2}{d}$. We define
 \begin{equation}
 \label{def:BNq}
B_{N,q}(S) := \Big\{ \forall I \subset \lint 1, N\rint \ : \ 
  \prod_{i\in I} (1+\eta_{i, S_{i}}) < q^{|I|}  V_N^{|I| } \big(N^{-|I|} \Pi_I\big)^{\gamma} \Big\},
 \end{equation}
where $\Pi_I := \prod_{j=1}^{|I|} (i_{j}-i_{j-1})$, with $i_1<\dots<i_{|I|}$ the ordered elements of $I$, and $i_0=0$.
Notice in particular that on the event $B_{N,q}(S)$ we have $1+\eta_{i,S_i}< q V_N$ for all $1\leq i \leq N$.
We then set
\begin{equation}
\label{def:WNaq}
 W^{a,q}_N:=  \bE\bigg[\prod_{n=1}^N \big( 1+\beta_N \eta^{(a)}_{n,S_n} \big)\ind_{B_{N,q}(S) } \bigg] = \bE\bigg[\prod_{n=1}^N \big( 1+\beta_N \eta^{[a,q)}_{n,S_n} \big)\ind_{B_{N,q}(S) } \bigg] .
\end{equation}
We apply Proposition \ref{metaprop2} to  $W^{a,q}_N$ which reduces the proof of Proposition \ref{letrucdur} to
that of the following statements.

\begin{lemma}\label{firstmt}
 When $\alpha\in (1,\alpha_c)$ we have
 \begin{equation}
  \lim_{q\to \infty}\inf_{a\in [0,1)} \inf_{N\ge 1} \bbE\left[ W^{a,q}_N\right]=1 \,.
 \end{equation}
 \end{lemma}

 \begin{lemma}\label{firstmtbis}
When $\alpha=1$ we have
 \begin{equation}
  \lim_{q\to \infty}\sup_{a\in [0,1)} \sup_{N\ge 1} \bbE\left[ \left( e^{-\hat \gb \gamma_N} \big(Z^{\eta,a}_{N,\beta_N}-W^{a,q}_N\big) \right)\wedge 1\right]=0.
 \end{equation}
 \end{lemma}

\begin{proposition}\label{secondmt}
 We have for every $q>0$
\begin{equation}
\lim_{a\to 0}\sup_{N\geq 0} \bbE\big[ (W^{a,q}_N- W^{0,q}_N)^2  \big]=0.
\end{equation}
\end{proposition}

These three results are similar to Lemma~\ref{lem:Wd=1} (and Proposition~\ref{prop:tronquons} in the case $\ga=1$), but the required computations to prove them are considerably more involved.
We prove each one separately in Sections~\ref{sec:fm}, \ref{sec:fm1bis} and \ref{sec:sm} respectively.

 \subsection{Proof of Lemma \ref{firstmt}}\label{sec:fm}
 
 Let $\tilde \bbP^{a}_S$ be the probability measure 
 whose density with respect to $\bbP$ is
 \begin{equation}\label{lamesuretilde}
 \frac{ \dd \tilde \bbP^{a}_S}{\dd \bbP} = \prod_{n=1}^N \big(1+\beta_N \eta^{(a)}_{n,S_n} \big) \, .
 \end{equation}
We stress here that since $\bbE[\eta]=0$ we have $\bbE\big[ \prod_{n=1}^N (1+\beta_N \eta^{(a)}_{n,S_n}) \big]=1$.
By Fubini, we have  $\bbE[ W^{a,q}_N] =\bE\big[ \tilde  \bbE_S^{a} [\ind_{B_{N,q}(S)}] \big]$ and hence 
we simply need to prove that
\begin{equation}
\label{cestuniforme}
\lim_{q \to \infty} \sup_{a\in[0,1)} \sup_{N\geq 1}  \sup_{S}\tilde \bbP^{a}_S\big( B^{\complement}_{N,q}(S) \big)=0.
\end{equation}
Let us make  two important observations: 
\begin{itemize}
\item[(i)] In view of the definition~\eqref{def:BNq} of the event $B_{N,q}$, the above probability does not depend on the specific trajectory~$S$:
. Indeed  $B_{N,q}(S)$ is a function of the random sequence $(\eta_{n,S_n})_{n\geq 1}$ which is i.i.d.\ distributed, and in particular has the same distribution for every $S$.
 \item[(ii)] The event $B_{N,q}^{\complement}(S)$ is increasing in $\eta$, and 
 the measures $(\tilde \bbP^{a}_S)_{a\geq 0}$ are stochastically decreasing. The supremum in $a$ is thus attained for $a=0$.
\end{itemize}
To check point (ii), since we are dealing with product measures it is sufficient to check the domination for one dimensional marginal.
Recall the continuous version of Chebychev's sum inequality: for any probability measure $\nu$ on $\bbR$ and $f$ and $g$ non-decreasing functions, we have
\begin{equation}\label{chebcheb}
 \int f(x)g(x)  \nu(\dd x) \ge  \int f(x)  \nu(\dd x) \int g(x) \nu(\dd x).
\end{equation}
This implies that for any non-decreasing $f$ we have
\begin{multline}\label{chebcheb2}
 \bbE\big[(1+\beta_N\eta)f(\eta)\big]-\bbE\big[(1+\beta_N\eta^{(a)} )   f(\eta) \big]\\
 = \beta_N \bbP( 1+\eta <a V_N  )  \, \bbE\big[ f(\eta) \big(\eta- \bbE [ \eta \, | \, 1+\eta<a V_N ] \big)    \, \big| \,  1+\eta <  a V_N  \big]\ge 0 \,,
\end{multline}
which proves that $\tilde \bbP_S^{(0)}$ stochastically dominates $\tilde \bbP_S^{(a)}$ for $a \in(0,1)$.
We therefore only have to prove that 
\begin{equation}
\label{newgoal}
\lim_{q\to\infty} \sup_{N\geq 1} \tilde \bbP_N (B^{\complement}_{N,q} ) =0 \, ,
\end{equation}
where 
  \begin{equation}
  \label{def:BNqbis}
 B_{N,q}:= \Big\{ \forall I \subset \lint 1, N\rint, \  \ 
  \prod_{i\in I}  (1+\eta_i) < q^{|I|}  V_N^{|I|} \big(N^{-|I|} \Pi_I\big)^{\gamma}  \Big\},
 \end{equation}
with $(\eta_i)_{1\le i \le N}$ i.i.d.\ random variables distributed as $\eta$ under $\bbP$ and $\tilde \bbP_N$ the  probability measure with density $\prod_{n=1}^N (1+\beta_N \eta_n)$ with respect to $\bbP$.
First of all, by a union bound, we have
\begin{equation}
\label{newgoal2}
\tilde \bbP_N (B^{\complement}_{N,q} )  \leq 
\sum_{n=1}^N 
\tilde \bbP_N \big( 1+\eta_n \geq  q V_N \big)  +
 \tilde \bbP_N\big(  B_{N,q}^{\cc}  \, ;  \, \forall\, i \in \lint 1,N \rint ,  \, 1+\eta_i < q V_N  \big) \, .
\end{equation}
Now, thanks to~\eqref{moments2}  and monotonicity in $q$, given $\gep>0$ there exists $q_0(\gep)$ such that for all $q\ge q_0$
for every $N\ge 1$  
\begin{equation}
\label{newgoal3}
\tilde \bbP_N \big( 1+\eta_n \geq  q V_N \big)= \bbE\big[ (1+\gb_N \eta) \ind_{\{1+\eta \geq  q V_N \}} \big] \le (\gep/2) N^{-1} \,.
\end{equation}
Hence the first sum 
in \eqref{newgoal2} is bounded by  $\gep/2$ for $q\ge q_0(\gep)$.
To estimate the remaining probability in~\eqref{newgoal2}, we are going to perform another union bound. The following claim will allow us to reduce the amount of error produced by  this bound.
\begin{claim}
\label{claiclaim}
If $q\ge 2^{\gamma}$ and $B_{N,q}$ is \emph{not} satisfied,
then there exists some non-empty set of indices $I \subset \lint 1,N\rint$
such that both the following conditions are satisfied
\begin{equation}\label{lescondits}
\prod_{i\in I} (1+\eta_i)  \geq q^{|I|}  V_N^{|I|} \big(N^{-|I|} \Pi_I \big)^{\gamma}  \quad \text{ and }\quad  \forall i \in I\,, \, 1+\eta_i \geq N^{-\gamma} V_N \, .
\end{equation}
\end{claim}
\begin{proof}
Note that the existence of a set of indices $I_0$ satisfying the first condition in \eqref{lescondits} simply comes from the definition of $B_{N,q}$. 

Now, if a set~$I$ satisfies the first condition in~\eqref{lescondits} and if there exists some $j_0$ with 
 $1+\eta_{i_{j_0}} < N^{-\gamma} V_N$ (where
 $i_{j_0}$ is the $j_0$-th element of $I$),  we necessarily have 
 that $I$ is not reduced to~$i_{j_0}$, and
\begin{align}\label{lacompa2}
\prod_{i \in I' = I\setminus \{i_{j_0}\}} (1+\eta_i) \geq \Big(q^{|I|-1} V_N^{|I|-1} \big(N^{1-|I|}\Pi_{I'} \big)^{\gamma} \Big)
 \times q \,  \bigg( \frac{\Pi_I}{\Pi_{{I'}}}\bigg)^{\gamma}\, .
\end{align}
Recalling the definition of $\Pi_I$, we have that $\frac{\Pi_I}{\Pi_{{I'}}} = \frac{(i_{j_0}-i_{j_0-1}) (i_{j_0+1}- i_{j_0})}{(i_{j_0+1}- i_{j_0-1})} \geq \frac12\, $:
the second factor in \eqref{lacompa2} is thus larger than $q 2^{-\gamma} \geq 1$; the same reasoning applies if $j_0=|I|$.
Hence, the set $I'=I\setminus\{i_{j_0}\}$ is non-empty and also satisfies the first condition  in~\eqref{lescondits}. 
Starting from $I_0$ and proceeding by induction, we therefore
end up with a non-empty set verifying both conditions in~\eqref{lescondits}.
\end{proof}

Thanks to Claim~\ref{claiclaim},  we apply a union bound over the possible choices for the set of indices~$I$ satisfying both conditions in \eqref{lescondits}. Recalling also the definition~\eqref{def:BNqbis} of $B_{N,q}$, we obtain that the last term in~\eqref{newgoal2}
is bounded by
\begin{equation}
\label{decomp2}
\sum_{I \subset  \lint 1,N\rint } \tilde \bbP_N\Big(  \prod_{i\in I} (1+\eta_i) \ind_{\{  N^{-\gamma} V_N \leq  1+\eta_i  < qV_N\}} \geq q^{|I|}  V_N^{|I|} \big(N^{-|I|} \Pi_I \big)^{\gamma}  \Big) \, .
\end{equation}
To conclude, we estimate the probabilities in the sum thanks to the following lemma, whose proof is postponed.

\begin{lemma}
\label{lekalkul}
There is some $N_0\geq 1$ such that for all $N\geq N_0$,
for any $k\ge 1$, any $t \in (0,1)$  and any $\gep\in (0,1)$ there is a constant $C$ (allowed to depend on $\varphi$, $\hat \beta$ and~$\gep$) such that
\begin{equation}\label{linek}
 \tilde \bbP_N\left( \prod_{i=1}^k (1+\eta_i)  \ind_{\{  N^{-\gamma} V_N \leq  1+\eta_i  < qV_N\}} \ge t  (q V_N)^k \right) 
 \le   ( C q^{\gep +1-\ga})^k   \,  t^{1-\alpha-\gep}  N^{-k} \, .
\end{equation}
\end{lemma}

Now, applying Lemma \ref{lekalkul} to the probabilities in \eqref{decomp2} with
\[t= \big(N^{-|I|} \Pi_I \big)^{\gamma} = \prod_{j=1}^{|I|} \left(\frac{i_j-i_{j-1}}{N} \right)^{\gamma}\,,
\]
we obtain that \eqref{decomp2} is smaller than
\begin{multline}
\label{decomp3}
 \sum_{k=1}^N   (C q^{\gep+ 1-\alpha})^{k} \sum_{ 1\leq i_1< \cdots < i_k \leq N}  N^{-k}  \prod_{j=1}^{k} \left(\frac{i_j-i_{j-1}}{N} \right)^{(1-\alpha-\gep)\gamma} \\
 \leq  \sum_{k=1}^N   (C q^{\gep+ 1-\alpha})^{k}   2^{k(\alpha+\gep-1)} \int_{0<s_1<\dots<s_k<1} \prod_{i=1}^{k} (s_i-s_{i-1})^{(1-\alpha-\gep)\gamma} \dd s_i  \, ,
\end{multline}
where we used  a standard comparison argument for the last inequality.
By Lemma~\ref{lemGamma},
provided that $\gep$ has been fixed small enough so that $\vartheta:=\gamma (\alpha+\gep-1) <1$, the last integral is
 equal to $\frac{\Gamma(1-\vartheta)^k}{ \gG( 1+ k(1-\vartheta))}$.
Altogether, we obtain that~\eqref{decomp2}
is bounded by
\begin{equation*}
  \sum_{k=1}^N   \big(  C'_{\ga,\hat\gb,\gep} q^{\gep +1-\alpha} \big)^{k}  \frac{\Gamma(1-\vartheta)^k}{ \gG( k(1-\vartheta)+1)}\, .
\end{equation*}
This series converges, and can be made arbitrarily small by choosing $q$ large, provided that~$\gep$ is small enough so that $\gep+1-\ga<0$.
Together with \eqref{newgoal2} and \eqref{newgoal3}, this concludes the proof of~\eqref{newgoal} and hence the proof of Lemma~\ref{firstmt}.
\qed

\begin{proof}[Proof of Lemma~\ref{lekalkul}]
First of all, recalling that $\beta_N V_N N^{-\gamma} = \frac{1}{2}d^{-d/2} \hat \gb  N^{\frac{d}{2} -\gamma}$ (see~\eqref{uzefulrel})
and the fact that we chose $\gamma<\frac{d}{2}$, we have  that
 $\gb_N \eta_i \geq 1$ if $ 1+\eta_i  \geq N^{-\gamma} V_N$, at least for large~$N$. Hence,
recalling the definition of $\tilde \bbP_N$, the probability we want to bound is
 \begin{multline}
 \label{insidelekalkul}
  \bbE\Big[  \Big(  \prod_{i=1}^k(1+\beta_N \eta_i) \ind_{\{    (1+\eta_i) \in [N^{-\gamma}V_N , qV_N )\}}  \Big) \ind_{\big\{ \prod_{i=1}^k  (1+\eta_i) \geq t (qV_N)^k  \big\}} \Big]  \\
 \leq (2\gb_N)^k  \bbE\bigg[ \Big( \prod_{i=1}^k ( 1+\eta_i )  \ind_{\{   (1+\eta_i) \in [N^{-\gamma}V_N , qV_N )\}} \Big)   \ind_{\big\{ \prod_{i=1}^k (1+ \eta_i) \geq t  (qV_N)^k  \big\}}   \bigg] \, .
 \end{multline}
Then, Proposition~\ref{lescroissants} in the appendix allows us to compare this expectation to an integral with respect to the measure $ u^{-(1+\ga)}  \gp(u) \dd u$ (which is not a probability measure). Applying Proposition~\ref{lescroissants}, we get that there is a constant $C$ (that depends on $\ga$) such that  for $N$ large enough the right-hand side of \eqref{insidelekalkul} is bounded by
\begin{equation}
\label{apreslekalkul}
C^k \beta^k_N \int_{[0,2q V_N)^k } 
   \ind_{\big\{ \prod_{i=1}^k  u_i  \ge t   (qV_N)^k \big\} }  \prod_{i=1}^k  \varphi(u_i) u^{-\alpha}_i \dd u_i  \, .
\end{equation}
With a  change of variable $u_i= V_N v_i$, and using that $\gb_N V_N^{1-\alpha}  = \hat \gb   N^{-1} \varphi(V_N)^{-1}$,
we get that this is bounded by
\begin{multline}
 (C_{\alpha,\hat \beta})^k  N^{-k} 
   \int_{[0,2q)^k }
   \ind_{ \big\{  \prod_{i=1}^k  v_i  \ge  t q^k \big\} } \prod_{i=1}^k \frac{ \varphi(V_N v_i )}{\varphi(V_N)}  v^{-\alpha}_i \dd v_i \\
   \le \big( C_{\alpha,\hat \beta,\gep} q^{\gep} \big)^k N^{-k}   \int_{[0,2q)^k }
   \ind_{ \big\{  \prod_{i=1}^k  v_i  \ge  t q^k  \big\} } \prod_{i=1}^k  v^{-\alpha-\frac{\gep}{2}}_i \dd v_i ,
   \label{compareintegral}
\end{multline}
where we used Potter's bound to get that 
$\gp(V_N v) \leq C_{\gep} v^{-\gep/2} \gp(V_N)$ if $v\leq 1$ and that
$\gp(V_N v) \leq C_{\gep}  v^{\gep/2} \gp(V_N)$ if $v\in [1,2q)$, with $v^{\gep/2} \leq (2q)^{\gep} v^{-\gep/2}$ for $v\in [1,2q)$.
An estimate for the integral in the r.h.s.~of~\eqref{compareintegral} has been proved in \cite{BL20_cont}, we 
apply  \cite[Equation $(4.32)$]{BL20_cont} to conclude the proof.
\end{proof}

 \subsection{Proof of Lemma \ref{firstmtbis}}\label{sec:fm1bis}

Let us consider $b>1$. 
We have 
\begin{multline}
 Z^{\eta,a}_{N,\beta_N}-W^{a,q}_N =\bE\Big[ \prod_{n=1}^N \big( 1+\beta_N\eta^{(a)}_{n,S_n} \big) \ind_{B_{N,q}^{\cc}(S)}\Big] \\
 \le  \big(Z^{\eta,a}_{N,\beta_N}-  Z^{\eta,[a,b)}_{N,\beta_N}\big)+
 \bE\Big[ \prod_{n=1}^N \big( 1+\beta_N\eta^{(a)}_{n,S_n} \ind_{\{ (1+\eta_{n,S_n})< b V_N\}} \big) \ind_{B_{N,q}^{\cc}(S)}\Big].
\end{multline}
Proposition~\ref{prop:tronquons} establishes that 
$e^{-\hat \beta \gamma_N}(Z^{\eta,a}_{N,\beta_N}-  Z^{\eta,[a,b)}_{N,\beta_N})$ converges to zero in probability when $b\to \infty$.
To conclude we thus only have to show that the second term also converges to zero in probability, or using first moment estimates, that for every $b>1$,
\begin{equation}
 \lim_{q\to \infty} \sup_{N\ge 1} e^{-\hat \beta \gamma_N} \bbE\bigg[ \bE\Big[ \prod_{n=1}^N \big( 1+\beta_N\eta^{(a)}_{n,S_n} \ind_{\{ (1+\eta_{n,S_n}< b V_N\}} \big) \ind_{B_{N,q}^{\cc}(S)}\Big]\bigg] = 0.
\end{equation}
Like for the proof of Lemma \ref{firstmt} (note that \eqref{chebcheb2} remains valid when adding the restriction $\ind_{\{ \eta < b V_N \} }$), we consider the probability measure $\tilde \bbP_N^{(b)}$ defined by
  \begin{equation}\label{lamesuretilde2}
 \frac{ \dd \tilde \bbP^{(b)}_N}{\dd \bbP} = \frac{\prod_{n=1}^N \big(1+\beta_N \eta^{[0,b)}_{n} \big)}{\bbE\big[ 1+\beta_N \eta^{[0,b)}_{n} \big]^N } \, .
 \end{equation}
and we obtain that  (recall the definition \eqref{def:BNqbis})
\begin{multline}
 e^{-\hat \beta \gamma_N} \bbE\bigg[ \bE\Big[ \prod_{n=1}^N \big( 1+\beta_N\eta^{(a)}_{n,S_n} \ind_{\{ (1+\eta_{n,S_n})\le b V_N\}} \big) \ind_{B_{N,q}^{\cc}(S)}\Big]\bigg] \\
 \le  e^{-\hat \beta \gamma_N}  \bbE\big[ 1+\beta_N \eta^{[0,b)}_{n} \big]^N  \tilde \bbP^{(b)}_N[B_{N,q}^{\cc} ] .
\end{multline}
From \eqref{lawrong}, there exists  a constant $C_{b,\hat \gb}$ such that for every $N\ge 1$  
\[
e^{-\hat \beta \gamma_N}\bbE\big[ 1+\beta_N \eta^{[0,b)}_{n} \big]^{N} \le C_{b,\hat \gb}
\] 
so that we only need to bound $\tilde \bbP^{(b)}_N[B_{N,q}^{\cc} ] $. Using Claim \ref{claiclaim} and a union bound, we can conclude the proof as in the previous section, replacing Lemma~\ref{lekalkul}
with the following: there exists a constant $C_b$ such that the following holds for all large $N$
\begin{equation}\label{linek2}
  \tilde \bbP_N\left( \prod_{i=1}^k (1+\eta_i)  \ind_{\{  N^{-\gamma} V_N \leq  1+\eta_i  < qV_N\}} \ge t  (q V_N)^k \right) 
 \le   (C_{b,\hat \beta})^k q^{-\gep k}   \,  t^{-\gep}  N^{-k} \, .
\end{equation}
To prove~\eqref{linek2}, we observe that similarly to \eqref{insidelekalkul}, for large values of $N$ the l.h.s.~of~\eqref{linek2} is bounded by
\begin{equation}
\label{linek22}
\bbE\big[1+ \gb_N\eta^{[0,b)} \big]^{-k} (2\beta_N)^{k} \bbE\bigg[  \Big(\prod_{i=1}^k (1+\eta_i) \ind_{\{  (1+\eta_i)  < bV_N\}}\Big) \ind_{\{\prod_{i=1}^k (1+\eta_i)\geq t  (q V_N)^{k} \}} \bigg] \, .
\end{equation}
Similarly to \eqref{apreslekalkul} we then have that it is bounded by
\begin{equation}
\begin{split}
(C \beta_N)^k\bbE\bigg[  \Big(\prod_{i=1}^k   (1+\eta_i)  & \ind_{\{  (1+\eta_i)  < bV_N\}}\Big) \ind_{\{\prod_{i=1}^k (1+\eta_i)\geq t (qV_N)^k \}} \bigg]\\
&\le (C'\beta_N)^k
 \int_{[0,2 b V_N)^k } 
   \ind_{\big\{ \prod_{i=1}^k  u_i  \ge t   (qV_N)^k \big\} }  \prod_{i=1}^k \frac{ \varphi(u_i)}{u_i}  \dd u_i  
   \\
&\le (C'_{b,\hat \gb} N^{-1})^k  \int_{[0,2 b)^k } 
   \ind_{\big\{ \prod_{i=1}^k  v_i  \ge t   q^k \big\} }  \prod_{i=1}^k v^{-1-\gep/2}_i \dd v_i  \, , 
   \end{split}
\end{equation}
where we used a change of variable and the fact that 
$\varphi(V_N u_i) \leq C_b \varphi(V_N) u_i^{-\gep/2}$ uniformly for 
$u_i \in(0,2b)$, by Potter's bound.
Then, \eqref{linek2} finally  follows from \cite[Equation $(4.32)$]{BL20_cont} applied to $h= t (q/b)^k$ (note that when $h \geq 1$ then~\eqref{linek22} is equal to $0$).
 \qed

\subsection{Proof of Proposition~\ref{secondmt}}\label{sec:sm}

We need to control the value of 
\begin{equation}
   \bbE[(W^{a,q}_N-W^{ 0,q}_N)^2]= \bbE[(W^{ a,q}_N)^2] -2\bbE[W^{a,q}_N W^{ 0,q}_N ]+\bbE[(W^{ 0,q}_N)^2].
\end{equation}
and prove that it converges to $0$ when $a$ tends to $0$, uniformly in $N$
(when $\alpha=1$ we must multiply this by $e^{-2\hat \beta\gamma_N}$).
This is the most delicate part of the proof. For didactic purpose and for the sake of making computations more readable, we first show that 
$\bbE[(W^{ a,q}_N)^2]$ is uniformly bounded in $a$ and $N$ when $\alpha\in(1,\ga_c)$. In the case $\ga=1$ we bound $e^{-\hat \gb \gamma_N} \bbE[(W^{ a,q}_N)^2]$.
While this is not a required intermediate step, most of the computation made to prove this are going to be recycled for the actual proof of the Proposition~\ref{secondmt}.

\subsubsection{The second moment is uniformly bounded}

Let us prove the following estimate.
\begin{lemma}\label{lepetilem}
 We have for any $q\ge 1$
 \begin{equation}\label{boond}
  \sup_{a\in [0,1)}\sup_{N\ge 1}e^{-2\hat \beta\gamma_N \ind_{\{\ga=1\}}} \bbE\big[ (W^{ a,q}_N)^2 \big] <\infty.
 \end{equation}
\end{lemma}

\begin{proof}
Let us first treat the case $\alpha\in(1,\alpha_c)$; we deal with the case $\alpha=1$ at the end of the proof.
First of all, notice that by definition~\eqref{def:WNaq} of $W_N^{a,q}$ we have
 \begin{equation}\label{vlap}
  \bbE[(W^{a,q}_N)^2] =  \bE^{\otimes 2}\bbE\bigg[ \prod_{n=1}^N \big(1+\beta_N \eta^{(a)}_{n,S^{(1)}_n} \big) \big(1+\beta_N \eta^{(a)}_{n,S^{(2)}_n}\big) \ind_{\big\{B_{N,q}(S^{(1)})\cap B_{N,q}(S^{(2)})\big\}}\bigg] \, ,
 \end{equation}
 where $\bE^{\otimes 2}$ is the expectation with respect to  two independent walks $S^{(1)}$ and $S^{(2)}$.
Now we can consider the set $I_N(S^{(1)},S^{(2)}):=\{ n\in \lint 1, N\rint, S^{(1)}_n=S^{(2)}_n  \}$. 
As we are interested in an upper bound we can replace 
$B_{N,q}(S^{(1)})\cap B_{N,q}(S^{(2)})$ by the larger event
\begin{equation}\label{largerevent}
\bigg\{ \forall  I\subset I_N (S^{(1)},S^{(2)}) , \,  \prod_{i\in I}   (1+\eta_{i,S_i})  \le q^{|I|} V_N^{|I|} \big(N^{-|I|} \Pi_I \big)^{\gamma} \bigg\} \, .
\end{equation}
The expectation with respect to $\eta_{n,S_n^{(1)}} ,\eta_{n,S_n^{(2)}}$ then simplifies for $n\notin I_N(S^{(1)},S^{(2)})$, and we obtain 
 \begin{equation}
    \bbE \big[ (W^{ a,q}_N)^2 \big] \le \bE^{\otimes 2} \big[ H\big(  I_N(S^{(1)},S^{(2)})  \big) \big]\, ,
 \end{equation}
where for $I\subset \lint 1, n \rint$ we define
\begin{equation}
\label{def:H}
 H( I):= \bbE \bigg[ \prod_{i\in I} \big(1+\beta_N \eta_i^{(a)} \big)^2 \ind_{\tilde B_{N,q} (I)}\bigg] \, ,
\end{equation}
with $\tilde B_{N,q} (I)$  defined by
\begin{equation}
\label{def:tildeB}
 \tilde B_{N,q}(I):= \Big\{ \forall I' \subset I, \  \ 
  \prod_{i\in I'}  (1+\eta_i) < q^{|I'|}  V_N^{|I'|} \big(N^{-|I'|} \Pi_{I'}\big)^{\gamma}  \Big\}\, .
\end{equation}
In \eqref{def:H}, the $(\eta_{i})_{i\ge 1}$ are i.i.d.\
random variables with the same distribution as $\eta_{n,x}$.
Writing 
\[
(1+\beta_N \eta_i^{(a)} )^2 =\big( 1+2 \gb_N (1-\gb_N) \eta_i^{(a)}   -\gb_N^2 \big) + \gb_N^2 \big(1+\eta_i^{(a)} \big)^2
\] 
and expanding the product in $H(I)$, we obtain 
$H(I) = \sum_{J\subset I} \gb_N^{2|J|} H(I,J)$ with 
\begin{align}
\label{expansion1}
 H(I,J) := 
 \bbE \bigg[ \prod_{i\in  I\setminus J} \big( 1+ 2\beta_N (1-\gb_N) \eta^{(a)}_i-\beta^2_N \big) \prod_{j\in J}   (1+\eta^{(a)}_j)^2  \ind_{\tilde B_{N,q} (I)}\bigg] \, .
 \end{align}
Denoting $p_I := \bP^{\otimes 2} ( I_N(S^{(1)},S^{(2)}) = I )$,
we therefore get that
\begin{equation}
\label{eq:EW2notlast}
\bbE[(W^{a,q}_N)^2] \leq  \sum_{I\subset \lint 1, N \rint} p_I \sum_{J\subset I} \gb_N^{2|J|} H(I,J) \, .
\end{equation}
Now, for $J\subset I$, we define
\begin{equation}\label{defdehat}
 \hat B_{N,q} (J) = \bigg\{   \prod_{i \in J} (1+\eta_i) < (qV_N)^{|J|} \big(N^{-|J|} \Pi_{J} \big)^{\gamma}\bigg\} \cap \left\{ \forall i \in J, (1+\eta_i) < q V_N \right\}.
\end{equation}
Of course $\tilde B_{N,q}(I) \subset  \hat B_{N,q} (J)$ since
the bound on $\eta_i$ is obtained when
considering $I'=\{ i\}$ in~\eqref{def:tildeB}.
We therefore get
\begin{equation}
\label{eq:HIJ}
\begin{split}
H(I,J)& \le \bbE \bigg[ \prod_{i\in  I\setminus J} (1+ 2\beta_N (1-\gb_N) \eta^{(a)}_i -\gb_N^2) \prod_{j\in J} \big(1+\eta^{(a)}_j \big)^2\ind_{ \hat B_{N,q} (J)}\bigg]
 \\
 & \leq   (1-\gb_N^2)^{|I|-|J|}     \bbE \Big[  \prod_{i\in J}(1+ \eta_j^{(a)})^2\ind_{ \hat B_{N,q} (J)}\Big] \, .
 \end{split}
 \end{equation}
Overall, using that for any fixed $J$ we have that $\sum_{I \supset J} p_I = \bP^{\otimes 2} ( \forall n\in J, \  S^{(1)}_n=S^{(2)}_n )$,
we obtain that 
\begin{equation}
\label{eq:EW2last}
\begin{split}
   \bbE\big[ (W^{ a,q}_N)^2 \big] & \le \sum_{J\subset \lint 1, N\rint}  
   \beta^{2|J|}_N\bbE \Big[  \prod_{i\in J}(1+ \eta_j^{(a)})^2\ind_{ \hat B_{N,q} (J)}\Big] \bP^{\otimes 2} \big( \forall n\in J, \  S^{(1)}_n=S^{(2)}_n \big)   \\
&   \le  \sum_{J\subset \lint 1, N\rint}  \beta^{2|J|}_N 
   \bbE \Big[   \prod_{i\in J}(1+ \eta_j^{(a)})^2\ind_{ \hat B_{N,q} (J)}\Big]  \, C_0^{|J|} (\Pi_J)^{-\frac d2} \, ,
 \end{split}
\end{equation}
the last line being a simple application of the local Central Limit Theorem.
Now we use  the following estimate which will allow us to conclude  (its proof is postponed).

\begin{lemma}
\label{lekalkul2}
Fix $0<\gep< 2-\ga$. Then
for any $k\ge 1$ and any $t \in (0,1)$, there exists some  $ C_1=   C_{\ga, \hat \gb, q, \gep}$
such that
\begin{equation}
\label{2ndmoment}
  \beta^{2k}_N \, \bbE\bigg[ \prod_{i=1}^k  \big(1+\eta^{(a)}_i\big)^2  \ind_{\big\{ \prod_{i=1}^k  (1+\eta_i) \le t (qV_N)^{k} \big\}}\, \ind_{\{ \forall i\in \lint 1,k \rint, (1+\eta_i)\le  qV_N\}} \bigg] 
 \le     C_1^{k} N^{k(\frac{d}{2}-1)} t^{2-\alpha-\gep} .
\end{equation}
\end{lemma}

Now, using Lemma~\ref{lekalkul2} with $t=  (N^{-|J|} \Pi_J)^{\gamma} =  \prod_{i=1}^{|J|}\big(\frac{ j_i-j_{i-1}}{ N}  \big)^\gamma$  (recall the definition~\eqref{defdehat} of $\hat B_{N,q}(J)$), 
 we finally obtain in \eqref{eq:EW2last} 
\begin{equation}\begin{split}\label{compasumint}
    \bbE[(W^{ a,q}_N)^2]
  &  \le \sum_{k\ge 0} \frac{(C_0 C_1)^k}{N^k}   \sum_{1\le j_1<\dots<j_k\le N} \prod_{i=1}^k\left(\frac{ j_i-j_{i-1}}{ N}  \right)^{\gamma(2-\alpha-\gep)-\frac{d}{2}}\\ 
  & \le \sum_{k\ge 0} (C')^k \int_{0\le t_1<\dots<t_k\le 1}\prod_{j=1}^k\left({ t_j-t_{j-1}} \right)^{  -\tau}\dd t_j \,,
\end{split}
\end{equation}
 where we have set $\tau = d/2 -\gamma(2-\alpha-\gep)$.
We conclude   using  Lemma~\ref{lemGamma},
assuming that~$\gep$ has been fixed small enough so that $\tau <1$, which establishes that
the last integral is equal to $\frac{\Gamma(1- \tau)^k}{ \Gamma( 1+ k (1-\tau))}$.
Hence, the sum in~\eqref{compasumint} is finite.

\smallskip

\noindent \textit{The case $\alpha=1$.}
 We can proceed exactly in the same way, except that instead of~\eqref{vlap} we start with 
 \begin{equation}\label{vlap2}
\bbE[(W^{a,q}_N)^2] =  \bE^{\otimes 2}\bbE\bigg[ \prod_{n=1}^N \big(1+\beta_N \eta^{[a,q)}_{n,S^{(1)}_n} \big) \big(1+\beta_N \eta^{[a,q)}_{n,S^{(2)}_n}\big) \ind_{\big\{B_{N,q}(S^{(1)})\cap B_{N,q}(S^{(2)})\big\}}\bigg] \, .
 \end{equation}
 Then we replace $B_{N,q}(S^{(1)})\cap B_{N,q}(S^{(2)})$ by the larger event in Equation \eqref{largerevent}: when integrating over $\eta_{n,S^{(i)}_n}$ for $n\notin I= I_N(S^{(1)},S^{(2)})$,
 this adds a factor
 \begin{equation}
 \label{eq:integrateonIc}
  \bbE\big[  1+\beta_N \eta^{[a,q)} \big]^{2(N- |I|)}   \, .
 \end{equation}
 Using \eqref{lawrong},
 we therefore can replace \eqref{eq:EW2notlast}  with
 \begin{align*}
 e^{-2 \hat \gb \gamma_N} \bbE[(W^{a,q}_N)^2] 
 \leq  C_q \sum_{I\subset \lint 1,N \rint}   p_I e^{-2\hat \gb \gamma_N   |I|/N}  \sum_{J\subset I} \gb_N^{2 |J|} H(I,J) \, .
\end{align*}
 Using again \eqref{lawrong} to bound above $\bbE\big[  1+2 \beta_N (1-\gb_N) \eta^{[a,q)}  -\gb_N^2\big]$, we may replace~\eqref{eq:HIJ}
  with
 \begin{equation}
 \label{eq:HIJalpha1}
 H(I,J) \leq   C'_q e^{ 2 \hat \gb \gamma_N N^{-1} (|I|-|J|) } \bbE \Big[  \prod_{i\in J}(1+ \eta_j^{(a)})^2\ind_{ \hat B_{N,q} (J)}\Big] \, .
\end{equation}
 We finally end up with the following inequality  in place of~\eqref{eq:EW2last}:
 \begin{align*}
 e^{-2 \hat \gb \gamma_N}  \bbE\big[ (W^{ a,q}_N)^2 \big]  &\le C''_q\sum_{J\subset \lint 1, N\rint}  
 \big( C_0 e^{- 2\hat \gb \gamma_N N^{-1} } \big)^{|J|}  \beta^{2|J|}_N  (\Pi_J)^{-\frac d2}\bbE \Big[  \prod_{i\in J}(1+ \eta_j^{(a)})^2\ind_{ \hat B_{N,q} (J)}\Big] \,.
 \end{align*}  
Since $e^{- 2\hat \gb \gamma_N N^{-1}}$ is bounded above by a constant (recall that $\gamma_N$ is slowly varying), it can be omitted if one changes the value of $C_0$.
 We can then conclude similarly  using Lemma~\ref{lekalkul2} and
 the following computations: we get that  $\sup_{N\in \bbN} e^{- 2\hat \gb \gamma_N}\bbE[(W^{a,q}_N)^2] \leq C_q$ for a constant $C_q$ that depends only on $q$.
\end{proof}

\begin{proof}[Proof of Lemma \ref{lekalkul2}]
Note that there exists a constant $C$ such that for every $a\in [0,1)$ and any $t \geq 0$, we have
\[
 \bbP\big( 1+\eta^{(a)} \ge t \big)\le C \gp(t) t^{-\alpha} \, .
\]
Hence we can apply Proposition \ref{lestartines} with $\mu$ the distribution of $1+\eta^{(a)}$, with a constant~$C$ that does not depend on $a$. After applying Proposition~\ref{lestartines}
the l.h.s.~of~\eqref{2ndmoment} is thus bounded by
\begin{multline*}
 \int_{\bbR_+} \ind_{\big\{
  \prod_{i=1}^k  u_i \le t(qV_N)^k \big \}}\,  \ind_{\{  \forall i\in  \lint 1,k \rint, u_i < qV_N \}} \prod_{i=1}^k u^{1-\alpha}_i  \gp(u_i) \dd u_i
  \\
   \leq (c_q)^k V_N^{k(2-\ga)} \gp(V_N)^k \int_{(0,q)^k} \ind_{\big\{
  \prod_{i=1}^k v_i \le t q^k\big\}} \prod_{i=1}^k v^{1-\alpha -\frac\gep2}_i \dd v_i \, .
\end{multline*}
where we used a change of variable, together with the fact that $\gp( V_N v) /\gp(V_N) \leq  c_q v^{-\gep/2}$ for all $N$ and $v\in (0,q)$, by Potter's bound.
The last integral has been studied in~\cite{BL20_cont}:
from \cite[Equations~$(4.34)$-$(4.35)$]{BL20_cont}, we get that 
it is bounded above by $(C_{\ga,\gep,q})^k t^{2-\ga -\gep}$.
We therefore end up with a constant $C_{\ga,\gep,q}'$ such that
the l.h.s.~of~\eqref{2ndmoment} is bounded by
\begin{multline*}
 \beta^{2k}_N \, \bbE\bigg[ \prod_{i=1}^k  \big(1+\eta^{(a)}_i\big)^2  \ind_{\big\{ \prod_{i=1}^k \eta_i \le t(qV_N)^k \big\}}\, \ind_{\{ \forall i\in \lint 1,k \rint, \eta_i<  qV_N\}} \bigg] 
 \\
 \le    (C_{\ga,q,\gep}' \gb_N^{2} V_N^{(2-\ga)} \gp(V_N) \big)^k  t^{2-\alpha-\gep} \, .
\end{multline*}
Together with~\eqref{uzefulrel}, this yields the conclusion.
\end{proof}

\subsubsection{Convergence of the second moment}\label{convsecmom}

Building on the techniques that we used to prove Lemma~\ref{lepetilem}, we are going to prove Proposition \ref{secondmt}. 
It follows from the following convergence results 
\begin{equation}
\label{limsupsecondmoment}
 \begin{split}
&\lim_{a\to 0} \sup_{N\ge 1} \, e^{- 2\hat \gb \gamma_N \ind_{\{\ga=1\}}}\big| \bbE[(W^{ a,q}_N)^2]- \bbE[(W^{ 0,q}_N)^2] \big|=0,\\
&\lim_{a\to 0} \sup_{N\ge 1}\, e^{- 2\hat \gb \gamma_N \ind_{\{\ga=1\}}} \big| \bbE[W^{ a,q}_N W^{ 0,q}_N]- \bbE[(W^{ 0,q}_N)^2] \big| =0\, .
\end{split}
\end{equation}
Again, we focus the exposition on the case $\ga\in (1,\ga_c)$ and we comment on the case $\ga=1$ along the proof.
Also, we focus on estimating $\bbE[(W^{ a,q}_N)^2]- \bbE[(W^{ 0,q}_N)^2]$ since the other convergence is proved similarly. Notice however that unlike in Section \ref{case1dim} we do not in general have $\bbE[W^{ a,q}_N W^{ 0,q}_N]= \bbE[(W^{ a,q}_N)^2]$,
since $W^{ a,q}_N$ is not a martingale in $a$.
We want to bound 
\begin{align*}
\bE^{\otimes 2}\bbE\Big[ \prod_{n=1}^N 
&\big(1+\beta_N \eta^{(a)}_{n,S^{(1)}_n} \big) \big(1+\beta_N \eta^{(a)}_{n,S^{(2)}_n} \big)
\ind_{\{B_{N,q}(S^{(1)})\cap B_{n,q}(S^{(2)})\}}\Big]
\\
&-\bE^{\otimes 2}\bbE\Big[ \prod_{n=1}^N 
\big(1+\beta_N \eta_{n,S^{(1)}_n} \big) \big(1+\beta_N \eta_{n,S^{(2)}_n} \big)
\ind_{\{B_{n,q}(S^{(1)})\cap B_{N,q}(S^{(2)})\}}\Big].
\end{align*}
Let us stress that from the definition of $B_{N,q}(S)$, we can replace $\eta^{(a)}$ with $\eta^{[a,q)}$
(and $\eta = \eta^{(0)}$ with~$\eta^{[0,q)}$).
Note that the expectations with respect to
$\eta$ and $\eta^{(a)}$ depend
on  $S^{(1)}$, $S^{(2)}$  only via the set
$I_N(S^{(1)}, S^{(2)}) = \{n\in \lint 1,n \rint \, : \, S^{(1)}_n=S^{(2)}_n\}$.
We use this fact to simplify the expression. 
Let  $(\eta_n)_{n\geq 0}$, $(\eta_{n,1})_{n\geq 0}$ and $( \eta_{n,2})_{n\geq 0}$ be i.i.d.\ random variables with the same distribution as $\eta$, and let
$p_I:=\bP^{\otimes 2} ( I_N(S^{(1)}, S^{(2)}) = I )$.
Then, denoting $I^{\cc}=\lint 1,N\rint\setminus I$, we can rewrite the above quantity as
\begin{multline}\label{3sequences}
 \sum_{I\subset \lint 1, N\rint } p_I \bigg(\bbE \Big[  \prod_{n\in I^{\cc} }  \big(1+\beta_N  \eta_{n,1}^{(a)} \big) \big(1+\beta_N  \eta_{n,2}^{(a)} \big) \prod_{n\in I}
 \big(1+\beta_N \eta^{(a)}_n\big)^2    \ind_{\bar B_{N,q}(I)}\Big]\\
 - \bbE \Big[   \prod_{n\in I^{\cc}} \big(1+\beta_N  \eta_{n,1} \big) \big(1+\beta_N   \eta_{n,2} \big) \prod_{n\in I}
 \big(1+\beta_N \eta_n\big)^2    \ind_{\bar B_{N,q}(I)}\Big]\bigg),
\end{multline}
where $\bar B_{N,q}(I)$ is defined as
\begin{equation}\label{barbie}
 \bigg\{ \forall K \subset \lint 1, N \rint, \forall r\in \{1,2\}  , \,     \prod_{i \in  I^{\cc} \cap K } (1+\eta_{i,r})  \prod_{i \in I\cap K} (1+\eta_i) < (qV_N)^{|K|}  \big( N^{- |K|} \Pi_{K} \big)^{\gamma}   \bigg\} \, .
\end{equation}
We can perform a second decomposition of \eqref{3sequences} by expanding the squares and developing the products, as in~\eqref{expansion1}-\eqref{eq:EW2notlast}.
We obtain the following upper bound
\begin{equation}
\label{hereisy}
\left| \bbE[(W^{ a,q}_N)^2]- \bbE[(W^{ 0,q}_N)^2] \right| \leq 
\sum_{I\subset \lint 1,N \rint} \sum_{J \subset I} p_I \gb_N^{2 |J|}  \big| \tilde H^{(a)}(I,J)  - \tilde H (I,J) \big|
\end{equation}
where for $J\subset I$ we have defined $ \tilde H^{(a)} (I,J)$ to be equal to
\begin{multline}
\label{def:tildeH}
\bbE \Big[ \prod_{n\in I^{\cc}}  \big(1+\beta_N  \eta_{n,1}^{(a)}   \big) \big(1+\beta_N  \eta_{n,2}^{(a)}  \big) \prod_{n\in I\setminus J} \big( 1+ 2\beta_N (1-\gb_N) \eta^{(a)}_{n}  -\gb_N^2 \big) 
   \prod_{n\in J} (1+\eta_n^{(a)})^2  \ind_{\bar B_{N,q}(I)}\Big] \, ,
\end{multline}
and $\tilde H(I,J) = \tilde H^{(0)}(I,J)$ (recall that $\eta^{(0)} =\eta$).
In the case $\ga=1$, we define similarly $\tilde H^{(a)} (I,J)$ using $\eta^{[a,q)}$
 in place of $\eta^{(a)}$ (and $\eta^{[0,q)}$
 in place of $\eta$ in $\tilde H(I,J)$).

\smallskip
We are now going to prove upper bounds for
every term in the r.h.s.~of \eqref{hereisy}.
Let us fix some small $\gd>0$.
We will use different estimates for the sets $J$ whose points are 
macroscopically $\delta$-spaced, that is sets $J$ belonging to
\[
\Xi(\gd,N):= \big\{ J \subset \lint 1, N\rint \colon \text{for all } \{j,j'\} \subset J\cup\{0,N\} \text{ we have } |j'-j|\ge \gd N  \big\} \,,
\]
and for the sets $J$ with at least a pair of points at distance smaller than $\delta N$.
The most difficult part will consist in estimating the contribution of
sets $J\in \Xi (\gd,N)$.
By convention, the empty set $J=\emptyset$ belongs to $\Xi(\gd,N)$.
Also, we denote $(j_i)_{i=1}^{|J|}$ the ordered points of $J$,
with by convention $j_0=0$.

\medskip
\noindent
{\it (a) Estimate for sets $J\notin \Xi(\gd,N)$.} 
For sets $J$ that are not macroscopically $\delta$-spaced,
we can use 
the computations made in the previous section to
prove that their contribution to the sum in~\eqref{hereisy} is negligible.

\begin{lemma}\label{smallcontrib}
For every $q \geq 1$ and $0<\gep<2-\ga$ fixed (with $d/2- \gamma(2-\ga-\gep)<1$), there exists a constant $C =C_{\ga,\hat \gb, q,\gep}>0$ such that for all $N\geq 1$, 
 $J\subset \lint 1,N\rint$, and
$a\in [0,1)$  we have
 \begin{equation}\label{fgts}
  e^{- 2\hat \gb \gamma_{N} \ind_{\{\ga=1\}}}  \sumtwo{I \subset \lint 1, N\rint}{I\supset J} p_I \beta^{2|J|}_N  \tilde H^{(a)}(I,J) \le  \left(\frac{C}{N}\right)^{|J|} \prod_{i=1}^{|J|}   \left(\frac{ j_i-j_{i-1}}{ N}  \right)^{\gamma(2-\alpha-\gep)-\frac{d}{2}}. 
 \end{equation}
 As a consequence, for all $\zeta>0$ and any $q\geq 1$,
 there exists $\gd = \gd(\epsilon,q)$ such that 
 \begin{equation}\label{lespetits}
\sup_{a\in [0,1)}  \sup_{N\in \bbN}    e^{- 2\hat \gb \gamma_{N} \ind_{\{\ga=1\}}}  \sum_{I\subset \lint 1, N\rint } \sumtwo{J\subset I}{J\notin \;\Xi(\gd,N)}  p_I \beta^{2|J|}_N\big| \tilde H^{(a)}(I,J)  - \tilde H (I,J) \big|\le \zeta \, .
 \end{equation}
\end{lemma}

\begin{proof}
For the inequality \eqref{fgts}, recalling the definition  \eqref{defdehat} of $\hat B_{N,q}(J)$ and observing that 
$\bar B_{N,q}(I)\subset \hat B_{N,q}(J)$, we get as in~\eqref{eq:HIJ} that for $\ga>1$ we have
 \begin{equation}
  \tilde H^{(a)} (I,J) \le    \bbE \Big[  \prod_{n \in J} \big( 1+ \eta_n^{(a)}\big)^2\ind_{ \hat B_{N,q} (J)}\Big] \, .
 \end{equation}
In the case $\ga=1$, integrating over $\eta_{n,r}^{[a,q)}$ for $n\notin I$ and $\eta_{n}^{[a,q)}$ for $n\in I\setminus J$ in analogy with~\eqref{eq:integrateonIc}-\eqref{eq:HIJalpha1}, 
and using that $e^{ - 2\gb_N \gamma_N N^{-1}}$ is bounded by a constant $C$, this is replaced with
 \begin{equation}
 \label{tildeHIJ1}
  e^{- 2\hat \gb \gamma_N} \tilde H^{(a)} (I,J)  \leq   C'_q  C^{|J|} \bbE \Big[  \prod_{n \in J} \big( 1+ \eta_n^{(a)}\big)^2\ind_{ \hat B_{N,q} (J)}\Big]\,.
 \end{equation}
Summing over sets $I$ containing $J$
and using that 
$
\sum_{I\supset J} p_I  =\bP^{\otimes 2} \big( \forall n\in J, \  S^{(1)}_n=S^{(2)}_n \big)$,
we therefore get,
 analogously to~\eqref{eq:EW2last}, 
\begin{align*}
 e^{- 2\hat \gb \gamma_N \ind_{\{\ga=1\}}}   \sum_{I\supset J} p_I \beta^{2|J|}_N  \tilde H^{(a)}(I,J)
 &  \le C'_q (C'_0 )^{|J|} \beta^{2|J|}_N  \bbE \Big[  \prod_{n\in J} \big(1+ \eta^{(a)}_n\big)^2\ind_{ \hat B_{N,q} (J)}\Big]    (\Pi_J)^{-\frac d2}     \notag\\
 &   \le  C'_q  (C')^{|J|}  N^{-|J|} \big( N^{-|J|} \Pi_J \big)^{\gamma(2-\alpha-\gep)-\frac{d}{2}} \,,
\end{align*}
where we used
Lemma \ref{lekalkul2} for the second inequality with $t =( N^{-|J|} \Pi_J)^{\gamma}$.
This proves~\eqref{fgts}.
Now, summing~\eqref{fgts} over all $J\notin \Xi(\gd,N)$ (we necessarily have $|J|\geq 1$),  and setting
$\tau := d/2 -\gamma(2-\alpha-\gep) <1$, we obtain, using a standard sum-integral comparison that for all $a\in [0,1)$,
 in analogy with~\eqref{compasumint},
\begin{multline}
 e^{- 2\hat \gb \gamma_N \ind_{\{\ga=1\}}}   \sumtwo{J\subset I}{J\notin \Xi(\gd,N)}  p_I \beta^{2|J|}_N \big( \tilde H^{(a)} (I,J) + \tilde H (I,J) \big)\\
 \le 2\sum_{k = 1}^{\infty} (C')^k \int_{0<t_1<\dots<t_k<1} 
\ind_{\{\min_i(t_i-t_{i-1})\le \delta \}} 
 \prod_{i=1}^{k} \left(t_i-t_{i-1}  \right)^{ - \tau} \dd t_i\, .
\end{multline}
By symmetry, we can assume that the minimum $\min_i(t_i-t_{i-1})$ is attained for $i=1$, \textit{i.e.} that $t_1\leq \gd$, loosing only a factor $k$.
We thus get that the l.h.s.~of~\eqref{lespetits} 
is bounded by
\begin{equation*}
2 \Big( \int_0^{\gd}  t_1^{\gamma(2-\alpha-\gep)-\frac{d}{2}} \dd t_1 \Big) \times  \sum_{k = 1}^{\infty} k (C')^k  \int_{0<t_2<\dots<t_k<1}  
 \prod_{i=2}^{k} \left(t_i-t_{i-1}  \right)^{  -\tau} \dd t_i \, ,
\end{equation*}
where by convention we set the integral inside the sum is equal to $1$ for $k=1$.
As for~\eqref{compasumint}, the last integral is equal to
$\frac{\Gamma(1- \tau)^k}{ \Gamma( 1+ k (1-\tau))}$
(cf.\ Lemma~\ref{lemGamma}), and the series converges.
Then, the first integral can be made arbitrarily small by taking $\delta$ small.
\end{proof}

\smallskip
\noindent
{\it (b) Estimate for sets $J\in \Xi(\gd,N)$.}
We now turn to the case of sets $J$ that are
macroscopically $\gd$-spaced: we give an
estimate on the contribution to the r.h.s.~of~\eqref{hereisy}
of the sets $J\in  \Xi(\gd,N)$ that enables us to conclude
the proof of Proposition~\ref{secondmt}.

\begin{lemma}\label{maincontrib}
 For any $\zeta>0$, for all $\delta>0$, $q\ge 1$, there exists $a_0 = a_0(\zeta,\delta,q)  >0$ such that for any $N\geq 1$ and every $J\in  \Xi(\gd,N)$ we have, for all $a\leq a_0$
 \begin{equation}\label{fgts2}
  e^{- 2\hat \gb \gamma_N \ind_{\{\ga=1\}}}     \sumtwo{I\subset \lint 1, N\rint}{I\supset J}  p_I \beta^{2|J|}_N \big| \tilde H^{(a)}(I,J)  - \tilde H(I,J) \big|  \le \zeta N^{-|J|}. 
  \end{equation}
 As a consequence we obtain that, for all $N\geq 1$, for all $a\leq a_0$,
 \begin{equation}\label{lasomme}
  e^{- 2\hat \gb \gamma_N \ind_{\{\ga=1\}}}   \sum_{I\subset \lint 1, N\rint } \sumtwo{J\subset I}{J\in \Xi(\gd,N)}  p_I \beta^{2|J|}_N\big| \tilde H^{(a)}(I,J)  - \tilde H(I,J) \big| 
 \le e \, \zeta.
 \end{equation}
\end{lemma}

Combining this lemma with Lemma~\ref{smallcontrib}
shows that for any fixed $q$, \eqref{hereisy}
can be made arbitrarily small uniformly in $N$
by choosing first $\gd$ small (so that \eqref{lespetits} holds)
and then~$a$ small (so that \eqref{lasomme} holds). This therefore shows the first part in~\eqref{limsupsecondmoment}.

\begin{proof}[Proof of Lemma~\ref{maincontrib}]
Of course, \eqref{lasomme} follows easily from \eqref{fgts2}. Indeed from \eqref{fgts2} the sum in the r.h.s.~of \eqref{lasomme} is smaller than (using the binomial expansion)

\begin{equation}
 \zeta \sum_{J\subset \lint 1, N\rint} N^{-|J|}= \zeta (1+ N^{-1})^N.
\end{equation}
%
%
%
%
%
Let us warn the reader that the proof of \eqref{fgts2} is quite lengthy and technical.
We are going to use another representation for  $\tilde H^{(a)}(I,J)$ and $\tilde H(I,J)$
as probabilities.
Define, for $J\subset I \subset \lint 1,N\rint$,
\begin{multline}
\label{def:YIJ}
Y^{(a)}(I,J) = \prod_{n\in I^{\cc}}  \big(1+\beta_N  \eta_{n,1}^{(a)}   \big) \big(1+\beta_N  \eta_{n,2}^{(a)}  \big) \prod_{n\in I\setminus J} \big( 1+ 2\beta_N (1-\gb_N) \eta^{(a)}_{n}  -\gb_N^2 \big) 
   \prod_{n\in J} (1+\eta_n^{(a)})^2 \\
   \times \ind_{ \big\{ \forall i\in I, 1+\eta_i < q V_N \big\} \cap \big\{\forall i\in I^c,  \forall r \in \{1,2\}  , 1+\eta_{i,r} < q V_N \big\} } 
\end{multline}
and $Y(I,J)= Y^{ (0)}(I,J)$; we omit the dependence in $(\eta_{n,1}),(\eta_{n,2}),(\eta_n)$
in the notation for simplicity.
This way, we have $\tilde H^{(a)}(I,J) = \bbE[ Y^{(a)}(I,J)\ind_{\bar B_{N,q} (I) } ]$
We now interpret $Y^{(a)}(I,J)$ and $Y(I,J)$ as probability densities for $(\eta_n,\eta_{n,1},\eta_{n,2})_{n=1}^N$.
We define for $J\subset I \subset \lint 1, N\rint$ 
\begin{equation}
 \frac{\dd \tilde \bbP^{(a)}_{I,J}}{\dd \bbP}:=    \frac{Y^{(a)}(I,J)}{\bbE\big[Y^{(a)}(I,J)\big]}\, ,
\end{equation}
and $\tilde \bbP_{I,J}=\tilde \bbP^{(0)}_{I,J}$.
We have therefore reduced the problem to comparing the probability of the event $\bar B_{N,q}(I)$
under $\tilde \bbP^{(a)}_{I,J}$ and $\tilde \bbP_{I,J}$.
The left-hand side of~\eqref{fgts2} is
 equal to
\begin{equation}\label{leqacontrol}
  e^{- 2\hat \gb \gamma_N \ind_{\{\ga=1\}}}  \sumtwo{I\subset \lint 1, N\rint }{ I\supset J } p_I\beta^{2|J|}_N \bbE\big[Y(I,J)\big]\;
  \Bigg|  \frac{ \bbE\big[Y^{(a)}(I,J)\big]}{\bbE\big[Y(I,J)\big]}  \tilde \bbP^{(a)}_{I,J} \big( \bar B_{N,q}(I) \big) - 
 \tilde \bbP_{I,J} \big( \bar  B_{N,q}(I) \big) \Bigg| \, .
\end{equation}
The expectation of $Y^{(a)}$ and $Y$ can be expressed as follows
\begin{multline}
\bbE\big[ Y^{(a)}(I,J) \big]
= \Big(1+\beta_N \bbE\big[ \eta \ind_{\{ 1+ \eta < q V_N \}} \big]\Big)^{2(N-|I|)} \\
\times \Big(1 - \gb_N^2 +2\beta_N (1-\gb_N) \bbE \big[ \eta \ind_{\{1+\eta <  qV_N \big\}} \big] \Big)^{|I|-|J|}
\, \bbE \big[ (1+\eta^{(a)})^2 \ind_{\{1+\eta < q V_N\}} \big]^{|J|} \,.
\label{calculEYIJ}
\end{multline}
In particular, when $\ga>1$,  using that $\bbE[\eta]=0$ we get that $\bbE\big[ \eta \ind_{\{ 1+ \eta < q V_N \}}]\leq 0$, so 
\begin{equation*}
\bbE\big[ Y^{(a)}(I,J) \big] \le  \bbE \big[ (1+\eta^{(a)})^2 \ind_{\{1+\eta < q V_N \}} \big]^{|J|} \, .
\end{equation*}
When $\ga=1$, using again \eqref{lawrong} (or see~\eqref{tildeHIJ1})
 and the fact that
$e^{- \hat\gb \gamma_N N^{-1}}$ is bounded by a constant, we get that
\begin{equation*}
 e^{-2\hat \gb \gamma_N} \bbE\big[ Y^{(a)}(I,J) \big] \leq  C'_q C^{|J|} \bbE \big[ (1+\eta^{(a)})^2 \ind_{\{1+\eta < q V_N \}} \big]^{|J|} \, .
\end{equation*}
In particular, we obtain that
\begin{multline}
   e^{-2\hat \gb \gamma_N \ind_{\{\ga=1\}}}    \sumtwo{I\subset \lint 1, N\rint }{ I\supset J } p_I\beta^{2|J|}_N\bbE\big[Y(I,J)\big] \\
  \le  C'_q  C^{|J|}  \beta^{2|J|}_N \bbE \big[ (1+\eta)^2 \ind_{\{1+\eta < q V_N \}} \big]^{|J|}   \bP^{\otimes 2} \big( \forall n\in J, \  S^{(1)}_n=S^{(2)}_n \big) \, .
\end{multline}
Now, we can use that we are working with $J \in \Xi(\gd,N)$, for which we have
\[
\bP^{\otimes 2} \big( \forall n\in J, \  S^{(1)}_n=S^{(2)}_n \big) \leq  \Big( \frac{C_0}{ (\gd N)^{d/2}}\Big)^{|J|} \leq C_{\gd}\, N^{- \frac{d}{2} |J|} \, ,
\]
where  $C_{\gd} = (C_0 \gd^{-d/2})^{1/\gd}$, using  that $|J|\leq 1/\gd$
for $J \in \Xi(\gd,N)$.
From the calculations of Section~\ref{sec:uzeful} (or from~\eqref{lonzo}), we have
$\beta^{2}_N \bbE \big[ (1+\eta)^2 \ind_{\{1+\eta \le qV_N  \}}\big] \le C_{\hat \beta, q} N^{\frac{d}{2}-1}$,
so we have that 
\begin{equation}
  e^{-2\hat \gb \gamma_N \ind_{\{\ga=1\}}}   \sumtwo{I\subset \lint 1, N\rint }{ I\supset J } p_I\beta^{2|J|}_N\bbE\big[Y(I,J)\big] \le C'_{\hat\beta, q,\gd} N^{-|J|} \, ,
\end{equation}
with $C'_{\hat\beta, q,\gd} =  C'_q C_{\gd} ( C C_{\hat \beta, q})^{1/\gd}$ (using again that $|J|\leq 1/\gd$).
Hence to prove that \eqref{leqacontrol} converges indeed to $0$ uniformly in $N$, we need to prove that for a fixed value of $q$ and $\delta$ we have
\begin{equation}\label{ogniuno}
  \lim_{a\to 0} \sup_{N\ge 1} \suptwo{J\in \Xi(\gd,N)}{J\subset I \subset \lint 1, N\rint}\Bigg| \frac{ \bbE\big[Y^{(a)}(I,J)\big]}{\bbE\big[Y(I,J)\big]}  \tilde \bbP^{(a)}_{I,J} \big(\bar B_{N,q}(I)\big)- 
 \tilde \bbP_{I,J}\big( \bar  B_{N,q}(I) \big) \Bigg|=0\, .
\end{equation}
 We are in fact going to split the proof of \eqref{ogniuno}, by proving separately the following two following convergences 
\begin{equation}\label{ogniunoss}
  \lim_{a\to 0} \sup_{N\ge 1} \suptwo{J\in \Xi(\gd,N)}{J\subset I \subset \lint 1, N\rint}\Bigg| \frac{ \bbE\big[Y^{(a)}(I,J)\big]}{\bbE\big[Y(I,J)\big]}-1\Bigg|=0.
  \end{equation}
  \begin{equation}\label{ogniunos}
   \lim_{a\to 0} \sup_{N\ge 1} \suptwo{J\in \Xi(\gd,N)}{J\subset I \subset \lint 1, N\rint}\Big|  \tilde \bbP^{(a)}_{I,J} \big(\bar B^{\cc}_{N,q}(I) \big)- 
 \tilde \bbP_{I,J} \big(\bar  B^{\cc}_{N,q}(I)\big) \Big|=0\, .
\end{equation}
Now, recalling the expression~\eqref{calculEYIJ},
we get that 
\[
\frac{ \bbE\big[Y^{(a)}(I,J)\big]}{\bbE\big[Y(I,J)\big]} = \frac{\bbE \big[ (1+\eta^{(a)})^2 \ind_{\{1+\eta < q V_N\}} \big]^{|J|}}{\bbE \big[ (1+\eta)^2 \ind_{\{1+\eta < q V_N\}} \big]^{|J|}  } \, .
\]
 Again, we notice that the above ratio is non-decreasing in $a$.
Analogously to the calculation done in \eqref{eq:etasquare}-\eqref{lonzo},
thanks to \eqref{moments3}
we find that for any fixed $q\geq 1$ and $a\in [0,1)$  we have,
as $N\to\infty$,
\begin{align}
\bbE \big[ (1+\eta^{(a)})^2 \ind_{\{1+\eta < q V_N\}} \big]
 & =  \bbE \Big[(1+\eta)^2 \ind_{\{ a V_N \leq 1+ \eta <   q V_N\}} \Big] 
 + ( 1-\kappa_N^{(a)} )^2 \bbP\big( 1+\eta \leq   a V_N \big) \notag\\
 & =  \frac{\ga d^{-d/2}}{2(2-\ga)}  \hat \gb^2 \big( q^{2-\ga} - a^{2-\ga} \big) N^{\frac d2 -1}  +o(N^{\frac{d}{2}-1})\,.
 \label{expsquare}
\end{align}
Using monotonicity in $a$ and continuity at $0$ we obtain that given $\gep>0$, there exists $a_0(\gep,\delta)$ such that for $a\in (0,a_0)$ and every $N\ge 1$ we have
\[
 1-\delta \gep \le \frac{\bbE \big[ (1+\eta^{(a)})^2 \ind_{\{1+\eta < q V_N\}} \big] }{\bbE \big[ (1+\eta)^2 \ind_{\{1+\eta < q V_N\}} \big] } \le 1 \,.
 \]
Note that  the $o(N^{\frac{d}{2}-1})$ term in \eqref{moments3} is  uniform in $a\in(0,1)$.
%
%
Using the fact that $|J| \leq 1/\gd$ for any $J\in \Xi(\gd,N)$, we therefore  get that given $\gep>0$,
for every $N\ge 1$ and $a\in (0,a_0)$
\begin{equation}
\left|  \frac{ \bbE\big[Y^{(a)}(I,J)\big]}{\bbE\big[Y(I,J)\big]}-1\right| =\Bigg|\bigg( \frac{\bbE \big[ (1+\eta^{(a)})^2 \ind_{\{1+\eta < q V_N\}} \big] }{\bbE \big[ (1+\eta)^2 \ind_{\{1+\eta < q V_N\}} \big] } \bigg)^{|J|}-1\Bigg| \le  \gep \,,
\end{equation}
 which conclude the proof of 
\eqref{ogniunoss}. Let us now prove \eqref{ogniunos}.
Recalling the definition~\eqref{barbie} of $\bar B_{N,q}$,
we have that $\bar B_{N,q}(I)$ \emph{does not occur} if the product of the  variables $\eta$ for some subset $K\subset \lint 1,N \rint$ assumes a high value. 
We are going to split this event according to whether the points in $K$ are well spaced or not.
Given $\delta'>0$ (which we are going to chose small and depending on $\delta$) we have
\begin{equation}
 \bar  B^{\cc}_{N,q}(I)=  C^{(1)}_{q,\gd'}(I)\cup C^{(2)}_{q,\gd'}(I),
\end{equation}
where 
\begin{equation*}
\begin{split}
  C^{(1)}_{q,\gd'}(I):=
   \bigcup_{r=1,2}\Big\{ \exists K  \in \Xi(\gd',N) ,
 \prod_{i\in K\cap I} (1+\eta_i)   \prod_{i\in K\cap I^\cc} (1+\eta_{i,r})\geq  (qV_N)^{|K|}   \big( N^{-|K|} \Pi_K \big)^{\gamma}\Big\},\\
   C^{(2)}_{q,\gd'}(I):=
  \bigcup_{r=1,2} \Big\{  \exists K  \notin \Xi(\gd',N) ,
  \prod_{i\in K\cap I} (1+\eta_i)   \prod_{i\in K\cap I^\cc} (1+\eta_{i,r})\geq  (qV_N)^{|K|}   \big( N^{-|K|} \Pi_K \big)^{\gamma} \Big\}\,.
 \end{split}
\end{equation*}

Hence we have
\begin{align*}
\Big| \tilde \bbP^{(a)}_{I,J}\big( \bar B_{N,q}^{\cc}(I)\big)-   \tilde \bbP_{I,J}\big(\bar B_{N,q}^{\cc}(I) \big)\Big|
& \le \Big |\tilde \bbP^{(a)}_{I,J} \big(C^{(1)}_{q,\gd'}(I) \big)- \tilde \bbP_{I,J}\big(C^{(1)}_{q,\gd'}(I)\big) \Big| \\
 & \qquad \qquad \qquad  + \tilde \bbP^{(a)}_{I,J}\big(C^{(2)}_{q,\gd'}(I)\big)+\tilde \bbP_{I,J}\big(C^{(2)}_{q,\gd'}(I)\big) \, ,
\end{align*}
and we need to show that all three terms are small.
To estimate the probability of $C^{(2)}_{q,\gd'}(I)$
under  $\tilde \bbP^{(a)}_{I,J}$ (and $\tilde \bbP_{I,J}$),
we use a union bound and estimates which are similar to the ones previously used in  Lemma \ref{smallcontrib}.
The calculations are heavy and we postpone the details to the end of the section: the conclusion is summarized by the following lemma.

\begin{lemma}\label{brutforce}
For any $\zeta>0$ and any
 fixed $\delta,q$, there exists $\delta' = \delta'(\delta,q,\zeta)$ such that for all  $J\subset I \subset \llbracket 1, N\rrbracket$ with $J\in \Xi(\gd,N)$ and $a\in [0,1)$ we have
 \begin{equation}
  \tilde\bbP^{(a)}_{I,J}\big(C^{(2)}_{q,\gd'}(I)\big)\le \zeta. 
 \end{equation}
\end{lemma}

To conclude the proof of~\eqref{ogniunos} (and thus of Lemma~\ref{maincontrib}) we therefore need to show that
for fixed $\zeta,\gd',\gd,q$, if $a$ is sufficiently small,
\textit{i.e.}\ if $a\le a_0(\zeta,\delta',\delta,q)$, then we have
for all  $J\subset I \subset \llbracket 1, N\rrbracket$ with $J\in \Xi(\gd,N)$
\begin{equation}
\label{probaCsmall}
\Big |\tilde \bbP^{(a)}_{I,J} \big(C^{(1)}_{q,\gd'}(I) \big)- \tilde \bbP_{I,J}\big(C^{(1)}_{q,\gd'}(I)\big) \Big|\le \zeta\, .
\end{equation}
To show this we are going to prove that: 
\begin{itemize}
 \item [($*$)]  If $a$ is sufficiently small, then for all values of $N$, the event $C^{(1)}_{q,\gd'}(I)$ is measurable with respect to the $\sigma$-field
 $$\sigma\big( (\eta^{(a)}_n, \eta^{(a)}_{n,1},\eta^{(a)}_{n,2})_{n=1}^N   \big).$$
 \item [($**$)] The two distributions 
 \[ \tilde \bbP^{(a)}_{I,J} \Big( (\eta^{(a)}_n, \eta^{(a)}_{n,1},\eta^{(a)}_{n,2})_{n\in\lint 1,N\rint} \in \cdot \Big) \quad  \text{and}  
 \quad \tilde \bbP_{I,J} \Big( (\eta^{(a)}_n, \eta^{(a)}_{n,1},\eta^{(a)}_{n,2})_{n\in\lint 1,N\rint} \in \cdot \Big) \] 
 are close in total variation.
\end{itemize}

\smallskip
\noindent
($*$).\ \  
Note that we have, by definition of $\Xi(\gd',N)$ and $\Pi_K$
\begin{equation*}
\forall K\in \Xi(\gd',N), \quad  \big( N^{-|K|} \Pi_K \big)
\ge (\delta')^{|K|}\ge (\delta')^{1/\delta'}.
\end{equation*}
Also, recalling the definition~\eqref{def:YIJ} of $Y^{(a)}(I,J)$, we notice that under both $\tilde \bbP^{(a)}_{I,J}$ and $\tilde \bbP_{I,J}$ all the environment variables $1+\eta_{i}$, $1+\eta_{i,r}$ are capped by $q V_N$.
Hence, in order to have
\[
 \prod_{i\in K\cap I} (1+\eta_i)   \prod_{i\in K\cap I^\cc} (1+\eta_{i,r})\geq  (qV_N)^{|K|}   \big( N^{-|K|} \Pi_K \big)^{\gamma},\quad \text{ for some } K\in \Xi(\gd',N)\,,
\]
 it is necessary that all of the variables $1+\eta_{i}$, $1+\eta_{i,r}$ involved are larger than $q  (\delta')^{\gamma /\delta'} V_N$
 and thus part~($*$) of the statement holds for $a\le  q (\delta')^{\gamma/\delta'}$. 

\medskip
\noindent
($**$).\ \
Let us prove the following lemma, corresponding to our claim ($**$).

%

\begin{lemma}\label{lemlacompa}
There exists a coupling $\bbQ$ 
between  $\tilde \bbP_{I,J}$ and $\tilde \bbP^{(a)}_{I,J}$
(the marginals of $\bbQ$ are denoted $(\eta_n,\eta_{n,1}, \eta_{n,2})_{n=1}^N$, $( \hat\eta_n,\hat\eta_{n,1}\hat\eta_{n,2})_{n=1}^{N}$)  which is such that $\bbQ$-a.s.
 \begin{equation}\begin{cases}\label{woopz}
  \eta^{(a)}_{n,r}=\hat \eta^{(a)}_{n,r} &\text{ for  } r\in 1,2, \   \forall \, n\in  I^{\cc} \, ,\\
    \eta^{(a)}_n=\hat \eta^{(a)}_n &\text{ for  }\  n\in I \setminus J,
  \end{cases}
 \end{equation}
and 
\begin{equation}\label{lezaut}
\bbQ\big(\exists n\in J,\ \eta_n \ne \hat \eta_n \big) \le C \,  |J| \,  a^{2-\alpha}.
\end{equation}
\end{lemma}

\begin{proof}
 Recalling the definition~\eqref{def:YIJ} of $Y^{(a)}(I,J)$, 
 we observe that both $ \tilde \bbP_{I,J}$ and $\tilde \bbP^{(a)}_{I,J}$ are product measures. We then use three independent
 couplings for the marginals of both measures, \textit{i.e.}\
 we couple $\eta_{n}$ with $\hat\eta_{n}$, $\eta_{n,1}$ with $\hat\eta_{n,1}$
 and $\eta_{n,2}$ with $\hat\eta_{n,2}$,
  independently.
To obtain a coupling such that \eqref{woopz} holds, it is sufficient to observe that the density of the distributions
of $\eta_{n,r}$ and $\hat\eta_{n,r}$ coincide on $[a V_N,\infty)$ for $r=1,2$ and $n\in I^{\cc}$,
and similarly for $\eta_n$ and $\hat \eta_n$ for $n\in I\setminus J$.
For \eqref{lezaut}, we only need to check that the total variation between the two marginal distributions of $\eta_n$ and $\hat \eta_n$ is small (for $n\in J$): it is sufficient to prove that 
 \begin{equation}
  \bbE \left[\left| \frac{(1+\eta)^2 \ind_{\{1+\eta < q V_N\} } }{\bbE \left[(1+\eta)^2  \ind_{\{\eta < q V_N \} } \right]}- \frac{(1+\eta^{(a)})^2  \ind_{\{1+ \eta < q V_N \} }}{\bbE \left[(1+\eta^{(a)})^2  \ind_{\{1+\eta < q V_N\} }\right]} \right| \right]\le C_{q,\ga,\hat \gb} a^{2-\alpha} \,.
 \end{equation}
The above inequality can be checked using
 the computation in~\eqref{expsquare}.
\end{proof}

We can now conclude the proof of~\eqref{probaCsmall}.
Using the coupling $\bbQ$ of Lemma~\ref{lemlacompa},
and using that $C^{(1)}_{q,\gd'}(I)$ is measurable
with respect to $(\eta^{(a)},\eta^{(a)}_1,\eta^{(a)}_2)$
 for $a\leq (\gd')^{\gamma/\gd'}$, we therefore get for
 $a\leq (\gd')^{\gamma/\gd'}$
\begin{equation*}
 \Big| \tilde \bbP^{(a)}_{I,J} \big( C^{(1)}_{q,\gd'}(I) \big)- \tilde \bbP_{I,J} \big( C^{(1)}_{q,\gd'}(I) \big) \Big|
 \le \bbQ \Big(  (\eta^{(a)},\eta^{(a)}_1,\eta^{(a)}_2)\ne (\hat \eta^{(a)},\hat \eta^{(a)}_1,\hat \eta^{(a)}_2) \Big)\le  C |J| a^{2-\alpha},
\end{equation*}
which can be made arbitrarily small by taking $a$ small (recall that $|J|\leq 1/\gd$ for $J\in \Xi(\gd,n)$).
Hence, we have established~\eqref{probaCsmall} and we are only left with proving Lemma~\ref{brutforce}.
\end{proof}

\begin{proof}[Proof of Lemma \ref{brutforce}]
First of all, notice that instead of considering the event $C^{(2)}_{q,\gd'}(I)$, we may by symmetry consider
only the event corresponding to $r=1$.
To simplify notation, we may also consider only one sequence $(\eta_n)^N_{n=1}$. We let   $\hat \bbP^{(a)}_{I,J}$ denote
 the corresponding \emph{reduced} version of $\tilde \bbP^{(a)}_{I,J} $, defined by
\[
 \frac{\dd \hat \bbP^{(a)}_{I,J}}{\dd \bbP} := \frac{\hat Y^{(a)}(I,J)}{\bbE\big[ \hat Y^{(a)}(I,J)\big]} \, ,
\]
where $\hat Y^{(a)}(I,J)$ is  the \emph{reduced} version of $Y^{(a)}(I,J)$, \textit{i.e.}
\begin{multline*}
\hat Y^{(a)}(I,J) :=
 \prod_{n\in I^{\cc}}  \big(1+\beta_N  \eta_{n}^{(a)}   \big)  \prod_{n\in I\setminus J} \big( 1+ 2\beta_N (1-\gb_N) \eta^{(a)}_{n}  -\gb_N^2 \big) 
   \prod_{n\in J} (1+\eta_n^{(a)})^2 \\
\times    \ind_{ \big\{ \forall n\in \lint 1,N\rint, 1+\eta_n< q V_N \big\} } \, .
\end{multline*}
Recalling the definition of the event $C_{q,\gd'^{(2)}}(I)$, 
we also define its \emph{reduced} version
 \begin{equation}
 \label{def:Dq}
 D_{q,\gd'}:=
   \Big\{ \exists K  \notin \Xi(\gd',N) ,
 \prod_{i\in K} (1+\eta_i)  \geq  (qV_N)^{|K|}  \big( N^{-|K|} \Pi_K \big)^{\gamma} \Big\} \,.
 \end{equation}
 We therefore need to prove that if $\gd'$
 is fixed sufficiently small (depending on $\gd,q$ and $\zeta$ but not on $a$), we have
 $\hat \bbP^{(a)}_{I,J}( D_{q,\gd'} )\leq \zeta/2$
 for any $J\in \Xi(\gd,N)$ and $I\supset J$ (uniformly in $N$).

%

Here, one difficulty is that the set $K$ in the event $D_{q,\gd'}$
may have non-empty intersection with $J$ and $J^\cc$, on
which the densities of $\eta_n$ are very different.
We are hence going to simplify once more the problem,
to reduce the event $D_{q,\gd'}$ to sets 
$K \subset J^{\cc}$.
Since $D_{q,\gd'}$ is an increasing event,
it is sufficient to bound the probability of $D_{q,\gd'}$ 
for a probability that stochastically
dominates $\hat \bbP^{(a)}_{I,J}$.
For instance we can consider a (product) measure $\hat \bbP_J^{(a)}$
under which 
 $1+\eta_n^{(a)}=  q V_N$ for $n\in J$ and
for which all the other coordinates have density  $1+2 \beta_N \eta_n^{(a)}$ (instead of $1+ \beta_N \eta_n^{(a)}$ or $1+2 \beta_N(1-\gb_N) \eta_n^{(a)} - \gb_N^2$). 
The stochastic domination follows from
the fact that the functions
$$ \eta\mapsto \frac{1+2\beta_N \eta^{(a)}}{1+\beta_N \eta^{(a)}} \quad \text{ and } \quad  \eta\mapsto \frac{1+2\beta_N \eta^{(a)} }{1-\beta^2_N+2\beta_N(1-\beta_N)\eta^{(a)}} $$
are non-decreasing and \eqref{chebcheb}. Then we can observe repeating the computation \eqref{chebcheb2} that $\hat \bbP^{(a)}_J$ is stochastically dominated by $\hat \bbP^{(0)}_J= \hat \bbP_J$.

Under $\hat \bbP_{J}$,  since $1+\eta_n =q V_N$ for $n \in J$, the description of the event $D_{q,\gd'}$ may be simplified.
In fact, we can observe that 
if  $K \notin \Xi(\gd',N)$
satisfies the condition in $D_{q,\gd'}$ then $K\cup J$ also does.
Indeed, setting $K'=J \setminus K$ we have
\begin{align*}
\prod_{i\in K\cup J}(1+\eta_{i}) = \prod_{i\in K}(1+\eta_{i})  (qV_N)^{|K'|}  &\geq 
 (qV_N)^{-|K\cup J|}  \big( N^{-|K|} \Pi_K \big)^{\gamma}\\ 
 &\ge  (qV_N)^{-|K\cup J|}  \big( N^{-|K\cup J|} \Pi_{K\cup J} \big)^{\gamma}
 \, ,
\end{align*}
where the last inequality follows from the fact that $\big( N^{-|K|} \Pi_K \big)$ decreases when points are added. 
%
We have thus $\hat \bbP_J(D_{q,\gd'})  =  \hat \bbP_J(\hat D_{q,\gd'}(J) )$,
where we define
 \begin{equation*}
  \hat D_{q,\delta'}(J):=\Big\{ \exists K\subset J^{\cc}, \   K\cup J \notin \Xi(\gd',N) , \
\prod_{i\in K}(1+\eta_{i}) \geq  (qV_N)^{|K|}  \big( N^{-|K\cup J|} \Pi_{K\cup J} \big)^{\gamma}\Big\}\, .
 \end{equation*}
Noticing now that $\hat D_{q,\delta'}(J)$
does not depend on $\eta_n$ for $n\in J$,
we may change the distribution of $\eta_n$ for $n\in J$.
In other words, we have $\hat \bbP_J^{}(\hat D_{q,\gd'}(J) )=\hat \bbP_N^{}( \hat D_{q,\gd'}(J))$
where $\hat \bbP_N^{}$ is  the measure defined by
\begin{equation}\label{siouz}
 \frac{\dd \hat \bbP_N^{}}{\dd \bbP}:=    \frac{1}{\bbE \big[ (1+2\gb_N\eta^{}) \ind_{\{1+\eta <   q V_N\}} \big]^{N} } \prod_{n=1}^{N} \big(1+2 \beta_N \eta_n^{} \big) \ind_{\{\forall n\in \lint 1, N\rint, 1+\eta_n < q V_N\}}.
\end{equation}
For this, we will use some of the computations made in the proof of Lemma \ref{firstmt}. 
Then, recalling Claim~\ref{claiclaim} (in particular~\eqref{lescondits}), we can add the condition 
that for all $ n\in K$ one has $1+\eta_n^{(a)}\ge  N^{-\gamma}  V_N$,
without modifying the event $\hat D_{q,\gd'}(J)$.
Hence using a union bound like in \eqref{decomp2} we simply need to show that when $\delta'$ is small the following sum is small
\begin{equation*}
 e^{-2 \hat \gb \gamma_N \ind_{\{\ga =1\}}}\sumtwo{ K\subset J^{\cc}}{ K\cup J \notin \Xi(\gd',N)} 
\hat \bbP_N^{}\Big(  \prod_{i\in K} (1+\eta_i) \ind_{\{  1+ \eta_i\ge N^{-\gamma} V_N \}} \geq (qV_N)^{|K|}   \big(N^{ -|K \cup J|} \Pi_{K\cup J} \big)^{\gamma}  \Big) \,,
\end{equation*}
uniformly in $N$ and $J\in \Xi(\gd,N)$.
We can use Lemma \ref{lekalkul} with $\hat \bbP_N^{}$ instead of~$\tilde \bbP_N^{}$ to bound the above sum.
We can easily reduce to~\eqref{insidelekalkul}, replacing the factor $2$ in the upper bound by some larger constant -- there is a factor $2$ in front of $\gb_N$ in the l.h.s., using also that $\bbE[ (1+2\gb_N\eta^{}) \ind_{\{1+\eta <   q V_N\}}]$ is bounded below by a positive constant.
We stress that the remainder of the proof of Lemma~\ref{lekalkul} is also valid when $\ga=1$,
since we only need that $\ga+\frac{\gep}{2} >1$ in~\eqref{compareintegral}.
Therefore, applying the conclusion of Lemma \ref{lekalkul} with $t=(N^{-|K \cup J|} \Pi_{K\cup J} )^{\gamma}$ we obtain for any arbitrary $\gep>0$
\begin{equation}
  \hat \bbP_N^{} \big( \hat D_{q,\delta'}(J) \big) \le  \sumtwo{ K\subset J^{\cc}}{ K\cup J \notin \Xi(\gd',N)}   C_{\hat \beta, q,\gep}^{|K|}
N^{-|K|}\big(N^{-|K\cup J|} \Pi_{K\cup J} \big)^{\gamma(1-\alpha-\gep)} \,.
\end{equation}
We assume for the rest of the proof that $\vartheta:= \gamma(\alpha-1+\gep) \in(0,1)$.
Let $s_1< \dots < s_m$ denote the position of the points of $J$  divided by $N$ (we work with fixed values of
$(s_i)_{i=0}^{N}$); recall that by
assumption we have $J\in \Xi(\gd,N)$ so that $s_{i}-s_{i-1}\ge \delta$.
By a sum/integral comparison, 
we therefore get that
\begin{equation}
\label{eq:hatP}
 \hat \bbP_N^{} \big( \hat D_{q,\delta'}(J) \big) \le
   \sum_{k=1}^{\infty}    (C'_{\hat \beta, q,\gep})^{k}
\int_{\Delta_k({\bf s},\delta')} 
    \pi_{\bf s}({\bf t})^{-\vartheta} \dd t_1\dots \dd t_k \, ,
\end{equation}
where the set $\Delta_k({\bf s},\delta')$ is defined by (we set  $s_0=0$ and $s_{m+1}=1$ by convention)
\[
 \Delta_k({\bf s},\delta'):=\bigg\{ 0<t_1<\dots<t_k<1 \ :  \min_{i\in \lint 1,k-1\rint}(t_{i+1}-t_{i})\wedge \mintwo {i\in \lint 1,k\rint}{j\in \lint 0, m+1\rint}|t_i-s_j|\le \delta' \bigg\},
\]
and the function $ \pi_{\bf s}({\bf t})$ is defined by 
\[
  \pi_{\bf s}({\bf t}):=\prod_{i=0}^{k+m} (u_{i+1}-u_{i}),
 \]
where the $(u_i)^{k+m}_{i=1}$ are the ordered elements of $\{t_i\}_{i=1}^{k} \cup \{s_j\}_{j=1}^m$, $u_0=0$, $u_{k+m+1}=1$. Note that we have added a last factor in the product  $\pi_{\bf s}({\bf t})$ compared to $\Pi_{K\cup J}$, but it is smaller than one and the exponent is $-\vartheta<0$.
Now we have to bound \eqref{eq:hatP} uniformly
 over all sets ${\bf s}$ that are $\gd$-spaced, 
which still requires some work. Let us define, for $j\in \lint 0,m\rint$, the following subset of $\Delta_k({\bf s}, \gd')$:
\begin{multline}
\Delta^{j}_k({\bf s},\delta')= \big\{  0<t_1<\dots<t_k<1 \ : \ \exists i\in\lint 1,k\rint, \\
\  s_{j} \le t_i< s_{j+1}, 
\text{ and } \min(t_i-s_{j},t_{i}-t_{i-1},s_{j+1}-t_i) \le \delta' \big\} \, ,
\end{multline}
with by convention $s_{m+1}=1$, $s_0=t_0=0$. 
From~\eqref{eq:hatP}, and using that  $\Delta_k({\bf s},\delta')=\bigcup_{j=0}^m\Delta^{j}_k({\bf s},\delta')$,
we therefore get that
\begin{equation}
\label{eq:hatP2}
\hat \bbP_N ^{}\big( \hat D_{q,\delta'}(J) \big) \le
  \sum_{j=0}^m  \sum_{k=1}^{\infty}    (C'_{\hat \beta, q,\gep})^{k} \int_{\Delta^j_k({\bf s},\delta')} 
     \pi_{\bf s}({\bf t})^{-\vartheta} \dd t_1\dots \dd t_k \, ,
\end{equation}
and we control each integral in the last sum.
Decomposing over the number $\ell_{i}$ of points~$t_r$ falling in $(s_{i},s_{i+1})$ (which may be equal to $0$ except for $i=j$), we have after scaling on each interval $(s_{i},s_{i+1})$ (and using Lemma~\ref{lemGamma} for the intervals with index $i\neq j$)
\begin{multline*}
\int_{\Delta^j_k({\bf s},\delta')} 
    \pi_{\bf s}({\bf t})^{- \vartheta} \dd t_1\dots \dd t_k\\
  = \sumtwo{\ell_0, \ldots, \ell_{m} , \, \ell_j\geq 1 }{ \ell_0+\cdots + \ell_m =k}
\prod_{i \in \lint 0,m \rint \setminus \{j\} }   (s_{i+1} -s_i)^{ (1-\vartheta)\ell_{i}-\vartheta }   \frac{\Gamma( 1-\vartheta )^{\ell_{i}+1}}{ \Gamma( (\ell_{i} +1) (1-\vartheta) )} \\
 \times(s_{j+1} -s_{j})^{( 1- \vartheta) \ell_j- \vartheta}
\!\!\!\!\!\!\!\!\!\!\!\!\!\!\!\!
\inttwolimits{0<t_1< \cdots < t_{\ell_j} < 1} {\min_{0\leq r \leq \ell_j} (t_{r+1}-t_{r}) \leq \gd'/ (s_{j+1} -s_j)}
\!\!\!\!\!\!\!\! \prod_{r=0}^{\ell_j } 
(t_{r+1}-t_r)^{-\vartheta} \dd t_1 \ldots \dd t_k\, .
\end{multline*}
In the last integral, we can use that $s_{j+1} -s_j\geq \gd$.
By symmetry, we can assume that the 
minimum $\min_{0\leq r \leq \ell_j} (t_{r+1}-t_{r})$
is attained for $r=1$, losing a factor $\ell_j$.
 Using  Lemma~\ref{lemGamma} again, this allows to bound the integral in the last line by 
\begin{multline}
 \ell_j \int_{0<t_1< \cdots < t_{\ell_j} < 1}
 \prod_{r=0}^{\ell_j } 
(t_{r+1}-t_r)^{-\vartheta}\ind_{\{t_1\le \delta'/\delta\}} \dd t_1 \ldots \dd t_k \\ = \frac{\ell_j \gG(1-\vartheta)^{\ell_j}}{ \gG(\ell_j(1-\vartheta))}\int^{\delta'/\delta}_0 t_1^{-\vartheta}(1-t_1)^{ (1-\vartheta)(\ell_j-1) -\vartheta}   \dd t_1\,,
\end{multline}
Altogether, bounding all the other $(s_{i+1} -s_i)$  by either $1$ or $\delta$ and using that $\ell_j \geq 1$,
we get that
\begin{multline*}
\int_{\Delta^j_k({\bf s},\delta')} 
    \pi({\bf t},{\bf s})^{1-\vartheta} \dd t_1\dots \dd t_k \\\leq  \sumtwo{\ell_0, \ldots, \ell_{m}, \ell_j\geq 1 }{ \ell_0+\cdots + \ell_m =k}  \left(\prod_{i \in \lint 0,m\rint \setminus \{j\} } \frac{ \delta^{-\vartheta}\Gamma( 1-\vartheta )^{\ell_{i}+1}}{ \Gamma( (\ell_{i} +1) (1-\vartheta) )} \right)
  \times \frac{ \gd^{-\vartheta} \ell_j \gG(1-\vartheta)^{\ell_j}}{ \gG(\ell_j(1-\vartheta))} \frac{  (1-\gd'/\gd)^{-\vartheta}}{(1-\vartheta)}  (\delta'/\delta)^{1-\vartheta} \,.
\end{multline*}
Note that by symmetry the upper bound does not depend on $j$.
Going back to~\eqref{eq:hatP2}, summing of all values for $j$ and $k$
 and factorizing the sum we get that
\begin{multline*}
  \hat \bbP_N^{} \big( \hat D_{q,\delta'}(J) \big)
   \leq (m+1) \bigg( \sum_{\ell=0}^{\infty}   (C'_{\hat \beta, q,\gep})^{\ell} \frac{ \delta^{-\vartheta} \Gamma( 1-\vartheta) )^{ \ell+1} }{  \Gamma( (\ell +1)(1-\vartheta)+1) )}  \bigg)^{m-1} \\
  \qquad   \times \frac{  (1-\gd'/\gd)^{-\vartheta}}{(1-\vartheta)}  (\delta'/\delta)^{1-\vartheta}\sum_{\ell=1}^{\infty} \ell \,  (C'_{\hat \beta, q,\gep})^{\ell} \, \frac{\Gamma(1 -\vartheta) )^{\ell }}{ \Gamma( \ell (1-\vartheta) )} \,.
\end{multline*}
All the sums are finite, and using that $m = |J|\leq 1/\gd$
(recall that $J\in \Xi(\gd,N)$),
we can therefore choose $\gd'$ small enough (how small depends only on $\gd$) so that $\hat \bbP_N^{} ( \hat D_{q,\delta'}(J) ) \leq \zeta/2$.
\end{proof}

 \subsection{Adapting the proof to the case of general bounded $f$}\label{modifgenf}

 Let us focus on the proof developed for $d\ge 2$, since it also works in dimension $1$.
 For simplicity of notation, we assume here  that $\ga\in (1,\ga_c)$ but the case $\alpha=1$ is exactly the same.
We define 
\begin{equation}
\label{def:WNaqf}
 W^{a,q}_N(f):=  \bE\bigg[f(S^{(N)})\prod_{n=1}^N \big( 1+\beta_N \eta^{(a)}_{n,S_n} \big)\ind_{B_{N,q}(S) } \bigg].
\end{equation}
Similarly to Lemma \ref{firstmt} and Proposition \ref{secondmt}, we need to prove 
\begin{equation}\begin{split}\label{twothingz}
 &\lim_{q\to \infty}\sup_{a\in [0,1)} \sup_{N\ge 1}\bbE \left[| W^{a,q}_N(f)-Z^{a}_N(f)|\right]=0 \\
 &\lim_{a\to 0} \sup_{N\ge 1} \bbE \left[\big( W^{a,q}_N(f)-W^{0,q}_N(f)\big)^2\right]=0.
\end{split}\end{equation}
For the first line, recalling \eqref{lamesuretilde} we have
\begin{equation}
\bbE \left[| W^{a,q}_N(f)-Z^{a}_N(f)|\right]]\le  \bE \left[ |f(S^{(N)})|  \, \tilde \bbP^{(a)}_S [    B^{\cc}_{N,q}(S)]\right] \le \|f \|_{\infty}  \bE \left[ \tilde \bbP^{(a)}_S [    B^{\cc}_{N,q}(S)]\right],
\end{equation}
and we can conclude using the proof of Lemma \ref{firstmt} (recall we proved \eqref{cestuniforme}).
Now for the second line of \eqref{twothingz}, we need to prove the analog of \eqref{limsupsecondmoment}. As in the case $f\equiv 1$, we focus on  
\begin{equation}
\lim_{a\to 0} \sup_{N\ge 1} \Big(  \bbE \big[ W^{a,q}_N(f)^2 \big]-\bbE \big[ W^{0,q}_N(f)^2\big] \Big)=0.
\end{equation}
We can follow the computation of Section \ref{convsecmom}. We can rewrite the above quantity as in \eqref{3sequences} but replacing $p_I$ by 
\[
p_I(f):= \bbE^{\otimes 2}\Big[  f(S^{(N,1)})f(S^{(N,2)})\ind_{\{I_N(S^{(1)},S^{(2)})=I\}} \Big] 
\]
where, with some abuse of notation $S^{(N,1)}$ and $S^{(N,2)}$ denote the rescaled version of the two independent random walks $S^{(1)}$ and $S^{(2)}$.
One can then proceed with the proof  exactly as above, observing that since $p_I(f)\le \|f\|^2_{\infty} p_I$, Lemma \ref{smallcontrib} and Lemma \ref{maincontrib} remain valid when 
$p_I$ is replaced by $p_{I}(f)$.
\qed

\appendix

\section{Stochastic comparison: expectation vs.\ integrals}
\label{app:compare}

We present here a technical result which allows to replace some
expectations with respect to
a random variable whose law $\mu$ satisfies  $\mu([u,+\infty))= \varphi(u)u^{-\ga}$  by integrals with respect to the measure with density $ \alpha u^{-(1+\alpha)}\varphi(u) \dd u$ (which is not necessarily a probability).

\begin{proposition}\label{lescroissants}
Let $\mu$ be a probability measure on $\bbR_+$ that satisfies $\mu\big( [t,+\infty) \big)=\varphi(t)t^{-\ga}$
for some slowly varying $\varphi$.
There exist constants $C$ and $B_0$ (depending on $\varphi$ and~$\alpha$)  such that for all  $k\in \bbN$ and for all non-decreasing function $f  :  \bbR_+^k \to \bbR_+$ with $f(0)=0$, and all $B\ge B_0$, we have 
 \begin{align*}
 \int_{[0,B)^k} f(u_1,\cdots,u_k)  \prod_{i=1}^k \mu (\dd u_k)  & \le C^k \int_{[0, 2B)^k} f(u_1,\cdots,u_k)  \prod_{i=1}^k u_i^{-(1+\alpha)} \varphi(u_i)\dd u_i \, .
 \end{align*}
\end{proposition}

\begin{proof}
The first remark is that we need to show the result only in the case $k=1$. Integrating successively the functions $u_i \mapsto f(u_1, \ldots, u_k)$ then yields the result.
Also, it is sufficient to check the result for a function $f$ that is differentiable and bounded
(the other cases can be obtained by monotone convergence).
We set  $\bar F(u) = \mu([u,\infty))$. 
Using an integration by parts, and applying these inequalities (recall that $f'(u)\geq 0$) we get
\begin{equation}
\int_{[0,B)}   f(u) \mu(\dd u) 
\, =  
\int_{[0,B)}  f'(u)  \bar F(u) \dd u  - f(B)  \bar F(B)
\end{equation}
Now we set 
$$\bar\varphi(u):= u^{\alpha} \int^{\infty}_{u} \alpha v^{-(1+\alpha)} \varphi(v) \dd v.$$ 
Since $\bar\varphi$ is asymptotically equivalent to $\varphi$ at $\infty$, and to $\alpha u^{\alpha}\log u$ at $0$,  we have $\varphi\le C  \bar \varphi$.

\begin{align*}
 \int_{[0,B)}  f'(u)  \bar F(u) \dd u
 &\le C  \int_{[0,B)}  f'(u) u^{-\alpha} \bar\varphi(u) \dd u \\
 &=C\alpha\left(\int_{[0,B)}  f(u) u^{-(1-\alpha)} \varphi(u) \dd u
 + f(B) B^{-\alpha}\bar\varphi(B)\right)  \,,
\end{align*}
where we used another integration by parts for the last identity.
Hence we obtain that 
\begin{equation}\label{zlooz}
 \int_{[0,B)}   f(u) \mu(\dd u) \le  C \alpha \int_{[0,B)}  f(u) u^{-(1-\alpha)} \varphi(u) \dd u  +  C\alpha  f(B) B^{-\alpha} \bar \varphi(B) \,. 
\end{equation}
Now, to conclude with use that $f$ is non-decreasing and that $\bar \varphi$ and $\varphi$ are asymptotically equivalent to obtain that for $B$ sufficiently large
 $$f(B) B^{-\ga}\bar  \varphi(B)  \le C' \int_{[B,2B)} f(u) u^{-(1+\ga)} \varphi(u) \dd u ,$$
 so that the second term in \eqref{zlooz} can be absorbed into the first one.
\end{proof}

\begin{proposition}\label{lestartines}
Let $\mu$ be a probability  measure on $\bbR_+$ such that  $\mu\big([t,\infty) \big) \leq  t^{-\alpha} \varphi(t)$  for all $t\geq 0$, with  $\ga \in(0,2)$.
Then  there is  some constant $C$ such that for all  $k\in \bbN$ and  for all non-increasing function $f  :  \bbR_+^k \to \bbR_+$  with bounded support,  we have
 \begin{equation}
  \int_{\bbR^k_+}  f(u_1,\cdots,u_k) \prod_{i=1}^k u_i^2 \mu(\dd u_i)
  \le C^k  \int_{\bbR^k_+} \ f(u_1,\cdots,u_k) \prod_{i=1}^k u_i^{1-\alpha} \varphi(u_i)\dd u.
 \end{equation}
 \end{proposition}

 \begin{proof}
 As for Proposition~\ref{lescroissants}, we only need to prove the result in the case 
 $k=1$, for a differentiable function $f$.
Let $T$ be such that the support of $f$ is included in $[0,T)$ and that of $f'$ is included in $(0,T)$.
We define $\tilde \varphi$ by
\begin{equation}
 \tilde \varphi(u)= (2-\alpha)u^{\alpha-2}\int^t_0 \varphi(u) u^{1-\alpha} \dd u.
\end{equation}
Let  also $\tilde \mu$ denote the measure on $[0,T]$ defined by $\tilde \mu(\dd t)= t^2 \mu(\dd t)$.  
By an integration by parts, we get, 
 using our assumption on $\mu$, that for all $t\geq 0$
\begin{align*}
 \tilde \mu([0,t))  = \int_0^t s^2 \mu(\dd s)
  & = - t^2 \mu((t,\infty)) + 2 \int_0^t s \mu((s,\infty)) \dd s\\
& \leq 2 \int_0^t s^{1-\ga} \gp(s) \dd s  = \frac{2}{2-\ga} t^{2-\ga} \tilde \varphi(t) \,.
\end{align*}

 
 Therefore, thanks to an integration by parts (using that $f(T)=0$), we get that
   \begin{align*}
   \int_{[0,T]} f(u) u^2 \mu(\dd u) & =  
-\int_{[0,T]}    f'(u) \tilde \mu([0,u])\dd u  \le   C\int_{[0,T]}  (-f'(u))  u^{2-\alpha} \tilde \varphi(u)\dd u \, ,
   \end{align*}
where we have used that $-f'(u) \geq 0$ so the inequality goes in the right direction.
We conclude the proof by another integration by parts.
 \end{proof}

Let us conclude this section with the proof of a useful (and standard) identity,
used repeatedly in the paper.

\begin{lemma}
\label{lemGamma}
For any $t >0$, $k\ge 0$ and $\zeta_1,\dots ,\zeta_{k+1} >0$, using the convention $s_0=0$ and $s_{k+1}=t$ we have,
for all $k\geq 1$,
\[
\int_{0<s_1< \cdots <s_k < t }  \prod_{i=1}^{k+1} (s_i-s_{i-1})^{\zeta_i-1} \dd s_i    = t^{\sum_{i=1}^{k+1}\zeta_i-1} \frac{\prod_{i=1}^{k+1}\Gamma(\zeta_i)}{\Gamma(\sum_{i=1}^{k+1}\zeta_i)} \,.
\]
\end{lemma}

\noindent
Let us stress that in this paper we use this identity with $\zeta_i=\zeta$ for all $i=1,\dots, k$ and either
$\zeta_{k+1}=\zeta$ or $\zeta_{k+1}=1$.

\begin{proof}
By scaling it is sufficient to prove the identity for $t=1$.
We have
\begin{equation}
\prod_{i=1}^{k+1}\Gamma(\zeta_i)
=\int_{(0,\infty)^{k+1}}  \prod_{i=1}^{k+1} u^{\zeta_i-1}_i e^{-\sum_{i=1}^{k+1} u_i}\dd u_i   \,.
\end{equation}
Using the change of variables $(u_1,\dots,u_{k+1})\to (s_1,\dots,s_k,v)$ where  $v:=\sum_{i=1}^{k+1} u_i$ and   $s_j:=\frac{\sum_{i=1}^j u_i}{v}$ for $j\in \lint 0,k\rint$, we obtain
\begin{equation}
 \prod_{i=1}^{k+1}\Gamma(\zeta_i) = \int_{(0,\infty)} v^{\sum_{i=1}^{k+1}\zeta_i-1}  e^{-v}  \dd v \int_{0<s_1< \cdots <s_k < 1 }   \prod_{i=1}^{k+1} (s_i-s_{i-1})^{\zeta_i-1}  \prod_{i=1}^k \dd s_i \, ,  
\end{equation}
which yields the result.
%
%
\end{proof}

\section{Tightness for $\xi^{\eta}_N$}\label{app:taixi}

First of all, let us recall the definition of the functional space $H_{\rm loc}^s(\bbR^{d+1})$.
Given $s\in \bbR$,
let $H^s(\bbR^{d+1})$
 be defined as the topological closure
of the space of smooth and compactly supported functions,
with respect to the norm
\[
\|f\|_{H^s} =\Big( \int_{\bbR^{d+1}} (1+|z|^2)^s |\hat f (z)|^2 \dd z  \Big)^{1/2} \, ,
\]
where $\hat f (z) = \int_{\bbR^{d+1}} f(x) e^{ -i x \cdot z} \dd x$ is the Fourier transform of $f$. The  associated local Sobolev space is given by  
\[
H_{\rm loc}^{s} (\bbR^{d+1})  := \big\{ f \colon f\psi \in H^{s}
\text{ for every compactly supported } \psi \in C^{\infty}  \big\}
\]
with the topology induced
by the family of semi-norms $(\|f \psi\|_{H^s} )_{\psi}$.

\begin{proof}[Proof of Lemma~\ref{lem:uniformxi}]
First of all, let us notice that
 we can write
\begin{equation}
\xi_{N,\eta} - \xi_{N,\eta}^{ (a)}  := \frac{1}{V_N}  \sum_{(n,x)\in \mathbb{H}_d} \bar \eta^{ (a)}_{n,x}\, \gd_{(\frac{n}{N}, \frac{x}{\sqrt{N/d}})}  \,,
\end{equation}
where
\begin{equation}
\bar \eta^{ (a)}_{n,x} := \left(\eta_{n,x}  - \bbE\big[\eta \, \big| \,  1+\eta <   a V_N \big] \right)\ind_{\{  1+ \eta_{n,x} <  a V_N \}}\, .
\end{equation}
Notice that $\bbE[ \bar \eta^{ (a)}_{n,x}] =0$. Now using~\eqref{moments},   the monotonicity in $a$ and continuity at $0$, there exists a function $\gep: (0,1)\to \bbR_+$ with $\lim_{a\to 0} \gep(a)=0,$ such that for every $N\ge 1$ 
$$V^{-2}_N\bbE[ (\bar \eta^{ (a)}_{n,x} )^2] \leq  \gep(a)   N^{ -(\frac{d}{2} +1)}. $$
Hence we have
\begin{equation}
\label{eq:approxxi}
\bbE \left[ \langle \xi_{N,\eta} - \xi_{N,\eta}^{ (a)}, \psi \rangle^2 \right] \le   \gep(a) N^{ - (\frac{d}{2} +1)} \sum_{(n,x)\in \bbH^d} \psi\left(\frac{n}{N}, \frac{x}{\sqrt{N/d}} \right)^2.
\end{equation}
Since the Riemann sum in the r.h.s.~converges,
we have
\[
\limsup_{N \to +\infty} \bbE \left[ \langle \xi_{N,\eta} - \xi_{N,\eta}^{(a)}, \psi \rangle^2 \right] \leq C_{\psi} \, a^{2-\ga} \,,
\]
which concludes the proof. 
%
\end{proof}

\begin{proof}[Proof of Lemma~\ref{letight}]
We have to show that for every smooth $\psi$ with compact support, the sequence   $ \xi^{\eta,\psi}_N:=\psi\times\xi^{\eta}_N$ is tight in $H^{s}(\bbR^{d+1})$.
This corresponds to showing that $\hat \xi^{\eta,\psi}_N$ is tight in  $L^2(\mu^s)$ for $\mu^s=(1+ |z|^2)^{-s}\dd z $.

\smallskip
We are going to show that with large probability $\hat \xi^{\eta,\psi}_N\in K_{R}$ where 
$K_R$ is defined (for a fixed $s'>s$)
\begin{multline}
 K_R:= \bigg\{ f  \ : \   \int |f(z)|^2 (1+ |z|^2)^{-s'} \dd z \le R  \\ \text{ and }   
  \forall a\in \bbR^{d+1}, \        \int |f(z+a)-f(z)|^2 (1+ |z|^2)^{-s} \dd z \le R  |a|    \bigg\}.
\end{multline}
Since $K_R$ is compact (by Frechet-Kolmogorov criterion) this is sufficient to conclude that the distribution of $\xi^{\eta,\psi}_N$ is tight.

\smallskip

To see that $\hat \xi^{\eta,\psi}_N\in  K_R$ with large probability, we first observe that 
 $\xi^{\eta,\psi}_N$ coincides with large probability with $\xi^{\eta,\psi,[0,b)}_N$ (constructed from the environment $\eta^{[0,b)}$, recall \eqref{doubletronc}). 
Then we have by a computation similar to \eqref{eq:approxxi}, for all $N$ sufficiently large
\begin{equation}\label{controlz}
\bbE\left[ \big|\hat \xi^{\eta,\psi,[0,b)}_N(z)\big|^2 \right]\le C_b \Big(\int |\psi|^2\Big)
\end{equation}
so that 
\begin{equation}
 \bbP\left[  \int  \left|\hat \xi^{\eta,\psi,[0,b)}_N(z)\right|^2 (1+ |z|^2)^{-s'} \dd z \ge R   \right]\le  \frac{1}{R} C_{b,\psi}\, .
\end{equation}
For the second point we observe that 
\begin{equation}
\bbE\left[ \left|\hat \xi^{\eta,\psi,[0,b)}_N(z+a)- \hat \xi^{\eta,\psi,[0,b)}_N(a)\right|^2 \right]\le C'_{b,\psi} |a|^2 \int |x|^2 |\psi (x)|^2 \dd x.
\end{equation}
(Note that $\hat \xi^{\eta,\psi,[0,b)}_N(z+a)- \hat \xi^{\eta,\psi,[0,b)}_N(a)$ is the Fourier transform  of the map
$ x\mapsto (e^{i a.x}-1) \psi\times  \xi^{\eta,\psi,[0,b)}_N$, so we are simply bounding the first factor by $|a||x|$.)
We therefore have that
\begin{equation}
 \bbP \left(  \int  \left|\hat \xi^{\eta,\psi,[0,b)}_N(z+a)- \hat \xi^{\eta,\psi,[0,b)}_N(a)\right|^2(1+ |z|^2)^{-s} \dd z \ge |a| \right)\le C_{b,\psi}'' |a|.
\end{equation}
Hence, using a union bound, we obtain that 
\[
\bbP\left(\exists k\ge k_0, \exists i\in \lint 1,d\rint\,\, \int  \left|\hat \xi^{\eta,\psi,[0,b)}_N(z+2^{-k}e_i)- \hat \xi^{\eta,\psi,[0,b)}_N(z)\right|^2(1+ |z|^2)^{-s} \dd z \ge 2^{-k} \right)\le \gep(k_0),
\]
with $\lim_{k_0\to \infty} \gep(k_0)=0$. This is sufficient to conclude that  $\hat \xi^{\eta,\psi,[0,b)}_N\in K_R$ with probability close to one, and thus so is $\hat \xi^{\eta,\psi}_N$.
\end{proof}

{\noindent \bf Acknowledgements:} We are grateful to Francesco Caravenna, Ronfeng Sun and Nikos Zygouras for enlightening discussions. 
We are also grateful for the referee's extremely detailed report,
which greatly helped us improve the presentation.
This work was realized  during H.L. extended stay in Aix-Marseille University funded by the European Union’s Horizon 2020 research and innovation programme under the Marie Skłodowska-Curie grant agreement No 837793. Q.B. acknowledges the support of ANR grant SWiWS (ANR-17-CE40-0032-0).

\bibliographystyle{plain}
\bibliography{biblio.bib}

\begin{thebibliography}{10}

\bibitem{AKQ10}
Tom Alberts, Konstantin Khanin, and Jeremy Quastel.
\newblock Intermediate disorder regime for directed polymers in dimension
  $1+1$.
\newblock {\em Phys. Rev. Letters}, 105(9):090603, 2010.

\bibitem{AKQ14b}
Tom Alberts, Konstantin Khanin, and Jeremy Quastel.
\newblock The continuum directed random polymer.
\newblock {\em J. Stat. Phys.}, 154:305--326, 2014.

\bibitem{AKQ14}
Tom Alberts, Konstantin Khanin, and Jeremy Quastel.
\newblock The intermediate disorder regime for directed polymers in dimension
  $1+1$.
\newblock {\em Ann. Probab.}, 42(3):1212--1256, 05 2014.

\bibitem{AY15}
Kenneth Alexander and G{\"o}khan Yıldırım.
\newblock Directed polymers in a random environment with a defect line.
\newblock {\em Electron. J. Probab.}, 20:20 pp., 2015.

\bibitem{AL11}
Antonio Auffinger and Oren Louidor.
\newblock Directed polymers in a random environment with heavy tails.
\newblock {\em Commun. Pure Appl. Math.}, 64(2):183--204, 2011.

\bibitem{BC20}
Erik Bates and Sourav Chatterjee.
\newblock The endpoint distribution of directed polymers.
\newblock {\em Ann. Probab.}, 48(2):817--871, 03 2020.

\bibitem{BL17}
Quentin Berger and Hubert Lacoin.
\newblock The high-temperature behavior for the directed polymer in dimension
  $1+2$.
\newblock {\em Ann. Inst. Henri Poincar{\'e}, Probab. Stat.}, 53(1):430--450,
  02 2017.

\bibitem{BL20_cont}
Quentin Berger and Hubert Lacoin.
\newblock The continuum directed polymer in {L}\'evy noise.
\newblock {\em arXiv:2007.06484v2}, 2020.

\bibitem{BT19}
Quentin Berger and Niccol\`o Torri.
\newblock Directed polymers in heavy-tail random environment.
\newblock {\em Ann. Probab.}, 47(6):4024--4076, 2019.

\bibitem{BC98}
Lorenzo Bertini and Nicoletta Cancrini.
\newblock The two-dimensional {S}tochastic {H}eat {E}quation: renormalizing a
  multiplicative noise.
\newblock 31(2):615--622, 1998.

\bibitem{BGT89}
Nicholas~H. Bingham, Charles~M. Goldie, and Jef~L. Teugels.
\newblock {\em Regular variation}, volume~27.
\newblock Cambridge university press, 1989.

\bibitem{Bol89}
Erwin Bolthausen.
\newblock A note on the diffusion of directed polymers in a random environment.
\newblock {\em Commun. Math. Phys.}, 123(4):529--534, 1989.

\bibitem{BS2020}
Adam Bowditch and Rongfeng Sun.
\newblock The two-dimensional continuum random field {I}sing model.
\newblock {\em arXiv:2008.12158}, 2020.

\bibitem{CSZ14}
Francesco Caravenna, Rongfeng Sun, and Nikos Zygouras.
\newblock The continuum disordered pinning model.
\newblock {\em Probab. Theory Relat. Fields}, 164:17--59, 2016.

\bibitem{CSZ13}
Francesco Caravenna, Rongfeng Sun, and Nikos Zygouras.
\newblock Polynomial chaos and scaling limits of disordered systems.
\newblock {\em J. {EMS}}, 19:1--65, 2017.

\bibitem{CSZ15}
Francesco Caravenna, Rongfeng Sun, and Nikos Zygouras.
\newblock Universality in marginally relevant disordered systems.
\newblock {\em Ann. Appl. Probab.}, 27(5):3050--3112, 2017.

\bibitem{CSZ18scaling}
Francesco Caravenna, Rongfeng Sun, and Nikos Zygouras.
\newblock On the moments of the $(2+1)$-dimensional directed polymer and
  stochastic heat equation in the critical window.
\newblock {\em Commun. Math. Phys.}, 372(2):385--440, 2019.

\bibitem{CSZ20}
Francesco Caravenna, Rongfeng Sun, and Nikos Zygouras.
\newblock The two-dimensional {KPZ} equation in the entire subcritical regime,
  2020.

\bibitem{CH02}
Philippe Carmona and Yueyun Hu.
\newblock On the partition function of a directed polymer in a gaussian random
  environment.
\newblock {\em Probab. Theory Relat. Fields}, 124(3):431--457, 2002.

\bibitem{clark2019}
Jeremy Clark.
\newblock Weak-disorder limit at criticality for directed polymers on
  hierarchical graphs.
\newblock {\em arXiv:1908.06555}, 2019.

\bibitem{C17}
Francis Comets.
\newblock {\em Directed Polymers in Random Environments}, volume 2175 of {\em
  \'Ecole d'Et{\'e} de probabilit{\'e}s de {S}aint-{F}lour}.
\newblock Springer International Publishing, 2016.

\bibitem{CSY03}
Francis Comets, Tokuzo Shiga, and Nobuo Yoshida.
\newblock Directed polymers in a random environment: strong disorder and path
  localization.
\newblock {\em Bernoulli}, 9(4):705--723, 2003.

\bibitem{CSY04}
Francis Comets, Tokuzo Shiga, and Nobuo Yoshida.
\newblock Probabilistic analysis of directed polymers in a random environment:
  a review.
\newblock In {\em Stochastic analysis on large scale interacting systems},
  volume~39 of {\em Adv. Stud. Pure Math.}, pages 115--142. Math. Soc. Japan,
  Tokyo, 2004.

\bibitem{CV06}
Francis Comets and Vincent Vargas.
\newblock Majorizing multiplicative cascades for directed polymers in random
  media.
\newblock {\em ALEA, Lat. Am. J. Probab. Math. Stat.}, 2:267--277, 2006.

\bibitem{CY}
Francis Comets and Nobuo Yoshida.
\newblock Directed polymers in a random environment are diffusive at weak
  disorder.
\newblock {\em Ann. Probab.}, 34(5):1746--1770, 2006.

\bibitem{DZ16}
Partha~S. Dey and Nikos Zygouras.
\newblock High temperature limits for $(1+ 1)$-dimensional directed polymer
  with heavy-tailed disorder.
\newblock {\em Ann. Probab.}, 44(6):4006--4048, 2016.

\bibitem{GB07}
Giambattista Giacomin.
\newblock {\em Random Polymer Models}.
\newblock Imperial College Press, World Scientific, 2007.

\bibitem{gu2019}
Yu~Gu, Jeremy Quastel, and Li-Cheng Tsai.
\newblock Moments of the {2D} {SHE} at criticality.
\newblock {\em arXiv:1905.11310}, 2019.

\bibitem{HH85}
David~A. Huse and Christopher~L. Henley.
\newblock Pinning and roughening of domain walls in {I}sing systems due to
  random impurities.
\newblock {\em Phys. Rev. Letters}, 54:2708--2711, 1985.

\bibitem{IS88}
John~Z. Imbrie and Thomas Spencer.
\newblock Diffusion of directed polymers in a random environment.
\newblock {\em J. Stat. Phys.}, 52:608--626, 1988.

\bibitem{Lac10pol}
Hubert Lacoin.
\newblock New bounds for the free energy of directed polymer in dimension $1+1$
  and $1+2$.
\newblock {\em Commun. Math. Phys.}, 294:471--503, 2010.

\bibitem{LS17}
Hubert Lacoin and Julien Sohier.
\newblock Disorder relevance without {H}arris criterion: the case of pinning
  model with {$\gamma$}-stable environment.
\newblock {\em Electron. J. Probab.}, 22:26, 2017.

\bibitem{LL10}
G.~F. Lawler and V.~Limic.
\newblock {\em Random Walk: A Modern Introduction}.
\newblock Cambridge Studies in Advanced Mathematics. Cambridge University
  Press, 2010.

\bibitem{Li1968}
Thomas~M Liggett.
\newblock An invariance principle for conditioned sums of independent random
  variables.
\newblock {\em Journal of Mathematics and Mechanics}, 18(6):559--570, 1968.

\bibitem{Nak16}
Makoto Nakashima.
\newblock Free energy of directed polymers in random environment in
  {$1+1$}-dimension at high temperature.
\newblock {\em Electron. J. Probab.}, 24:Paper No. 50, 43, 2019.

\bibitem{Soh09}
Julien Sohier.
\newblock Finite size scaling for homogeneous pinning models.
\newblock {\em ALEA, Lat. Am. J. Probab. Math. Stat.}, 6:163--177, 2009.

\bibitem{vargas07}
Vincent Vargas.
\newblock Strong localization and macroscopic atoms for directed polymers.
\newblock {\em Probab. Theory Relat. Fields}, 138(3-4):391--410, 2007.

\bibitem{Vi20}
Roberto Viveros.
\newblock Directed polymer for very heavy tailed random walks.
\newblock {\em arXiv:2003.14280}, 2020.

\bibitem{Vi19}
Roberto Viveros.
\newblock Directed polymer in $\gamma$-stable random environments.
\newblock {\em Ann. Inst. H. Poincaré Probab. Stat.}, (to appear).

\end{thebibliography}

\end{document}